\documentclass[11pt,leqno,twoside]{amsart}
\usepackage{fullpage}
\usepackage{amsmath, amsfonts, amsthm, amssymb, amscd, stmaryrd, url, amstext, amsxtra, amsopn}
\usepackage{mathrsfs}
\usepackage{longtable}
\usepackage{verbatim}
\usepackage{caption}
\usepackage[usenames, dvipsnames]{color}
\usepackage{enumerate}
\usepackage{enumitem}
\usepackage{color}
\usepackage{soul}

\usepackage{pdflscape}

\usepackage{hyperref}
\hypersetup{
  colorlinks   = true,
  urlcolor     = blue,
  linkcolor    = blue,
  citecolor   = purple
}
\usepackage[capitalize]{cleveref}

\usepackage{multirow}
\usepackage{todonotes}

\renewcommand{\Re}{{\mathfrak{Re}}}
\renewcommand{\Im}{{\mathfrak{Im}}}

\renewcommand{\b}{\beta}
\newcommand{\g}{\gamma}

\newcommand{\s}{\sigma}

\newtheorem{theorem}{Theorem}

\newtheorem{lemma}[theorem]{Lemma}
\newtheorem{conjecture}[theorem]{Conjecture}
\newtheorem{proposition}[theorem]{Proposition}
\theoremstyle{remark}

\reversemarginpar

\theoremstyle{definition}
\newtheorem{definition}[theorem]{Definition}

\title{Bounds for Mertens sums}
\thanks{ Samuel Broadbent was supported by NSERC USRA grants in Summers 2017, 2018, and 2020.
Andrew Fiori was supported by the NSERC Discovery Grant RGPIN-2020-05316. Habiba Kadiri was supported by the  NSERC Discovery Grant
RGPIN-2020-06731.  Nathan Ng was supported by the NSERC discovery grants  RGPIN-2020-06032,
RGPIN-2026-05876.
Kirsten Wilk was supported by an NSERC USRA grant in Summer 2017.
}

\author{Samuel Broadbent, Andrew Fiori,
Habiba Kadiri, Nathan Ng, Kirsten Wilk}

\address{University of Lethbridge \\ Department of Mathematics and Computer Science \\ 4401 University Drive \\ Lethbridge, AB \ T1K 3M4 \\ Canada}
\email{sambroadbent13@hotmail.com}
\email{andrew.fiori@uleth.ca}
\email{habiba.kadiri@uleth.ca}
\email{nathan.ng@uleth.ca}
\email{kirsten.wilk42@gmail.com}


\begin{document}

\begin{abstract}
In this article we provide new bounds for the Mertens sums and products including  $\sum_{p \le x} p^{-1}$ and  $\prod_{p \le x} (1-\frac{1}{p})$
which provide superior exponential and log type bounds for these sums in all ranges.
These weighted prime number sums and products were extensively studied by Rosser and Schoenfeld (1962) and are employed in a wide range of applications in number theory, cryptography, and combinatorics.  Extensive tables are provided in this article which will be useful for these types of applications.
The main new ideas in this article are sharp bounds for weighted sums of zeros of zeros of the zeta function.
We make use of a novel technique of Fiori-Kadiri-Swidinsky (2023) which relies on a recent explicit zero-density estimate for $N(\s,T)$ of Kadiri-Lumley-Ng (2018).
The bounds and techniques in this article for weighted zeros sums will likely be useful in many other arithmetic applications.   Our main theorem significantly
improves the exponential decay result of Vanlalngaia (2017) and fills a gap in the literature by correcting work of Dusart (2018).
The results are also presented in a way that are amenable to future improvements.
 In addition, we prove an exact ``Riemann-Guinand explicit formula" for the Mertens sum
$\sum_{p \le x} p^{-1}$ that appears to be new.
\end{abstract}

\maketitle

\newcommand{\andrew}[1]{{\color{blue} \sf  Andrew: #1}}
\newcommand{\nathan}[1]{{\color{purple} \sf  Nathan: #1}}
\newcommand{\habiba}[1]{{\color{red} \sf  Habiba: #1}}

\section{Introduction}

 In 1874, Mertens \cite{Mertens1} showed that
 \begin{align}
  \label{Mertens1}
   & \lambda(x)= \sum_{p \le x} \frac{1}{p}   = \log \log x + M +  o(1), \\
     \label{Mertens2}
   & \Upsilon(x) = \sum_{p \le x} \frac{\log p}{p}   = \log x - M' + o(1),
   \end{align}
 where $o(1)$ is a function which approaches $0$ as $x \to \infty$.
   When one includes prime powers in \eqref{Mertens2}, we obtain
   \begin{align}
       \label{Mertens3}
     &  \widetilde{\psi}(x)= \sum_{n \le x} \frac{\Lambda(n)}{n}   = \log x - \gamma + o(1)
    \end{align}
 where $\gamma$ is the Euler-Mascheroni constant (OEIS\footnote{OEIS is the online encyclopedia of integers available at {\tt https://oeis.org}.
 Further accuracy for this constant is available on this website.}: A001620):
  \begin{equation}
   \label{EulerGamma}
   \gamma = 0.577215664901532 \ldots.
 \end{equation}
 The constants $M$ and $M'$  are defined from the absolutely convergent series
 \begin{equation}
   \label{M}
    M = \gamma + \sum_{p} \Big( \log \Big(1 - \frac{1}{p} \Big) + \frac{1}{p} \Big) = 0.26149 72128 4764 \ldots
\end{equation}
(the Meissel-Mertens constant, OEIS: A077761) and
\begin{equation}\label{Mp}
    M'  = \gamma + \sum_{n=2}^{\infty} \sum_{p} \frac{\log p}{p^n} = 1.33258 22757 3322 \ldots,
 \end{equation}
 (OEIS: A083343).
  The constant $M$ was first determined in  \cite[Equation (17)]{Mertens1}
 and $M'$ was first determined by Landau (\cite[Section  \S 55]{Landau}, see also \cite[Section 5.2]{Axler2018}).
 The function names $\lambda(x)$ and $\Upsilon(x)$ was previously used  in \cite{Van} and
  $\widetilde{\psi}(x)$ was used in
  \cite{Ram}.
  It follows from \eqref{Mertens1} that
  \begin{equation}
   \label{Mertens4}
     \prod_{p \le x} \Big( 1- \frac{1}{p} \Big) \sim \frac{e^{-\gamma}}{\log x}
     \text{ and }
     \prod_{p \le x} \frac{p}{p-1} \sim e^{\gamma} \log x.
 \end{equation}
 The functions in \eqref{Mertens1}, \eqref{Mertens2}, and \eqref{Mertens3} are intimately related to the more familiar prime counting sums
 \begin{equation}
   \label{vartheta}
   \vartheta(x) = \sum_{p \le x} \log p
   \text{ and } \psi(x) = \sum_{n \le x} \Lambda(n).
 \end{equation}
 The prime number theorem
  is equivalent to each of the asymptotics
$\vartheta(x) \sim x $ and  $\psi(x) \sim x$.
The formulae in \eqref{Mertens1}, \eqref{Mertens2}, and \eqref{Mertens3} can be established independently of the prime number theorem
 and actually predate it. However, the best explicit versions of such asymptotic formulae follow from the prime number theorem.
 In a highly cited article, Rosser and Schoenfeld \cite{RS62} derive from estimates for the prime counting functions $\psi(x)$ and $\theta(x)$
 various explicit versions
 of \eqref{Mertens1}, \eqref{Mertens2},  and \eqref{Mertens4}.
In the past decade, since the publication of \cite{FaKa}, there has been a flurry of activity on bounds for $\psi(x)$ with
numerous improvements on these bounds.  Other articles that provided bounds include
\cite{But16}, \cite{FKS}, \cite{PlattTru}, \cite{JY}, \cite{CH}.  Since the beginning of 2026 there has been an effort to formally verify
many of these results with Lean. The Integrated Explicit Analytic Number Theory Network \cite{IEANTN}, a project led by Terence Tao
has begun to formalize some of these papers including \cite{BKLNW}, \cite{FKS}, \cite{FKS2}, \cite{CH}.

In this article, we make use of  the following bounds for $\vartheta(x)$ and $\psi(x)$.
\begin{theorem}\label{thetathm}
\begin{itemize}
\item[(i)] Bounds of the exponential-form:\\
Let $x_0 \ge 2$.
Then there exist positive constants $A_{\vartheta}(x_0)$ and $A_{\psi}(x_0)$ such that, for all $x\ge x_0$,
 \begin{equation}\label{ImplicitExpDecayApsi}
  \frac{| \psi(x) - x|}{x}  \le
  A_{\psi}(x_0) (\log x)^{B}
 \exp(-C \sqrt{\log x}) ,
   \end{equation}
  and
 \begin{equation}\label{ImplicitExpDecayAtheta}
  \frac{| \vartheta(x) - x|}{x}  \le
  A_{\vartheta}(x_0) (\log x)^{B}
 \exp(-C \sqrt{\log x}) ,
   \end{equation}
   where
  \begin{equation}
    \label{BC}
    B=3/2 , \ C = 2/\sqrt{R},
  \end{equation}
  and $R$ is the value in the classical zero-free region.
Values for $A_{\psi}(x_0)$ and $A_{\vartheta}(x_0)$ can respectively be found in \cite{FKS,FKS2}.
In particular, by \cite[Corollary 1.4]{FKS} and \cite[Corollary 14]{FKS2}, \cite[Theorem 1.3]{MTY}, we have
  \begin{equation}\label{ImplicitExpDecaypsitheta}
  A_{\psi}(2) = 9.22022 , \ A_{\vartheta}(2) = 9.22022, \ R = 5.5666305
  \footnote{In \cite{FKS} the value $R = 5.5666305$ was used. However, by the time \cite{MTY} was published $R$ was improved to $R=5.558691$. }
  , \ C = 0.84768 \ldots .
   \end{equation}
\item[(ii)] Bounds of the log-form:\\
Let $\ell \ge 0$ and $x_0 \ge 2$.
Then there exist positive constants $\widetilde{\eta}_{\ell} =\widetilde{\eta}_{\ell}(x_0)$ and $\eta_{\ell} = \eta_{\ell}(x_0)$ such that, for all $x\ge x_0$,
\begin{equation}
 \label{Epsithetalogbd}
   \frac{| \psi(x) - x|}{x} \le \frac{\widetilde{\eta}_{\ell}(x_0)}{(\log  x)^{\ell}} \ \text{and }\
   \frac{| \vartheta(x) - x|}{x} \le \frac{\eta_{\ell}(x_0)}{(\log  x)^{\ell}}.
\end{equation}
For $\ell\in\{0, 1,2,3,4,5\}$ and fixed values of $x_0$, values for $\widetilde{\eta}_{\ell}$ and $\eta_{\ell}$ can be found in Tables \ref{table:eta} and \ref{table:etatilde}
\footnote{Worse values may be found in \cite{BKLNW}. The values in Tables \ref{table:eta} and \ref{table:etatilde} were computed using the method of \cite{BKLNW} with updated inputs from \cite{FKS} and \cite{FKS2}}.
\end{itemize}
\end{theorem}
For the rest of the paper, we will refer to bounds of the shape \eqref{ImplicitExpDecayApsi}, \eqref{ImplicitExpDecayAtheta} as exponential-form and bounds of shape   \eqref{Epsithetalogbd} as log-form.
Extensive tables with values for $\eta_{\ell}(x_0)$ were computed in \cite{BKLNW}.
These values were improved in \cite{FKS}.  Using ideas from \cite{FKS} we have computed many values of
$\eta_{\ell}(x_0)$ and  $\widetilde{\eta}_{\ell}(x_0)$.  Note these values are based on computations that use B\"{u}the's method \cite{But16}
as corrected in the thesis of Sreerupa Bhattacharjee
\cite{Bhattacharjee2024} for small values of $x$ and the method of Fiori et al. \cite{FKS} for large values
of $x$.
Values for $A_{\psi}(x_0)$ and $A_{\vartheta}(x_0)$ are given in Table \ref{ApsiUpsilonwpsi} of Appendix \ref{Section:Tables}.

\noindent {\bf Remark}.
Note that many of the above results can currently be improved.   The recent work of Belotti, Trudgian, and Yang \cite{BTY} provides
an improved zero-free region constant $R=4.896$
and  Chirre-Helfgott \cite{CH} improves bounds for $\psi(x)$ with $x \le e^{2394}$.  Using either of these results in the previous articles
\cite{BKLNW}, \cite{FKS}, \cite{FKS2} will improve bounds for both $\psi(x)$ and $\vartheta(x)$. However, it would be a significant undertaking
to rework these previous articles and would require
extensive recomputations.  Current results in explicit number theory often depend on each other in complicated ways.  One hope for the work of the IEANTN \cite{IEANTN} is to ``automate" these types of improvements  and allow
for a simpler method to facilitate improvements to the best bounds for $\psi(x)$ and for related problems in prime number theory.  For instance, any time any of the inputs such at
the zero-free region constant $R$, the verification height $H$ of the Riemann Hypothesis, or even
a zero-density estamate is improved,  could their be a more  instantaneous improvement of a variety of results in explicit number theory.  This requires writing prior lemmas in very general forms in which there are minimal assumptions being made on $R$, $H$, and the shape of the zero-density estimates.

\subsection{Main Theorems}
We shall provide explicit error terms for the $o(1)$ terms in \eqref{Mertens1}, \eqref{Mertens2},
and \eqref{Mertens3} of the same shape as the  error terms in Theorem \ref{thetathm}.

\begin{theorem}[Bounds for $\lambda(x)$]\noindent
\label{1overpthm}
\begin{itemize}
\item[(i)] Let $x_0 \ge 2$. There exists a positive constant  $A_{\lambda}(x_0)$ such that
 \begin{equation}\label{ImplicitExpDecay-lambda}
  |\lambda(x) - \log \log x - M | \le
  A_{\lambda}(x_0) (\log x)^{1/2}
 \exp(- C \sqrt{\log x})\text{ for all }\  x \ge x_0.
 \end{equation}
 In particular, we establish
$A_{\lambda}(2) =   9.2203$
 and more generally,  $A_{\lambda}(x_0) $ is defined in \eqref{Alambda}:
$$  A_{\lambda}(x_0)= A_{\vartheta}(x_0) + A''(x_0)  \exp((-C(\sqrt{2}-1)) \sqrt{\log x_0})
 $$
 where $A_{\vartheta}(x_0), C,$ and $A''(x_0)$ are defined in \eqref{ImplicitExpDecayAtheta}, \eqref{BC}, and \eqref{App} respectively.
For fixed $x_0$, values for $A_{\vartheta}(x_0)$ can be found in Table \ref{ApsiUpsilonwpsi}.
See Table \ref{Apptable} for values of $A''(x_0)$.
\item[(ii)]  Let $\ell \ge 1$ and $x_0 \ge 2$.
There exists a positive constant  $\mathcal{A}_{\ell}(x_0)$ such that, for all $x \ge x_0$,
\begin{equation}\label{log-bnd-lambda}
   | \lambda(x) -  \log \log x -M | \le \frac{\mathcal{A}_{\ell}(x_0)}{(\log  x)^{\ell}}.
\end{equation}
A formula for $\mathcal{A}_{\ell}(x_0)$ is given in \eqref{Aell}:
\begin{equation}
  \label{Aellformula}
\mathcal{A}_{\ell}(x_0) = \eta_{\ell-1}(x_0)+\delta_{\ell}(x_0)+\kappa_{\ell}(x_0)
\end{equation}
where  $\eta_{\ell-1}(x_0)$, $\delta_{\ell}(x_0)$, and $\kappa_{\ell}(x_0)$ are given in
\eqref{Epsithetalogbd},
\eqref{deltalx0}, and \eqref{Ckx} respectively.
For $\ell \in\{1,2,3,4,5\}$ and some fixed $x_0$, values for $\mathcal{A}_{\ell}(x_0)$ are calculated in Tables \ref{Altable} and \ref{MkmktableThm2}.

\end{itemize}
\end{theorem}

We provide analogous results for the sums in \eqref{Mertens2} and \eqref{Mertens3}
in Theorems \ref{bnd:sum_logp} and \ref{bnd:sum_lambdan} which follow. Their proofs are very similar to the proof of Theorem \ref{1overpthm} and even slightly simpler.
\begin{theorem}[Bounds for $\Upsilon(x)$]\noindent
\label{bnd:sum_logp}
\begin{itemize}
 \item[(i)] There exists a positive constant  $A_{\Upsilon}(x_0)$ such that
 \begin{equation}\label{ImplicitExpDecay-Upsilon}
  |\Upsilon(x) - \log x +M' | \le
  A_{\Upsilon}(x_0) (\log x)^{3/2}
 \exp(-0.8746 \sqrt{\log x})\text{ for all }\  x \ge x_0.
   \end{equation}
 In particular, we establish $
    A_{\Upsilon}(2)  = 9.2203$
 and more generally,  $A_{\Upsilon}(x_0)$ is defined in \eqref{Aupsilon}:
   $$A_{\Upsilon}(x_0)   = A_{\vartheta}(x_0) +   D'(x_0)  \exp((-C(\sqrt{2}-1)) \sqrt{\log x_0})) ,$$
   where $A_{\vartheta}(x_0), C, $ and $D'(x_0)$ are defined in \eqref{ImplicitExpDecayAtheta}, \eqref{BC}, and \eqref{Dp} respectively.
See Table \ref{Dptable} for values of $D'(x_0)$.
\item[(ii)] Let $\ell \ge 1$ and $x_0 \ge 2$.
 There exists a positive constant $\mathcal{B}_{\ell}(x_0)$ such that, for all $x\ge x_0$,
 \begin{equation}\label{log-bnd-Upsilon}
  | \Upsilon(x) - \log x + M' | \le \frac{\mathcal{B}_{\ell}(x_0)}{(\log  x)^{\ell}} .
 \end{equation}
A formula for $\mathcal{B}_{\ell}(x_0)$ is given in \eqref{Bell}:
$$
\mathcal{B}_{\ell}(x_0) =   \eta_{\ell}(x_0) + \widetilde{\delta}_{1,\ell}(x_0) + \sum_{j \in J} \alpha_j
 \max_{x \ge x_0}( g_{j,\ell,0}(x)),
$$
where $\eta_{\ell}, \widetilde{\delta}_{1,\ell}$,
are defined in \eqref{Epsithetalogbd}, \eqref{tildedeltamlx0} respectively
and $\alpha_j \, (j \in J)$ are a finite set of constants given in \eqref{Jalphaj},
and
 \begin{align}
 \label{gabcx}
 g_{a,b,c}(x) & = x^{-a}(\log x)^b \exp(c\sqrt{\log x}).
 \end{align}
For $\ell \in\{1,2,3,4,5\}$ values of $\mathcal{B}_{\ell}(x_0)$ are given in Tables \ref{Bltable} and \ref{MkmktableThm3}.
\end{itemize}
\end{theorem}

From Theorem \ref{bnd:sum_logp}, we derive explicit bounds for $\widetilde{\psi}(x) = \sum_{n \le x} \Lambda(n)/n$.
\begin{theorem}[Bounds for $\widetilde{\psi}(x)$]\noindent
\label{bnd:sum_lambdan}
\begin{itemize}
\item[(i)] There exists a positive constant $A_{\widetilde{\psi}}=A_{\widetilde{\psi}}(x_0)$
such that
 \begin{equation}\label{ImplicitExpDecay-tildepsi}
  | \widetilde{\psi}(x) - \log x - \gamma | \le
  A_{\widetilde{\psi}}(x_0)(\log x)^{3/2}
 \exp(-0.8746 \sqrt{\log x}) \text{ for all }\  x \ge x_0.
   \end{equation}
In particular, we establish $A_{\widetilde{\psi}}(2)  = 9.2203$
 and more generally,
$A_{\widetilde{\psi}}(x_0)$ is defined in \eqref{Awidetildepsi}:
$$  A_{\widetilde{\psi}}(x_0)  = A_{\psi}(x_0) +   D''(x_0)  \exp((-C(\sqrt{2}- 1)) \sqrt{\log x_0}))
$$
where $A_{\psi}(x_0), C$, and $ D''(x_0) $ are defined in \eqref{ImplicitExpDecayApsi}, \eqref{BC}, and \eqref{sum_logp:ImplicitExpDecay}.
See Table \ref{Dpptable} for values of $D''(x_0)$.
\item[(ii)]  Let $\ell \ge 1$ and $x_0 \ge 2$.  There exists a positive constant  $\mathcal{C}_{\ell}(x_0)$ such that, for all $x\ge x_0$,
 \begin{equation} \label{log-bnd-tildepsi}
  |  \widetilde{\psi}(x) - \log x + \gamma | \le \frac{\mathcal{C}_{\ell}(x_0)}{(\log  x)^{\ell}} .
 \end{equation}
A formula for $\mathcal{C}_{\ell}(x_0)$ is given in \eqref{Cell}:
$$\mathcal{C}_{\ell}(x_0) = \widetilde{\eta_{\ell}}(x_0) + \widetilde{\delta}_{1,\ell}(x_0) +
    1.844 \max_{x \ge x_0}(g_{1,\ell,0}(x)), $$
where $\widetilde{\eta}_{\ell}(x_0), \widetilde{\delta}_{1,\ell}(x_0), $ and $g_{1,\ell,0}$
are defined in \eqref{Epsithetalogbd}, \eqref{tildedeltamlx0}, and \eqref{gabcx} respectively.
For $\ell \in\{1,2,3,4,5\}$ some values of $\mathcal{C}_{\ell}(x_0)$ are given in Table \ref{Cltable}.
\end{itemize}
\end{theorem}
\noindent {\bf Remarks on Theorems \ref{1overpthm}, \ref{bnd:sum_logp}, and  \cref{bnd:sum_lambdan}}
\begin{enumerate}
\item[1.]
In 2017, Vanlanlgaia derived from Dusart's bound \cite{Dusart2016} for $|\theta(x)-x|$ in exponential-form, a bound for the first Mertens sum of the same shape:
\begin{equation}
   \label{Vexp1}
  |\lambda(x) - \log \log x - M  | \le 1.1(\log x)^{-3/4} \exp (-0.4183 \sqrt{\log x} )
  \ \text{ for all }\  x \ge e^{4635},
\end{equation}
where $0.4183 \ldots =\sqrt{0.175}$.  This bound is to be compared to our Theorem \ref{1overpthm}
\begin{equation}
  \label{FKNexp}
 9.2203 (\log x)^{1/2} \exp (-0.8746\sqrt{\log x} )\ \text{ for all }\  x \ge 2.
\end{equation}
We improve the constant in the exponential decay term by a factor greater than 2, and
our error term is smaller by a factor of $6.5 \cdot 10^{-8}$ at $x=e^{4635}$.
Finally, we mention that Vanlanlgaia \cite[Theorem 9, Corollary 10]{Van} established the following error bound with exponential-type decay
\begin{equation}
   \label{Vexp2}
  | \widetilde{\psi}(x) -\log x - \gamma| \le 1.042 (\log x )^{1/4}
  \exp ( - 0.4238 \sqrt{ \log x} )
  \text{ for all } x \ge e^{2266}.
\end{equation}
Our bound in Theorem \ref{bnd:sum_lambdan} part (i) is
 smaller than this   by a factor of $2.4 \cdot 10^{-4}$ at $x = e^{2266}$.
Vanlanlgaia's proof of \eqref{Vexp1} and \eqref{Vexp2} made use of arguments  from \cite{Ram} where $\widetilde{\psi}(x)$
was studied.  Our improvements give an exponential error term of the type $O^{*}(x(\log x)^B \exp(- 2/\sqrt{R} \sqrt{\log x})$ whereas \cite{Ram} and \cite{Van} have an error term of the type
$O^{*}(x (\log x)^{B'} \exp(- 1/\sqrt{R} \sqrt{\log x})$
where $R$ is the zero-free region constant as in \eqref{ZFR} below.
The improvement in the constant within the exponential arises from an argument of Pintz \cite{Pintz} which uses zero-density estimates
and we apply the recent version of this argument due to Fiori et al.  \cite{FKS}.
\item[2.] We have written our proofs in the most general way so that the arguments in this article are more adaptable to any future improvements.  Futher, we have provided a more extensive list of values $x_0$ and of bounds for the functions $\lambda(x), \Upsilon(x)$, and $\widetilde{\psi}(x)$ and have provided general formulae for the constants  $A_{\lambda}(x_0)$,$A_{\Upsilon}(x_0)$, $A_{\widetilde{\psi}}(x_0)$, $\mathcal{A}_{\ell}(x_0)$, $\mathcal{B}_{\ell}(x_0)$, and $\mathcal{C}_{\ell}(x_0)$. Past articles tended to focus just on particular choices of $x_0$ that were tailored for specific applications.
\item[3.] Our results for the Mertens sums and products fills a gap in the literature.   In  Dusart's article \cite{Dusart2018} there are a number of results on Mertens sums and products.  It was previously pointed out in \cite[footnote, p. 2299]{BKLNW} that there was a mistake in \cite{Dusart2018} which affects various results on $\psi(x)$
including Theorems 3.2,3.3  and Table 1 of that article.
Consequently, the proof of all corollaries on Mertens sums and products given in \cite{Dusart2018} are not valid.  Despite this many authors have continued to use the results in \cite{Dusart2018}.
In this article we have recovered all of Dusart's  stated results  on Mertens sums and products.
\item[4.] The proof of Theorem \ref{1overpthm} can be found in  Section \ref{section:proofThm2}
and the proofs of Theorems  \ref{bnd:sum_logp}, and  \cref{bnd:sum_lambdan}  are in Section \ref{section:othersums}.
\item[5.]  Each of the  constants $A_{\lambda}(x_0)$, $A_{\Upsilon}(x_0)$, and $A_{\widetilde{\psi}}(x_0)$   are numerically very close to $A_{\vartheta}(x_0)$ for $x_0$ sufficiently large.  We find that
  $ | A_{\lambda}(x_0)- A_{\vartheta}(x_0) |
    \leq  10^{-11}$ for $x_0\ge  e^{200}$,  $ | A_{\Upsilon}(x_0)- A_{\vartheta}(x_0) |
    \leq  10^{-4}$ for  $x_0\ge  e^{70}$, and $ | A_{\widetilde{\psi}}(x_0)- A_{\vartheta}(x_0) |
    \leq  10^{-4}$ for $x_0\ge  e^{70}$.
  As a natural consequence, displayed values in the tables often round up to the same values.
 \item[6.] The constants $\mathcal{A}_{\ell}(x_0)$ are numerically very close to $\eta_{\ell-1}(x_0)$ for $x_0$ sufficiently large.  For example we find that  $|\mathcal{A}_{\ell}(x_0) - \eta_{\ell-1}(x_0)| \le 10^{-4}$ for $x_0 \ge e^{100}$ and $\ell \in \{1,2,3,4\}$.
  \item[7.]  The constants $\mathcal{B}_{\ell}(x_0)$ and  $\mathcal{C}_{\ell}(x_0)$ are numerically very close to $\eta_{\ell}(x_0)$ and $\widetilde{\eta_{\ell}}(x_0)$ respectively, for $x_0$ sufficiently large.  We find that  $|\mathcal{B}_{\ell}(x_0) - \eta_{\ell}(x_0)| \le 10^{-4}$ for $x_0 \ge e^{100}$ and $|\mathcal{C}_{\ell}(x_0) - \widetilde{\eta_{\ell}}(x_0)| \le 10^{-4}$ for $x_0 \ge e^{100}$.

 \item[8.]  Recently, new bounds for $\psi(x)$ and $\widetilde{\psi}(x)$ were derived by Chirre and Helfgott \cite{CH}.  In their work they
developed a smoothed Perron type formula and then studied a certain optimization problem related to the theory of Beurling-Selberg
entire majorants which occur in extremal problems in harmonic analysis.  Their work led to the bounds
 \begin{equation}
   \label{CH1}
 |\psi(x)-x|\leq \frac{\pi}{3\cdot 10^{12}}\cdot x + 113.67 \sqrt{x}
  \text{ for all } x \ge 1,
 \end{equation}
 \begin{equation}
  \label{CH2}
     | \widetilde{\psi}(x) -\log x - \gamma|
   \le \left(\frac{\pi}{3\cdot 10^{12}} + \frac{113.67}{\sqrt{x}}\right)
   \text{ for all } x \ge 1.
 \end{equation}
 Here the constant $3\cdot 10^{12}$ is directly related to the partial verification of the Riemann Hypothesis \cite{PlaTruH2020}.
 Note that our bound in \eqref{ImplicitExpDecay-tildepsi} is stronger than \eqref{CH2} for $x \ge e^{2395}$.

 \end{enumerate}

Bounds for the Mertens products can be deduced from bounds for Mertens sums.
\begin{theorem} \label{prod1}
Let $\ell \ge 0$  and $x_0 \ge 2$.
There exist positive constants $\mathcal{A}_{\ell}(x_0), \mathcal{A}'_{\ell}(x_0)$  such that
 \begin{equation}
  \label{bnd:prod2}
 \frac{e^{-\gamma}}{\log x} \Big( 1 - \frac{\mathcal{A}_{\ell}(x_0)}{(\log  x)^{\ell}} \Big) \le \prod_{p \le x} \Big( 1 - \frac{1}{p} \Big) \le \frac{e^{-\gamma}}{\log x}
 \Big( 1 + \frac{\mathcal{A}'_{\ell}(x_0)}{(\log  x)^{\ell}} \Big)
\text{ for all }\  x \ge x_0.
 \end{equation}
Note $\mathcal{A}_{\ell}(x_0)$ is the same value as in Theorem \ref{1overpthm} (i)
and a formula for $\mathcal{A}'_{\ell}(x_0)$ is given in
\eqref{defn:Atildem} below. For $\ell \in\{1,2,3,4,5\}$ some values of $\mathcal{A}'_{\ell}(x_0)$ are given in Table \ref{Aellptable}.
\end{theorem}

\begin{theorem} \label{prod2}
Let  $\ell \ge 1$ and $x_0 \ge 2$.
There exist positive constants $\mathcal{D}_{\ell}(x_0), \mathcal{D}'_{\ell}(x_0)$  such that
 \begin{equation}
  \label{bnd:prod4}
  e^\gamma \log x \Big( 1 - \frac{\mathcal{D}_{\ell}(x_0)}{(\log x)^{\ell}} \Big) \le \prod_{p \le x} \frac{p}{p-1}\le e^\gamma \log x \Big( 1 + \frac{\mathcal{D}'_{\ell}(x_0)}{(\log x)^{\ell}} \Big)
\text{ for all }\  x \ge x_0.
 \end{equation}
 Formulae for $\mathcal{D}_{\ell}(x_0), \mathcal{D}'_{\ell}(x_0)$ are given in \eqref{defn:Dellx0} and \eqref{defn:Dellprimex0} below.
 For $\ell \in\{1,2,3,4,5\}$ some values of $\mathcal{D}_{\ell}(x_0), \mathcal{D}'_{\ell}(x_0)$  are given in Tables \ref{Delltable} and \ref{Dellptable}.
\end{theorem}

The proofs of Theorems
\ref{prod1} and \ref{prod2} are based on the arguments from \cite[p. 87]{RS62} and \cite[Theorem 5.9]{Dusart2018} and may be found in Section \ref{products}. The basic idea is to take logarithms of each product, use the Taylor series of $\log$, and apply \cref{1overpthm}.

The two main ideas in the proof of Theorems \ref{1overpthm},
\ref{bnd:sum_logp}, and \ref{bnd:sum_lambdan} are a Riemann-Guinand style {\it explicit formula} for each of the functions
 $\lambda(x)$, $\Upsilon(x)$, and  $\widetilde{\psi}(x)$ along with sharp bounds for the zero sums
 \begin{equation}
   \label{zerosums}
    J_m(x) = \sum_{\rho} \frac{x^{\beta-1}}{|\gamma|^{m+1}} \text{ for } m \ge 1.
 \end{equation}
 Recall that the classical explicit formula for $\psi(x)$ is
  \begin{equation}
   \label{classical}
    \psi(x)
= x
- \sum_{\rho} \frac{x^{\rho}}{\rho}
- \log(2\pi)
- \frac{1}{2}\log\!\bigl(1 - x^{-2}\bigr),
  \end{equation}
  for $x$ not a prime power.
 The   explicit formula for $\lambda(x)$   is
 \begin{equation}
 \begin{split}
   \label{Mertensexplicit}
     \lambda(x) & =    \log \log x +M -\frac{1}{\log x} \sum_{\rho} \frac{x^{\rho-1}}{\rho}
   + \frac{\log x +1}{(\log  x)^2}   \sum_{\rho} \frac{x^{\rho-1}}{\rho(\rho-1)}
- \sum_{\rho} \frac{1}{\rho} \Gamma\!\bigl(-1,\,- (\rho-1)\log x\bigr) \\
& +  \sum_{\rho} \frac{\rho-1}{\rho}  \Gamma\!\bigl(-2,\,- (\rho-1)\log x\bigr) \\
& + \int_x^{\infty} \frac{1 + \log t}{t^2 \log ^2 t} (\psi(t)-\vartheta(t))dt
     +  \frac{\vartheta(x) - \psi(x)}{x \log x}
     + \int_x^\infty
\frac{dt}{t^2(t^2-1)\log t}
 \end{split}
 \end{equation}
 for $x$ not a prime power,
 where for $\Re(s) >0$
 \begin{equation}
   \label{incgamma}
 \Gamma(s,z) =  \int_{z}^{\infty} t^{\,s-1} e^{-t}\,dt
 \end{equation}
  is the incomplete gamma function \cite[Section 8.2]{NIST}.  Here $t^{s-1}$ is defined by the principal branch of the logarithm and the path from $z$ to $\infty$ avoids the negative real axis.  Note that the incomplete gamma function has a meromorphic continuation to the negative integers.
   The authors are uncertain if the formula \eqref{Mertensexplicit} has appeared in the literature before.
 Very roughly, the above explicit formula for $\lambda(x)$ can be interpreted as
 \begin{equation}
   \label{explicitapprox}
   \lambda(x) \sim  \log \log x + M + \frac{\vartheta(x)-x}{x \log x} +
    \frac{J_1(x)}{ \log  x}    +  \frac{J_2(x)}{\log ^2 x}
    \text{ as } x \to \infty
 \end{equation}
 where $J_m(x)$ are defined in \eqref{zerosums}.
 Note that $ \frac{\vartheta(x)-x}{x \log x}$ replaces the sum $\sum_{\rho} \frac{x^{\rho-1}}{\rho}$ by the classical explicit formula \eqref{classical}.   Further, $J_1(x)$ is a bound for the second sum in \eqref{Mertensexplicit} and $J_2(x)$ is a bound for the third and fourth sums in \eqref{Mertensexplicit}.
An explicit inequality version of \eqref{explicitapprox} was first used by Vanlalngaia \cite{Van} for $\lambda(x)$ following Ramar\'{e} who had
previously derived an analogous explicit formula for $\widetilde{\psi}(x)$.
Our version of this explicit inequality is given in  \cref{explicitformulaMertens} (ii) below.
The explicit formula \eqref{Mertensexplicit} follows from the partial sum identity \cite[equations (4.13), (4.15), pp. 73-74]{RS62}
\begin{equation}
  \label{ps}
\lambda(x) =   \log \log x +M+ \frac{\vartheta(x) - x}{x \log x}  - \int_x^{\infty} \frac{1 + \log t}{t^2 \log ^2 t} (\vartheta(t)-t)dt.
\end{equation}
Earlier work of Rosser-Schoenfeld \cite{RS62} and Dusart \cite{Dusart2018}) used this partial summation formula
\eqref{ps} to obtain bounds for $\lambda(x)$ by applying the best known bounds $\vartheta(x)-x$ (at the time of their work) in \eqref{ps}. However, this argument leads to  inferior estimates.

 The main improvements to our bounds for $\lambda(x)$ arise from using the best bounds for $\vartheta(x)-x$ along with superior estimates for the zero sums $J_m(x)$ in \eqref{explicitapprox}.   The best known published bounds for $\vartheta(x)$  may be found in Broadbent et al. \cite{BKLNW} and  Fiori et al. \cite{FKS}
 \footnote{ Incorporating the work of \cite{CH} into the articles \cite{BKLNW} and \cite{FKS} would lead to another round of improvements for bounds for $\vartheta(x)-x$.}.
 The bounds for $\vartheta(x)$ are largely deduced from  the best explicit bounds for $\psi(x)$.   Currently, the best bounds for $\psi(x)$ for $x \le e^{2394.19 \ldots}$ follow from the work of  Chirre-Helfgott \cite{CH}.  Prior to that, in this range the best bounds were due to
B\"{u}the \cite{But16}. For $x \ge e^{2394.19 \ldots}$ the work of \cite{FKS} and \cite{JY} give the best bounds for $\psi(x)$.
 Second, we give significant improvements to bounds for $J_m(x)$ than were proven in \cite{Ram} and \cite{Van}.  To bound these zero sums we make use of the zero-density result in \cite{KLN} and a recent technique for treating zero sums that was introduced in \cite{FKS}.

The following theorem gives two types of bounds for $J_m(x)$: an exponential-form and a log-form.
 \begin{theorem}\label{JmLemma}
 \begin{enumerate}[label=(\roman*)]
  \item
 Let $m \in \mathbb{N}$, $x_0 \ge 2$, and $\s_0 > \frac{5}{8}$.
 There exists a positive constant  $A_m(x_0,\s_0)$ such that
if  $x \ge x_0 \ge  \exp \Big( \frac{4(m+1)^3}{R(1-\s_0)^2}\Big)$,
then
 \begin{equation}
  \label{Jmbd1}
  J_m(x) \le
  A_m(x_0,\s_0) (\log x)^{B}
 \exp(-C \sqrt{m+1} \sqrt{\log x})
 \end{equation}
 where $A_m(x_0,\s_0)$ is defined in \eqref{Amx}
and $B,C$ are given in \eqref{BC}.
   \item
  Let $m \in \mathbb{N}$,  $\ell \ge 0$, and  $x_0 \ge e^{20}$.
Then there exists a positive constant
$\widetilde{\delta}_{m,\ell}=\widetilde{\delta}_{m,\ell}(x_0)$
(see equation \eqref{tildedeltamlx0})
such that
\begin{align}\label{Jmbd2}
 J_m(x) \le \frac{ \widetilde{\delta}_{m,\ell}(x_0)}{(\log x)^{\ell}}
 \quad \text{ for } x \ge x_0.
\end{align}
 \end{enumerate}
\end{theorem}
Here are sample values of the constants $A_m(x_0,\s_0)$.
\begin{center}
\begin{tabular}{|c|c|c|c|c|}
\hline
$m$ & $\sigma_0$ & $\log t(m,\sigma_0)$ & $\log x_0$ & $A_m(x_0,\sigma_0)$ \\
\hline
1 & 0.65 & 46.926 \ldots  & 50     & 0.0014464 \\
1 & 0.70 & 63.872 \ldots   & 70     & $1.4882\cdot10^{-6}$ \\
1 & 0.80 & 143.71 \ldots   & 150    & $1.6694\cdot10^{-11}$ \\
1 & 0.90 & 574.85  \ldots  & 600    & $5.0004\cdot10^{-14}$ \\
2 & 0.65 & 158.37 \ldots   & 160    & $9.4815\cdot10^{-19}$ \\
2 & 0.70 & 215.57 \ldots   & 220    & $1.5719\cdot10^{-20}$ \\
2 & 0.80 & 485.03  \ldots  & 500    & $7.0882\cdot10^{-24}$ \\
2 & 0.90 & $194013.2 \ldots$ & 200000 & $6.1156\cdot10^{-188}$ \\
\hline
\end{tabular}
\end{center}

In the above table, we have defined
\begin{equation}
  \label{tmsig}
   t(m,\s_0) = \frac{4(m+1)^3}{R(1-\s_0)^2}
\end{equation}
and this notation will be used throughout the article.

This theorem and ideas from its proof  likely will be useful for bounding other prime number error terms
as the sums $J_m(x)$ naturally arise in many smoothing arguments.
By applying the explicit formula for $\lambda(x)$ \eqref{Mertensexplicit} along with \cref{JmLemma} we obtain the following result.
 \begin{theorem} \label{thmexpdecay2nd3rdterms}
Let $x_0 \ge 2$ and $\s_0 > \frac{5}{8}$.  Then there exists a positive constant
$A''(x_0,\s_0) $ such that  for $x \ge x_0 \ge  \exp \Big( \frac{108}{R(1-\s_0)^2}\Big)$ we have
 \begin{equation}\label{Van:thm4:equ2}
 | \lambda(x) - \log \log x - M | \le \frac{|\vartheta(x) - x|}{x \log x} + A''(x_0,\s_0) (\log x)^{B-1}
 \exp(-C \sqrt{2} \sqrt{\log x})
 \end{equation}
 where $A''(x_0,\s_0)$ is defined in \eqref{App}.
\end{theorem}
We have the following values for $A''(x_0,\s_0)$.
\begin{center}
\begin{tabular}{|c|c|c|c|c|}
\hline
$m$ & $\sigma$ & $\log t(2,\s_0)$ & $\log x_0$ & $A''(x_0,\sigma_0)$ \\
\hline
 1 & 0.65 & $158.37 \ldots$  & 160& $6.2463  \cdot 10^{-11}$\\
 1 & 0.7 &  $215.57 \ldots$  & 220& $1.6075  \cdot 10^{-12}$\\
 1 & 0.8 & $485.03  \ldots$  & 500 & $7.5214 \cdot  10^{-14}$\\
 1 & 0.9 & $194013.23  \ldots$ & 200000 & $1.091  \cdot 10^{-114}$\\
  \hline
\end{tabular}
\end{center}

We can compare this to  Vanlalngaia's \cite[Theorem 4]{Van}.
which had an error term of the shape $\mathcal{O}^{*}(\frac{7 \cdot 10^{-6}}{\log x}
\exp(-0.8379\ldots\sqrt{\log x}))$, for $x \ge e^{4638}$.
 Note that we have
$\mathcal{O}^{*} ( (6.25 \cdot 10^{-11}) (\log x)^{\frac{1}{2}} \exp(-1.1988 \sqrt{\log x} ))$ for $x \ge e^{160}.$

\subsection{Expectations for the size of the error terms}

Finally, we discuss what should be the ``true size" of the error term for the Merten sum.
Note that Schoenfeld showed the Riemann hypothesis \cite[Corollary 2, p. 340]{Schoenfeld76} implies
  \begin{equation}\label{underRH}
       \lambda(x) =   \log \log x +M+ \mathcal{O}^{*} \Big(  \frac{3 \log x +4}{8 \pi x^{\frac{1}{2}}}  \Big)
     \  \text{ for all } \ x \ge 13.5.
  \end{equation}
  However, we expect the error term to be much smaller.
  Note there is the following conjecture
 of Montgomery (see \cite{Mo}).
 \begin{conjecture} \label{Montgomeryconjecture}
 We have
\begin{equation}
       \limsup_{x \to \infty} \frac{\psi(x)-x}{\sqrt{x} (\log \log  \log x)^2} = \frac{1}{2 \pi}
 \  \text{ and }\   \liminf_{x \to \infty} \frac{\psi(x)-x}{\sqrt{x} (\log \log \log x)^2} = -\frac{1}{2 \pi}.
\end{equation}
\end{conjecture}
Since $\psi(x)=\vartheta(x) + O(\sqrt{x})$ and by  formula \eqref{ps}
which relates $\sum_{p \le x} \frac{1}{p}$ to $\vartheta(x)$ we arrive at the following conjectures.
\begin{conjecture} \label{Mertensconjectures}
We have
\begin{align*}
   \underline{\overline{\lim}}
 \frac{\sqrt{x} \log x}{(\log \log \log x)^2} \Bigg( \lambda(x) -   \log \log x -M \Bigg) & = \pm \frac{1}{2 \pi},  \\
     \underline{\overline{\lim}}
   \frac{\sqrt{x} }{(\log \log \log x)^2} \Bigg( \Upsilon(x) -   \log  +M' \Bigg)  & =
   \pm \frac{1}{2 \pi}, \\
     \underline{\overline{\lim}}
   \frac{\sqrt{x} }{(\log \log \log x)^2} \Bigg(  \widetilde{\psi}(x) -   \log  +\gamma \Bigg)  & =
   \pm \frac{1}{2 \pi}.  \\
\end{align*}
\end{conjecture}
The first conjecture first appears in \cite{La}. The recent article \cite{Ng} provides an explanation of the first conjecture.   Numerics regarding these conjectures may be found in Tables \ref{GRHtable} and \ref{Conjtable} in the Appendix.

\subsection{Comparison with historical results}\label{History}

In the literature, the bounds for the error terms for the three Mertens sums introduced in \eqref{Mertens1}, \eqref{Mertens2}, \eqref{Mertens3} have historically and dominantly been of the log-form.
The following tables display constants involved in such bounds with a dash indicating that the related term does not exist in the bound (or the given coefficient is $0$).
We recall that $\ell$ is the power of $(\log x)$ in the bound for the error term and $x_0$ is the value for which the bound is valid for all $x\ge x_0$. Historically, bounds have been investigated for $\ell=1,2,$ or $3$. Our Theorems
\ref{1overpthm}, \ref{bnd:sum_logp}, \ref{bnd:sum_lambdan}
 are valid for larger values of $\ell$, and our Tables \ref{Altable}, \ref{Bltable}, \ref{Cltable}
 in Appendix \ref{Section:Tables} display values for $\ell=1,2,3,4$, and $5$.
 Below we provides lists of previous results.
Further, previous results have focused on very specific values of $x_0$.  In our work, we have provided more extensive tables of values of $x_0$.

For the first Mertens sum $\lambda(x)$, there exists positive constants $\mathcal{A}_{\ell}(x_0)$ and $\widetilde{\mathcal{A}}_{\ell+1}(x_0)$ such that
\begin{equation}
  |\lambda(x) - \log \log x - M  | \le \frac{\mathcal{A}_{\ell}(x_0)}{(\log  x)^{\ell}} + \frac{\widetilde{\mathcal{A}}_{\ell+1}(x_0)}{(\log  x)^{\ell+1}} \ \text{ for all }\  x \ge x_0.
\end{equation}
\begin{center}
 \begin{tabular}{|c|cccc|}
  \hline
  Author & $\ell$  &  $\mathcal{A}_{\ell}$ & $\widetilde{\mathcal{A}}_{\ell+1}$ & $x_0$ \\ \hline
  Rosser $\&$ Schoenfeld (1962) \cite{RS62}  & $2$  & $0.5$ & - & $286$ \\
  					    & $2$ & $1$ & - & $1$ \\ \hline
  Dusart (1999) \cite{Dusart1999} \vphantom{\Big(}  & $2$   & $\tfrac{1}{10}$ & $\tfrac{4}{15}$ & $10\ 372$ \\ \hline
  \multirow{4}{*}{Vanlalngaia  (2017) \cite{Van} } & $3$ &  $4$ & - & $2$ \\
				& $3$ & $2.3$ & - & $1000$ \\
				& $3$ & $1$ & - & $24\ 284$ \\
				& $3$ & $0.21$ & - & $\exp(4635)$ \\ \hline
 Dusart (2018) \cite{Dusart2018} & $3$ &  $0.2$ & - & $2\, 278\, 383$ \\ \hline
 Axler (2018) \cite{Axler2018} \vphantom{\Big(} & $3$  & $\tfrac{1}{20}$ & $\frac{3}{16}$ & $46\, 909\, 074$ \\ \hline
 \end{tabular}
\end{center}

Note that Dusart had superior estimates to Vanlalngaia's for $\ell=3$ as he
used superior estimates for $\psi(x)$. These are to be compared to our Theorem \ref{1overpthm} and Table
\ref{Altable}. Below we provide a sample of some selected values.  Note that these improve all of
Vanlalngaia's results.  As a special case of Theorem \ref{1overpthm}, we obtain the following table.  We prove the results in this table by combining the first row of Table \ref{Altable}, along with an exact computation as in Table \ref{MkmktableThm3}.

\begin{center}
 \begin{tabular}{|ccc|}
  \hline
  $\ell$ &  $\mathcal{A}_{\ell}(x_0)$ & $x_0$ \\ \hline
  $3$    & $3.69$                   & $2$ \\
	 $3$    & $2.246529865782...$                   & $1000$ \\
	 $3$    & $0.999440022051....$                  & $24\ 284$ \\
	 $3$    & $1.087\cdot10^{-11}$  & $\exp(4500)$ \\
  $3$    & $0.199916579111...$                  & $2\, 278\, 383$ \\
  $3$    & $0.0606100924826604...$                  & $46\, 909\, 074$ \\
   \hline
 \end{tabular}
\end{center}
Note Axler's comstant is $\frac1{20}+\frac3{16\times\log(46909074)}=0.0606149770\ldots$ which is essentially equivalent to our constant.

For the second Mertens sum $\Upsilon(x)$, there exists positive constants $\mathcal{B}_{\ell}(x_0)$ and $\widetilde{\mathcal{B}}_{\ell+1}(x_0)$ such that
\begin{equation}
|\Upsilon(x)  - \log x + M' | \le\frac{\mathcal{B}_{\ell}(x_0)}{(\log  x)^{\ell}} + \frac{\widetilde{\mathcal{B}}_{\ell+1}(x_0)}{(\log  x)^{\ell+1}}  \ \text{ for all }\ x \ge x_0.
\end{equation}
\begin{center}
 \begin{tabular}{|c|cccc|}
  \hline
  Author & $\ell$  &  $\mathcal{B}_{\ell}(x_0)$ & $\widetilde{\mathcal{B}}_{\ell+1}(x_0)$ & $x_0$ \\ \hline

  Rosser $\&$ Schoenfeld (1962) \cite{RS62} & $1$ &  $0.5$ & - & $319$ \\
                                            & $1$ & $1$   & - & $32$ \\ \hline
  Dusart (1999) \cite{Dusart1999} & $1$  & $0.2$ & $0.2$ & $2974$ \\ \hline
  Dusart (2018) \cite{Dusart2018}
   & $2$ & $0.3$ & - & $912\,560$ \\ \hline
  Axler (2018) \cite{Axler2018} & $2$ & $\tfrac{3}{40}$ & $\tfrac{3}{20}$ & $30\,972\,320$ \\ \hline
\end{tabular}
\end{center}

These are to be compared to our Theorem \ref{bnd:sum_logp} and Table \ref{Bltable} which yield the following results:
\begin{center}
 \begin{tabular}{|ccc|}
  \hline
$\ell$ &  $\mathcal{B}_{\ell}(x_0)$ & $x_0$ \\ \hline
$1$    &  $0.22467979505390646...$                 & $2974$ \\
$2$    & $0.299999453231456...$                   & $912\,560$ \\
$2$    & $0.0836945302901...$                  & $30\,972\,320$ \\
\hline
 \end{tabular}
\end{center}

Note that $0.2+\frac{0.2}{\log(2974)} = 0.2250073\ldots$
and
$\frac3{40}+\frac3{20\times\log(30\,972\,320)}=0.0836963\ldots$

The third Mertens sum $\widetilde{\psi}(x)$ is similar to the second one, so we essentially expect the same bounds. In his work towards  Landau's Equivalence Conjecture regarding how the error terms for $\psi$ and $\widetilde{\psi}$ relate, Ramar\'{e} establishes in \cite[Corollary]{Ram} the first explicit bound for $\widetilde{\psi}(x)$
 \begin{equation}
  |  \widetilde{\psi}(x) - \log x + \gamma | \le \frac{\mathcal{C}_{\ell}(x_0)}{(\log  x)^{\ell}} \text{ for all }\  x \ge x_0,
 \end{equation}
 for particular values of $\ell, x_0$
where $\mathcal{C}_{\ell}(x_0)$ is a positive constant.
Vanlanlgaia in \cite[Theorem 9, Corollary 10]{Van} improves Ramar\'{e}'s results.
Their results are recorded in the following table.
 \begin{center}
 \begin{tabular}{|c|ccc|}
  \hline
  Author & $\ell$  &  $\mathcal{C}_{\ell}(x_0)$ & $x_0$ \\ \hline
  Ramar\'e (2012) \cite{Ram}\footnote{Ramar\'e claims the value $x_0=10^{10}$ can be lowered to $x_0=23$, however, there is an error in this part of the proof as he evaluates a decreasing function at the right endpoint rather than left endpoint to find the maximum.} & $1$ &  $0.0067$ & $10^{10}$ \\
  & $2$ &  $1.833$ & $2$ \\ \hline
 Dusart (2018) \cite{Dusart2018} & $1$ & $0.0008999$ & $e^{20}$\\
 & $2$ & $0.01808$ & $e^{20}$ \\
  \hline
 \end{tabular}
\end{center}
These are to be compared to our Theorem \ref{bnd:sum_lambdan} and Table \ref{Cltable} which imply the following results:
\begin{center}
 \begin{tabular}{|ccc|}
  \hline
$\ell$ &  $\mathcal{C}_{\ell}(x_0)$ & $x_0$ \\ \hline
$1$    & $0.2967146958276\ldots$                   & $23$ \\
$1$    & $0.000586$            & $e^{20}$ \\
$1$    & $0.000157$            & $10^{10}$ \\
$2$    & $1.28$                   & $2$ \\
$2$    & $0.0118$            & $e^{20}$ \\
\hline
 \end{tabular}
\end{center}
Note that there is the following related bound of Balazard \cite[Appendix, Lemma A.1]{RZ}:  $ \widetilde{\psi}(x) \le \log x$  for $x \ge 1$.

Finally, as a consequence of the estimates on Mertens sums, we give historical results regarding Mertens products of primes.
It follows from bounds on the first sum $\lambda (x)$ that there exist positive constants $\mathcal{A}_{\ell}(x_0)$ and $ \mathcal{A}'_{\ell}(x_0)$ such that
\begin{equation}
\begin{split}
     \frac{e^{-\gamma}}{\log x} \Big( 1- \frac{\mathcal{A}_{\ell}(x_1)}{(\log x)^{\ell}} \Big) \le &\prod_{p \le x} \Big(1-\frac{1}{p} \Big)\  \text{ for all }\  x \ge x_1, \\
  &\prod_{p \le x} \Big(1-\frac{1}{p} \Big)  \le \frac{e^{-\gamma}}{\log x} \Big( 1+ \frac{\mathcal{A}'_{\ell}(x_0)}{(\log x)^{\ell}} \Big) \ \text{ for all }\  x \ge x_0.
\end{split}
\end{equation}

\begin{center}
 \begin{tabular}{|c|ccccc|}
  \hline
  Author & $\ell$ & $\mathcal{A}_{\ell}(x_1)$ & $x_1$ & $\mathcal{A}'_{\ell}(x_0)$ & $x_0$ \\ \hline
  \multirow{2}{*}{Rosser $\&$ Schoenfeld (1962) \cite{RS62}} & $2$ & $0.5$ & $285$ & $0.5$ & $1$ \\
					& $2$ & $1$ & $1$ & - & - \\ \hline
  Dusart 1999 \cite{Dusart1999} & $2$ & $0.2$ & $1$ & $0.2$ & $2973$ \\ \hline
  Dusart 2018 \cite{Dusart2018} & $3$ & $0.2$ & $2\, 278\, 382$ & $0.2$ & $2\, 278\, 382$ \\ \hline
 \end{tabular}
\end{center}

Finally, for the inverse of this product, we have that there exist positive constants $\mathcal{D}_{\ell}(x_0)$ and $\mathcal{D}'_{\ell}(x_0)$ such that
\begin{equation}
\begin{split}
  e^{\gamma} \log x \Big( 1- \frac{\mathcal{D}_{\ell}(x_1)}{(\log x)^{\ell}} \Big) \le &\prod_{p \le x} \frac{p}{p-1}  \ \text{ for all }\  x \ge x_1, \\
  &\prod_{p \le x} \frac{p}{p-1}  \le e^{\gamma} \log x \Big( 1+ \frac{\mathcal{D}'_{\ell}(x_0)}{(\log x)^{\ell}} \Big) \ \text{ for all }\  x \ge x_0.
\end{split}
\end{equation}

\begin{center}
 \begin{tabular}{|c|ccccc|}
  \hline
  Author & $\big. \ell$ & $\mathcal{D}_{\ell}(x_1)$ & $x_1$ & $\mathcal{D}'_{\ell}(x_0)$ & $x_0$ \\ \hline
  \multirow{2}{*}{Rosser $\&$ Schoenfeld (1962) \cite{RS62}} & $2$ & $0.5$ & $1$ & $0.5$ & $286$ \\
			    & $2$ & - & - & $1$ & $1$ \\ \hline
  Dusart (1999) \cite{Dusart1999} & $2$ & $0.2$ & $1$ & $0.2$ & $2973$ \\ \hline
  Dusart (2018) \cite{Dusart2018} & $3$ & $0.2$ & $2\, 278\, 382$ & $0.2$ & $2\, 278\, 382$ \\ \hline
 \end{tabular}
\end{center}

\subsection{Some key tools and lemmas.}

The rest of the article is largely devoted to bounding $J_m(x)$.
In bounding this sum we require three facts about the zeros of the zeta function:
a classical zero-free region, a zero-density estimate,  and a partial numerical verification of the Riemann Hypothesis.
For the last integral in \eqref{ps2} we require a bound on the difference on $\psi(x)-\vartheta(x)$.

 \noindent {\it Explicit zero-free region.} Let $R$ be a constant such that Riemann zeta function does not vanish
in the region:
\begin{equation}
  \label{ZFR}
 \mathfrak{Re} s \ge 1-\frac{1}{R\log \mathfrak{Im} s}\ \text{for } |\mathfrak{Im} s| \ge 3.
 \end{equation}
The current best published value for  $R$  is  \cite{MTY}
\begin{align}
   \label{Rnumeric}
  R = 5.558691
\end{align}
\noindent {\it Zero-density}. Let $N(\s, T)$ denote the number of non-trivial zeros $\rho$ of the Riemann zeta function with  $\sigma \le \Re(\rho) < 1$ and $0 \le \Im(\rho) \le T$.
\begin{definition}[Zero-density bound]\label{ZDB}
We say that $N(\s,T)$ satisfies (ZDB) if there exists a positive constant  $ \s_0$ such that
\begin{equation}\label{bnd-ZDB}
 N(\s,T) \le  \tilde{N}(\s,T) \ \text{ for all }\  \s >\s_0 ,
 \end{equation}
where $ \tilde{N}(\s,T) $ is of the form
\begin{equation}\label{def-Ntilde}
 \tilde{N}(\s,T) = c_1 T^{p} (\log T )^{q} + c_2 (\log T)^2
  \end{equation}
for some positive functions $c_1=c_1(\s), c_2=c_2(\s), p=p(\s)$, and $q=q(\s)$ with $0<p(\s)<1$.
\end{definition}
Note that \cite[Table 7]{FKS} calculates the bounds from \cite[Theorem 1.1]{KLN}. Values displayed are not affected by successive corrections (see \cite[Remark 1.5 (8)]{FKS}).

\noindent {\it Partial verification of RH.} We shall denote $H$ to be the height to which the Riemann hypothesis has been verified.  Namely,
 \begin{equation}
  \label{H} \text{Let}\ H>0 \ \text{such that if}\ \zeta(\b+i\g)=0 \ \text{and}\  0< \beta <1,  0<\gamma<H, \ \text{then}\ \b=1/2.
\end{equation}
The current best values for $H$ is \cite[Theorem 1]{PlaTruH2020}
\begin{align}
  \label{Hnumeric}
    H & = 3 \, 000 \,175 \, 332 \,  800.
\end{align}

\noindent {\it The difference $\psi(x)-\vartheta(x)$.}
We also invoke a sharp bound the difference for $\psi(x)-\vartheta(x)$ for all $x \ge2$.  This result allows us to convert estimates
for $\psi(x)$ into bounds for $\theta(x)$.
\begin{proposition}\label{psi-theta:ExplicitCor}
Let $x\ge  x_0 \ge  2$.
Then,
  \begin{equation}
    \label{psithetabd}
 \psi(x) - \vartheta(x) <  1.00000002(x^{1/2} + x^{1/3} +  x^{1/5}) + 0.94(x^{1/4}  + x^{1/6} +x^{1/10}).
 \end{equation}
Moreover, if
$ E_{\psi}(x) \leq \varepsilon_{\psi,num}(x_0) $,
then
\[  -\varepsilon_{\theta,num}(x_0)  \leq \frac{\theta(x)-x}{x}\leq \varepsilon_{\psi,num}(x_0) < \varepsilon_{\theta,num}(x_0)  ,\]
where
\begin{equation}
\label{def-epsnumtheta}
\varepsilon_{\theta,num}(x_0) = \varepsilon_{\psi,num}(x_0) + 1.00000002(x_0^{-1/2} + x_0^{-2/3} +  x_0^{-4/5}) + 0.94(x_0^{-3/4}  + x_0^{-5/6} +x_0^{-9/10}) .
\end{equation}
\end{proposition}
The above is easily derived from the proof of \cite[Proposition 17]{FKS2} and is a slight improvement on \cite[Corollary 9.1]{BKLNW}.

\subsection{Organization of the rest of the article and notation.}
 In Section \ref{section:proofThm2}, the proof of Theorem \ref{1overpthm} which bounds
$\lambda(x)$ is given.
The proof of Theorem \ref{1overpthm} is subject to certain bounds on weighted zeros sums of the zeta function.   Bounds for these zero sums
are proven in Section \ref{section:zerosums}.
In Section \ref{section:othersums}, the proofs of Theorems \ref{bnd:sum_logp}
and \ref{bnd:sum_lambdan}, which concern the sums
$\Upsilon(x)$ and $\widetilde{\psi}(x)$, are given.
In Section \ref{sec:ProdPrimes}, Theorems \ref{prod1} and \ref{prod2} on the products of primes are deduced.   In Section \ref{section:numerics}, numerical computations
of the Mertens sums are given in the region $x \in [1,10^{10}]$. \\

\noindent {\it Notation}.  Throughout this article we write $f= O^{*}(g)$  to mean $|f| \le g$. We write
$f(x) = O(g(x))$ to mean there exist positive constants $x_0,C$ such that $|f(x)| \le Cg(x)$ for $x \ge x_0$.  We write $f(x) \sim g(x)$ to mean $\lim_{x \to \infty} \frac{f(x)}{g(x)} =1 $.

\section{Proof of Theorem \ref{1overpthm}} \label{section:proofThm2}

In this section, we prove Theorem \ref{1overpthm}.  The bounds in this theorem will be derived from an
explicit version of \eqref{explicitapprox} which is given in Proposition \ref{keyprop} below.
In particular, the explicit version of \eqref{explicitapprox}  is given by \eqref{keyinequalityB} below. The starting point is the identity
\begin{equation}
  \label{ps2}
   \lambda(x) =   \log \log x +M+ \frac{\vartheta(x) - x}{x \log x}  - \int_x^{\infty} \frac{1 + \log t}{t^2 \log ^2 t} (\psi(t)-t)dt + \int_x^{\infty} \frac{1 + \log t}{t^2 \log ^2 t} (\psi(t)-\vartheta(t))dt
\end{equation}
which follows immediately from \eqref{ps}.  The key point is that the first integral possesses an explicit formula in terms of zeros of zeta,   by applying Lemma 4 of  \cite{RamSau}.
his idea appears to have been first used by  Ramar\'{e} in \cite{Ram} in the context of  $\widetilde{\psi}(x)$
and then Vanlalgaia applied the same idea to $\lambda(x)$.
The second integral in \eqref{ps2}  is a small error term of size $O((\log x)x^{-\frac{1}{2}})$ and partially accounts for the exclusion of prime powers.

For the first integral in \eqref{ps2} we apply the following lemma which gives a ``Riemann-Guinand type explicit formula." From the explicit formula we deduce a precise inequality.
\begin{proposition}
\label{explicitformulaMertens}
\begin{itemize}
\item[(i)]  For $x \ge 2$
\begin{equation}
\begin{split}
   \label{integratedexplicit}
   \int_x^{\infty} \frac{\log t +1}{t^2 (\log  t)^2} (\psi(t)-t) dt
& =   \frac{\log x +1}{(\log  x)^2}   \sum_{\rho} \frac{x^{\rho-1}}{\rho(\rho-1)}
 + \sum_{\rho} \frac{1}{\rho} \Gamma\!\bigl(-1,\,- (\rho-1)\log x\bigr) \\
 & -2
\sum_{\rho} \frac{\rho-1}{\rho}
 \Gamma\!\bigl(-2,\,- (\rho-1)\log x\bigr)  {\color{blue}-} \frac{ \log 2 \pi}{x \log x} -\frac{1}{2} \kappa(x)
\end{split}
\end{equation}
where   $\Gamma(s,x)$ is the incomplete gamma function defined in \eqref{incgamma} and
\begin{equation}
 \label{kappa}
 \kappa(x) :=
 \int_x^{\infty} \log(1-t^{-2}) \frac{\log t +1}{t^2 (\log  t)^2} dt
 =\frac{\log(1-x^{-2})}{x\log x}
-
2\int_x^\infty
\frac{dt}{t^2(t^2-1)\log t}.
\end{equation}
\item[(ii)]  Let $x \ge 2$.  Then
\begin{equation}
  \label{integral1bd}
 \Big| \int_x^{\infty} \frac{1 + \log t}{t^2 \log ^2 t} (\psi(t)-t)dt \Big| \le \Big(\frac{\log x +1}{ \log ^2 x} \Big) J_1(x) + \Big( \frac{\log x +4}{\log ^3 x} \Big) J_2(x) + \frac{2.15}{x \log x}.
\end{equation}
where  $J_m(x)$ are defined in \eqref{zerosums}.
\end{itemize}
\end{proposition}
A proof of  Proposition \ref{explicitformulaMertens} is provided at the end of the section and is an improvement of
Vanlalngaia's bound  $\frac{\log x +1}{ \log ^2 x} J_1(x) +  \frac{2\log x +4}{\log ^3 x} J_2(x) +
\frac{(1+\log x) \log (2 \pi e)}{x \log^2 x}
$ \cite[pp.9-11]{Van}
\footnote{Vanlalngaia uses the notation $S_m(x)$ instead of our $J_m(x)$.}.
The last integral in \eqref{ps2} is bounded with the following.
\begin{lemma}
  \label{integral2}
Let $x \ge 2$, $\eta_1'  = 1.00000002$,  and
   $\eta_2'  = 0.94$.
Then
\begin{equation}
  \label{integral2bd}
 \int_x^{\infty} \frac{1 + \log t}{t^2 \log ^2 t} (\psi(t)-\vartheta(t))dt \le \frac{\log x +1}{\log ^2 x} \Big(  \eta_1'(2x^{-\tfrac{1}2}+\tfrac32x^{-\tfrac{2}3}+\tfrac54x^{-\tfrac{4}5})+\eta_2'(\tfrac43 x^{-\tfrac{3}4}+\tfrac65x^{-\tfrac{5}6}+\tfrac{10}9x^{-\tfrac{9}{10}} )\Big).
\end{equation}
\end{lemma}
\begin{proof}
Let $x \ge 2$.
Since $\frac{1 + \log t}{\log^2 t}$  decreases for $t > 1$ we have
\begin{align*}
 \int_x^\infty \frac{\log t + 1}{t^2 \log^2 t} \Big( \psi(t) - \vartheta(t) \Big)dt & \le \frac{\log x + 1}{\log^2 x} \int_x^\infty \frac{ \psi(t) - \vartheta(t) }{t^2} dt .
\end{align*}
Inserting the bound   \eqref{psithetabd}, valid for $t \ge 2$, and then integrating term by term gives  the result.
\end{proof}

Combining \eqref{ps2} with \cref{explicitformulaMertens}
and
\cref{integral2}  yields the following
key bound. A version of this bound is implicit in \cite[Theorem 4]{Van}
\begin{proposition} \label{keyprop}
Let $x \ge 2$.
Then we have
\begin{equation}
\begin{split}
  \label{keyinequalityB}
  | \lambda(x)-  \log \log x - M |
  & \le \frac{|\vartheta(x) - x|}{x \log x}
+ \mathcal{E}_2(x)+\mathcal{E}_3(x)
\end{split}
\end{equation}
where
\begin{equation}
\label{E2}
     \mathcal{E}_2(x) =  \Big(\frac{\log x +1}{ \log ^2 x} \Big) J_1(x) +
      \Big( \frac{\log x +4}{\log ^3 x} \Big) J_2(x),
\end{equation}
\begin{equation}
   \label{E3}
        \mathcal{E}_3(x) =  \frac{\log x +1}{\log ^2 x}   \Big( \eta_1' r_1(x)+\eta_2' r_2(x)
        \Big) + \frac{2.18}{x \log x},
\end{equation}
\begin{equation}
  \label{etajconstants}
 \eta_1'  = 1.00000002, \,
  \eta_2'  = 0.94, \,
\end{equation}
\begin{equation}
  \label{r1r2}
   r_1(x) = 2x^{-\tfrac{1}2}+\tfrac32x^{-\tfrac{2}3}+\tfrac54x^{-\tfrac{4}5}, \,
  r_2(x) = \tfrac43 x^{-\tfrac{3}4}+\tfrac65x^{-\tfrac{5}6}+\tfrac{10}9x^{-\tfrac{9}{10}}.
\end{equation}
\end{proposition}
Note that  $\mathcal{E}_3(x) =O(  x^{-\frac{1}{2}} (\log x)^{-1})$ is clearly smaller than the other terms in \eqref{keyinequalityB}.
The largest error term is $  \frac{|\vartheta(x) - x|}{x \log x} $ and it will be bounded with the  bounds from
Theorem \ref{thetathm}.
In order to bound $\mathcal{E}_2(x)$ it suffices to bound $J_m(x)$ for $m \in \{1,2\}$.  This will be done in Section \ref{section:zerosums}.

\cref{E2Lemma} provides bounds for  $\mathcal{E}_2(x)$ and is a simple corollary of \cref{JmLemma}.
The proofs of  \cref{E2Lemma}  and \cref{JmLemma} will be provided at the end of Section \ref{section:zerosums} as it will require various results about zero sums.
\begin{lemma}\label{E2Lemma}
Let $\mathcal{E}_2(x)$ be defined as in \eqref{E2}.
 \begin{enumerate}[label=(\roman*)]
  \item
  Let $x_0 \ge 2$ and $\s_0 > \frac{5}{8}$.
 There exists a positive constant  $A'(x_0,\s_0)$ such that if  $x
 \ge x_0 \ge
  \exp \Big( \frac{108}{R(1-\s_0)^2}\Big)$, then
   \begin{equation}\label{bnd:E2a}
    \mathcal{E}_2(x) \le A'(x_0,\s_0)  (\log x)^{B-1}  \exp ( -C \sqrt{2} \sqrt{\log X} )
    \text{ for } x \ge x_0
   \end{equation}
   where $A'(x_0,\s_0)$ is defined in \eqref{Ap} and $B$, $C$ are defined in \eqref{BC}.
   \item
   Let $\ell \ge 1$ and $x_0 \ge {\color{red} e^{20}}$.
Then there exists a positive constant   $\delta_{\ell}(x_0)$
(see equation \eqref{deltalx0}) such that
\begin{align}\label{bnd:E2b}
 \mathcal{E}_2(x) \le \frac{\delta_{\ell}(x_0)}{(\log x)^{\ell}}
 \quad \text{ for } x \ge x_0.
\end{align}
 \end{enumerate}
\end{lemma}
The following lemma provides bounds for the functions $g_{a,b,c}(x)$.
This is a calculus exercise which
follows directly from \cite[Lemma 10, p. 11]{FKS2}.
\begin{lemma}\label{dec_funcs}
 Let $u,v,b > 0$.  Let   $x_0 > 1$ and let
 $g_{a,b,c}(x) = x^{-a}(\log x)^b \exp(c\sqrt{\log x})$.
 We have that
 \begin{enumerate}
  \item[(i)] \label{dec_func_f} $\max_{x \ge x_0}(g_{0,v,-b}(x))  = \begin{cases}
                           g_{0,v,-b}( \exp ( \frac{4v^2}{b^2} ))  , \qquad &\text{if }x_0 < \exp ( \frac{4v^2}{b^2} ) \\
                               g_{0,v,-b}( x_0)    \qquad &\text{if }x_0 \ge \exp ( \frac{4v^2}{b^2} )
                                \end{cases},$ \\
  \item[(ii)]  $\max_{x \ge x_0}(g_{u,v,0}(x))  =
  \begin{cases}
                               g_{u,v,0}(e^{v/u})  , \qquad &\text{if }x_0 < \exp ( \frac{v}{u} ) \\
                               g_{u,v,0}(x_0)   \qquad &\text{if }x_0 \ge \exp ( \frac{v}{u} )
                                \end{cases},$ \\
  $g_{u,v,0} $ is decreasing for $x \ge \max(1,e^{v/u})$,\label{dec_func_g} \\
  \item[(iii)]
  $g_{u,-v,b} $ is decreasing for $x >0$.
   \label{dec_func_h}
 \end{enumerate}
\end{lemma}

The next lemma gives a bound for $\mathcal{E}_3(x).$
\begin{lemma} \label{E3bd}
 Let $\ell \in \mathbb{Z}_{\ge 0}$ and $x_0 \ge e^{2(\ell-1)}$.
Then for $x \ge x_0$ we have
 \begin{equation}\label{bnd:E3_2}
  \mathcal{E}_3(x) \le \frac{\kappa_{\ell}(x_0)}{(\log x)^{\ell}}
 \end{equation}
where
\begin{equation}
\begin{split}
  \label{Ckx}
\kappa_{\ell}(x_0) & = (\log x_0)^{\ell} \mathcal{E}_3(x_0) =  ((\log x_0)^{\ell-1}
 + (\log x_0)^{\ell-2})
\Big( \eta_1' r_1(x_0)+\eta_2' r_2(x_0)
        \Big)
        +  \frac{2.18(\log x_0)^{\ell-1}}{x_0} ,
\end{split}
\end{equation}
and $r_1(x),r_2(x)$ are defined in \eqref{r1r2}.
\end{lemma}

\begin{proof}
By the definition \eqref{Ckx} we expand out to obtain  a linear combination of $g_{u,v,0}(x)$
where $(u,v)$ satisfy
$u \in \left\{\frac12,\frac23, \frac34, \frac45, \frac56, \frac{9}{10} \right\}$  and $v \in \{\ell-1,\ell-2\}$.
By an application of  \cref{dec_funcs}
to these values of $(u,v)$ we find that each $g_{u,v,0}(x)$ is decreasing as long as $x \ge e^{2(\ell-1)}$.
It follows that $\kappa_{\ell}(x)$ is decreasing for $x \ge x_0 \ge e^{2(\ell-1)}$ and this completes the lemma.
\end{proof}

\begin{proof}[Proof of Theorem \ref{thmexpdecay2nd3rdterms}]
Let $x_0 \ge 2$.  Assume that $x \ge x_0 \ge  \exp \Big( \frac{108}{R(1-\s_0)^2}\Big)$.
By applying the bound \eqref{bnd:E2a} from \cref{E2Lemma} (i) to \eqref{keyinequalityB}, we have
\begin{multline*}
 | \lambda(x) - \log \log x - M | \le \frac{|\vartheta(x) - x|}{x \log x} +
 A'(x_0,\s_0) (\log x)^{B-1}
 \exp(-C \sqrt{2} \sqrt{\log x})
  \\
  +  \frac{\log x +1}{(\log  x)^2} \Big(
   \eta_1' r_1(x) + \eta_2' r_2(x) \Big) +  \frac{2.18}{x \log x}
\end{multline*}
where $\eta_j'$ are defined in \eqref{etajconstants} and $r_j(x)$ are defined in \eqref{r1r2}.
Note that the last two terms satisfy
\begin{align*}
& A'(x_0,\s_0) (\log x)^{B-1}
 \exp(-C \sqrt{2} \sqrt{\log x}) \\
 & \times \Big(
 1+  \frac{1}{A'(x_0,\s_0) } \exp(C \sqrt{2} \sqrt{\log x})
 \left( \frac{\log x +1}{(\log  x)^{1+B}} \Big(
  \eta_1' r_1(x) + \eta_2' r_2(x)  \Big) +  \frac{2.18}{x (\log x)^B}  \right)
 \Big) \\
 & \le  A''(x_0,\s_0)   (\log x)^{B-1}
 \exp(-C \sqrt{2} \sqrt{\log x})
\end{align*}
where
 \begin{equation}
   \label{Appsigma}
  A''(x_0,\s_0) = A'(x_0,\s_0) \cdot \max_{x \ge x_0} \left(1+ \frac{\exp(C \sqrt{2} \sqrt{\log x})}{A'(x_0,\s_0)}  \left(
 \frac{\log x +1}{(\log  x)^{1+B}} \Big(
  \eta_1' r_1(x) + \eta_2' r_2(x) \Big) + \frac{2.18}{x (\log x)^B} \right) \right).
 \end{equation}
Note that if $x_0$ is large enough the maximum will occur at the left end point of $[x_0, \infty)$.
\end{proof}

With these results in hand, we now establish Theorem \ref{1overpthm}.
\begin{proof}[Proof of Theorem \ref{1overpthm}]
(i) Let $\varphi(x) = 9.2204(\log x)^{1/2}
 \exp(- C \sqrt{\log x})$ with $C= 0.84768$ given in \eqref{ImplicitExpDecaypsitheta}.
 Note we have the classic bounds of Rosser-Schoenfeld \cite[Theorem 5, Corollary, p.70]{RS62}
\[
  | \lambda(x)- \log\log x - M | \le
  \begin{cases}
  \frac{1}{2 (\log x)^2} & \text{ for } x \ge 286, \\
  \frac{1}{(\log x)^2} & \text{ for } x > 1.
  \end{cases}
\]
Observe that $\frac{1}{2 (\log x)^2}< \varphi(x)$ for $286 \le x \le e^{464.3}$ and
 $\frac{1}{ (\log x)^2}< \varphi(x)$ for $2 \le x \le 286$.  It follows that
 \begin{equation}
     | \lambda(x)- \log\log x - M | \le  9.2204(\log x)^{1/2}
 \exp(- C \sqrt{\log x}) \text{ for } 2 \le x \le e^{464.3}.
 \end{equation}
Thus, it suffices to establish a bound for $| \lambda(x)- \log\log x - M |$ for $x \ge e^{464}$.
Indeed, we shall establish a bound in the range $x \ge e^{160}$.
Let $x_0 \ge e^{160}$, $\s_0 \ge 0.65$ and $x \ge x_0 \ge \exp(t(2,\s_0))$
where $ t(m,\s_0)$ is defined in \eqref{tmsig}.
Thus we have $t(2,\s_0) = \frac{108}{R(1-\s_0)^2}$.
 By applying the bound
 \eqref{ImplicitExpDecayAtheta} in  \eqref{Van:thm4:equ2} it follows that
 \begin{multline}
   \label{firstbd}
  | \lambda(x)- \log\log x - M | \le
  A_{\vartheta}(x_0) (\log x)^{B-1}
 \exp(-C \sqrt{\log x}) + A''(x_0,\s_0) (\log x)^{B-1}
 \exp(-C \sqrt{2} \sqrt{\log x}).
 \end{multline}
 We now choose parameters $\sigma_{i}$  such that
 $$0.65=\sigma_1 < \cdots < \sigma_{r}=0.9$$
and define $y_i = \exp(t(2,\sigma_i))$.  Note that
$y_1 < y_2 < \cdots < y_r$ and $y_1 = \exp(158.3781523 )$.
Replacing $\sigma_0$ by $\sigma_j$ for $1 \le j \le R$ in \eqref{firstbd}
it follows that
\begin{multline}
  | \lambda(x)- \log\log x - M | \le
  A_{\vartheta}(x_0) (\log x)^{B-1}
 \exp(-C \sqrt{\log x}) + A''(x_0,\s_j) (\log x)^{B-1}
 \exp(-C \sqrt{2} \sqrt{\log x}) \text{ for } x \ge y_j
 \end{multline}
 Now set
 \begin{equation}
  \label{App}
 A''(x_0)
 = \begin{cases}
  A''(x_0,\s_j) & \text{ if } x_0 \in [y_j,y_{j+1} ] \text{ for some } 1 \le j \le r-1 \\
  A''(x_0,\s_r) &  \text{ for }  x_0 \ge y_r. \\
 \end{cases}
\end{equation}
 It follows that for $x \ge x_0 \ge e^{160}$,
 \begin{equation}
 \begin{split}
  | \lambda(x)- \log\log x - M | & \le
  A_{\vartheta}(x_0) (\log x)^{B-1}
 \exp(-C \sqrt{\log x}) + A''(x_0) (\log x)^{B-1}
 \exp(-C \sqrt{2} \sqrt{\log x})  \\
 & \le A_{\lambda}(x_0) (\log x)^{B-1}
 \exp(-C \sqrt{\log x})
 \end{split}
 \end{equation}
 where
 \begin{equation}
  \label{Alambda}
   A_{\lambda}(x_0)= A_{\vartheta}(x_0) + A''(x_0)  \exp((C(1- \sqrt{2})) \sqrt{\log x_0})
 \end{equation}
 and we recall $A_{\vartheta}(x_0)$ is defined in \eqref{ImplicitExpDecayAtheta}
 and
$A''(x_0)$  defined in \eqref{App}.
This establishes the first part of the theorem.
In the computations for Table \ref{Apptable} we chose $\s_1=0.65, \s_2=0.7, \s_3=0.8, \s_4=0.9$ in the above proof.
\\
(ii) Let $\ell \ge 0$ and $x_0 \ge 2$ then by Theorem \ref{thetathm} (ii),
\[
 \Big|\frac{\vartheta(x)-x}{x} \Big| \le \frac{\eta_{\ell-1}(x_0) }{(\log x)^{\ell-1}}
 \text{ for } x \ge x_0.
\]
Using this bound in  \eqref{keyinequalityB}
along with  Lemma \ref{E2Lemma} (ii) and
Lemma \ref{E3bd} it follows that
 \begin{equation}
  | \lambda(x) - \log \log x - M |
  \le \frac{\eta_{\ell-1}(x_0)+\delta_{\ell}(x_0)+ \kappa_{\ell}(x_0)}{(\log x)^{\ell}} \text{ for } x \ge x_0.
 \end{equation}
 Therefore in Theorem \ref{1overpthm} we have
 \begin{equation}
   \label{Aell}
 \mathcal{A}_{\ell}(x_0)=\eta_{\ell-1}(x_0)+\delta_{\ell}(x_0)+ \kappa_{\ell}(x_0)
 \end{equation}
where  $\eta_{\ell-1}(x_0)$,
$\delta_{\ell}(x_0)$,  and $\kappa_{\ell}(x_0)$
are given in \eqref{Epsithetalogbd},
\eqref{deltalx0}, and \eqref{Ckx} respectively. \\
 \end{proof}

\subsection{Proof of Proposition   \ref{explicitformulaMertens}}
To end this section, we provide proofs of Proposition \ref{explicitformulaMertens}.
For Proposition \ref{explicitformulaMertens} we require a lemma on the following integrals.
For $\Re(s)< 1$, $x \ge 2$,  and $j \in \mathbb{N}$, we define
\begin{equation}
  I_j(s,x) :=   \int_{x}^{\infty} \frac{t^{s-2}}{(\log t)^j} dt.
\end{equation}
\begin{lemma} \label{Ijsxlemma}
Let $s \in \mathbb{C}$ with $\Re(s) < 1$, $j \in \mathbb{N}$,  and $x \ge 2$. Then
\begin{itemize}
\item[(i)]
\begin{equation}
  \label{Ijbd}
  |I_j(s,x)|  \le \frac{2 x^{\Re(s)-1}}{|s-1| (\log x)^j}.
\end{equation}
\item[(ii)]
\begin{equation}
 \label{Ijparts}
I_j(s,x)
=
-\frac{x^{s-1}}{(s-1)(\log x)^j}
+
\frac{j}{s}\,I_{j+1}(s,x).
\end{equation}
\item[(iii)]
\begin{equation}
  \label{Ijinc}
I_j(s,x) = (1-s)^{j-1}\,\Gamma\!\bigl(1-j, (1-s)\log x\bigr)
\end{equation}
where
$\Gamma(s,x) $ is the incomplete gamma function \eqref{incgamma}.
\end{itemize}
\end{lemma}
\begin{proof}
(i)   We have
\begin{align*}
 I_j(s,x) & = \int_{x}^{\infty} \frac{t^{s-2}}{(\log t)^j} \,dt
 = \int_{x}^{\infty} t^{s-2}  \Big(j \int_{t}^{\infty} \frac{du}{j(\log u)^{j+1}}  \Big) \, dt  \\
 & = j \int_{x}^{\infty} \frac{1}{u(\log u)^{j+1}} \int_{x}^{u} t^{s-2} \, dt \, du \\
 & = j \int_{x}^{\infty} \frac{1}{u(\log u)^{j+1}}  \Big( \frac{u^{s-1}-x^{s-1}}{s-1} \Big) \, du.
\end{align*}
Taking absolute values implies
\[
  |I_j(s,x)| \le j \int_{x}^{\infty}  \frac{1}{u(\log u)^{j+1}}  \frac{2x^{\Re(s)-1}}{|s-1|} \, du = \frac{2x^{\Re(s)-1}}{|s-1| (\log x)^j}
\]
and the proof is complete.
Part (ii) is proven by integrating by parts.  \\
Part (iii) is proven by the change of variable $u=(1-s) \log t$ so that
 \[
  I_j(s,x) =
(1-s)^{j-1}
\int_{(1-s)\log x}^{\infty}
e^{-u}u^{-j}\,du.
 \]
 Note that, since $\Re(s) <1$, it follows that $\Re(u)= \Re( (1-s) \log t) >0$ for $t \ge x$.
\end{proof}
We now state a lemma  which provides an explicit formula for averages of $\psi(x)$.
\begin{lemma} \label{avgpsi}
Let $g$ be a continuously differentiable function on $[a,b]$ with $2 \le a < b < \infty$.
Then
\begin{equation}
  \label{explicitg}
\int_{a}^{b} (\psi(t)-t)g(t) \, dt =-\sum_{\rho}  \int_{a}^{b} \frac{t^{\rho}}{\rho} g(t) \, dt
{\color{blue} -}\log(2 \pi) \int_{a}^{b} g(t) \, dt -\frac{1}{2} \int_{a}^{b} \log(1-t^{-2}) g(t) \, dt.
\end{equation}
\end{lemma}
This is Lemma 4 of \cite[p.15]{RamSau}. Note there is a sign mistake in their statement of the lemma. They have $+\log(2 \pi)$ whereas it should be $-\log(2 \pi)$ as stated here.  Their proof has the incorrect sign for $\log(2 \pi)$ in the explicit formula for $\psi(x)$ (see equation (5) on p.14).  The correct equation can be found in Davenport \cite[equation (9), p.109]{Davenport} and note that $\frac{\zeta'(0)}{\zeta(0)} = \log (2 \pi)$.
With the above lemmas in hand we may now prove Proposition \ref{explicitformulaMertens}.

\begin{proof}[Proof of Proposition \ref{explicitformulaMertens} ]
(i)
We shall demonstrate that
\begin{equation}
\begin{split}
  \label{firstexplicit}
 \int_x^{\infty} \frac{\log t +1}{t^2 (\log  t)^2} (\psi(t)-t) dt
& =   \frac{\log x +1}{(\log  x)^2}   \sum_{\rho} \frac{x^{\rho-1}}{\rho(\rho-1)}
- \sum_{\rho} \frac{1}{\rho(\rho-1)}
(I_2(\rho,x) +2 I_3(\rho,x)) \\
&  {\color{blue} -} \frac{ \log 2 \pi}{x \log x} -\frac{1}{2}\kappa(x).
\end{split}
\end{equation}
By \eqref{Ijinc}
\[
  I_2(\rho,x) =(1-\rho)\Gamma(-1, (1-\rho) \log x)
    \text{ and } I_3(\rho,x) =(1-\rho)^2 \Gamma(-2, (1-\rho) \log x)
\]
and inserting this in \eqref{firstexplicit} establishes \eqref{integratedexplicit}.
We now demonstrate \eqref{firstexplicit} to complete the proof.
We apply Lemma \ref{avgpsi} with  $g(t) = \frac{\log t +1}{t^2 (\log  t)^2}$
and thus
\begin{align*}
\label{bnd:2}
 & \int_x^{\infty} \frac{\log t +1}{t^2 (\log  t)^2} (\psi(t)-t) dt \\
 & = - \lim_{Y \to \infty} \sum_{\rho} \int_x^Y \frac{t^{\rho-2}}{\rho} \frac{\log t +1}{(\log  t)^2}dt
   {\color{blue} -} \log2 \pi \int_x^{\infty} \frac{\log t +1}{t^2 (\log  t)^2}dt  - \frac12 \int_x^{\infty} \log(1-t^{-2}) \frac{\log t +1}{t^2 (\log  t)^2} dt
 \\
  & = - \lim_{Y \to \infty} \sum_{\rho} \int_x^Y \frac{t^{\rho-2}}{\rho} \frac{\log t +1}{(\log  t)^2}dt
   {\color{blue} -}\frac{ \log2 \pi}{x \log x} - \frac12
  \Bigg( \frac{\log(1-x^{-2})}{x\log x}
-
2\int_x^\infty
\frac{dt}{t^2(t^2-1)\log t} \Bigg).
  \end{align*}
  The second last integral was computed exactly and in the final integral we integrated by parts.  Note that in each of these integrals we have used
  $\frac{\log t+1}{t^2(\log t)^2}
=
-\frac{d}{dt}\left(\frac{1}{t\log t}\right)$.
Integrating by parts we find that
\begin{equation}
\begin{split}
\label{firstterm}
 & \int_x^Y \frac{t^{\rho-2}}{\rho} \frac{\log t +1}{(\log  t)^2}dt
  =  \frac{Y^{\rho-1}}{\rho(\rho-1)}  \Big( \frac{ \log Y +1}{(\log  Y)^2} \Big) - \frac{x^{\rho-1}}{\rho(\rho-1)}  \Big( \frac{\log x +1}{(\log  x)^2} \Big) + \int_x^Y \Big( \frac{t^{\rho-2}}{\rho(\rho-1)} \Big)  \Big( \frac{\log t +2}{(\log  t)^3} \Big) dt.
\end{split}
\end{equation}
Letting $Y \to \infty$ and using uniform and absolute convergence of the series we obtain
\begin{align*}
 \int_x^{\infty} \frac{\log t +1}{t^2 (\log  t)^2} (\psi(t)-t) dt
& =   \frac{\log x +1}{(\log  x)^2}   \sum_{\rho} \frac{x^{\rho-1}}{\rho(\rho-1)}
- \sum_{\rho} \frac{1}{\rho(\rho-1)} \int_{x}^{\infty}  t^{\rho-2} \Big( \frac{1}{(\log t)^2} + \frac{2}{(\log t)^3} \Big)dt \\
&  {\color{blue} -} \frac{ \log 2 \pi}{x \log x} -\frac{1}{2} \kappa(x)
\end{align*}
which equals \eqref{firstexplicit}. \\
(ii) We now derive the inequality \eqref{integral1bd}.
Throughout this proof we shall frequently use the inequalities $|\rho| \ge |\gamma|$ and $|\rho-1| \ge |\gamma|$ for any non-trivial zero $\rho$.
First, observe that
\begin{equation}
 \Big|  \sum_{\rho} \frac{x^{\rho-1}}{\rho(\rho-1)}  \Big|
 \le  \sum_{\rho} \frac{x^{\beta-1}}{|\rho(\rho-1)|} \le  \sum_{\rho} \frac{x^{\beta-1}}{|\gamma|^2}
 = J_2(x).
\end{equation}
From \eqref{Ijparts} it follows that
\begin{align*}
     I_2(\rho,x) +2 I_3(\rho,x) & =
     -\frac{x^{\rho-1}}{(\rho-1)(\log x)^2}
+
\frac{2}{\rho-1}\,I_{3}(\rho,x)+2 I_3(\rho,x) \\
& =  -\frac{x^{\rho-1}}{(\rho-1)(\log x)^2}+2 I_{3}(\rho,x) \frac{\rho}{\rho-1}.
\end{align*}
By taking absolute values and using \eqref{Ijbd} we find
\begin{align*}
 | I_2(\rho,x) +2 I_3(\rho,x)| \le \frac{x^{\beta-1}}{|\gamma| (\log x)^2}
 +4 \frac{|\rho|x^{\beta-1}}{|\rho-1|^2(\log x)^3}
 \le
 \frac{x^{\beta-1}}{|\gamma| (\log x)^2}
 +4 \frac{|\rho|x^{\beta-1}}{|\gamma|^2(\log x)^3}
\end{align*}
and thus
\begin{align*}
\Big| \sum_{\rho} \frac{1}{\rho(\rho-1)}
(I_2(\rho-1,x) +2 I_3(\rho-1,x)) \Big|
& \le \frac{1}{(\log x)^2}  \sum_{\rho} \frac{x^{\beta-1}}{|\gamma|^3}
+\frac{4}{(\log x)^3} \sum_{\rho} \frac{x^{\beta-1}}{|\rho-1| |\gamma|^2} \\
& \le  \Big( \frac{1}{(\log x)^2} + \frac{4}{(\log x)^3} \Big) J_2(x).
\end{align*}
Now we estimate $\frac{ {\color{blue} -} \log 2 \pi}{x \log x} -\frac{1}{2} \kappa(x)$.
Since $x \ge 2$, it follows that for $t \ge x$, $t^2 -1 \ge \frac{1}{2} t^2$ and thus
\[
   \Big| 2\int_x^\infty
\frac{dt}{t^2(t^2-1)\log t}  \Big| \le \frac{4}{(\log x)} \int_{x}^{\infty} t^{-4} \, dt = \frac{4}{3x^3 (\log x)}.
\]
We see that
\begin{align*}
  {\color{blue} -}
  \frac{  \log 2 \pi}{x \log x} -\frac{1}{2}\kappa(x)
  & =    {\color{blue} -} \frac{ \log 2 \pi}{x \log x} -\frac{1}{2} \Big(\frac{\log(1-x^{-2})}{x\log x}
  + O^{*} \Big( \frac{4}{3x^3 (\log x)} \Big)  \Big)  \\
  & = \frac{1}{x \log x}
  \Big(  {\color{blue} -}
   \log(2 \pi) -\tfrac{1}{2} \log(1-x^{-2})  +O^{*} \Big(\frac{4}{3x^2} \Big)
  \Big) \\
  & = O^{*} \Big( \frac{2.18}{x \log x} \Big),
\end{align*}
for $x \ge 2$,  where the final bound follows by calculus.
\end{proof}

\section{ Bounding weighted zero sums} \label{section:zerosums}

In this section, we derive various forms of bounds for the zero sums $J_1(x)$ and $J_2(x)$.
The methods here generalize those found in other places, for instance \cite[Lemma 2.4, p. 1346]{FaKa} and \cite{FKS}.

\begin{lemma}\label{lem:zerohalflow}Let $m>0$.
Then
\begin{equation}
\sum_{ |\gamma| < H} \frac{ x^{-1/2} }{|\gamma|^{m+1}}
\le  2x^{-1/2} \sum_{\gamma >0 } \frac{ 1 }{ \gamma^{m+1}}
\le  a_m x^{-1/2} ,
\end{equation}
where
\begin{equation}\label{def-am}
a_m=   \begin{cases}
0.04620999  & \text{ if } m = 1, \\
0.0014591 & \text{ if } m = 2.
  \end{cases}
\end{equation}
\end{lemma}
 \begin{proof}
 We have that
 \begin{equation}
   \label{sumsplit}
   \sum_{\gamma} \frac{1}{\gamma^{m+1}} = \sum_{0 < \gamma \le X_0}  \frac{1}{\gamma^{m+1}} +  \sum_{\gamma > X_0}  \frac{1}{\gamma^{m+1}}
 \end{equation}
 with $X_0 = 10^6$.  By a direct calculation, we find that
 \begin{equation}
   \label{zerosums}
    \sum_{0 < \gamma \le X_0}  \frac{1}{\gamma^{2}}= 0.0231029276 \ldots
    \text{ and }
     \sum_{0 < \gamma \le X_0}  \frac{1}{\gamma^{3}}= 0.0007295482 \ldots .
 \end{equation}
From \cite[(2.23),(2.24)]{FaKa} we have  for $X \ge 10$
\begin{equation}
  \label{s3B3}
 \sum_{\gamma > X} \frac{ 1 }{ \gamma^{m+1}} \le {\bf B}_3(m,X) :=  \Big( \frac{1}{2\pi} + q(X) \Big) \frac{1+m\log (X/2\pi)}{m^2X^m} + \frac{2r(X)}{X^{m+1}}
\end{equation}
where $r(X) :=  \kappa_1 \log X + \kappa_2 \log \log X + \kappa_3$,   $q(X) := \frac{\kappa_1 \log X + \kappa_2}{X \log X \log (X/2\pi)}$,
$\kappa_1= 0.1370, \kappa_2=0.4430$, and $\kappa_3 = 2.4630$.  Using this, we obtain
\begin{equation}
\label{B3values}
    {\bf B}_3(1,X_0) = 2.0654546524465851730 \cdot 10^{-6}
    \text{ and }
      {\bf B}_3(2,X_0) = 9.9293858789533279100 \cdot 10^{-13}.
\end{equation}
Combining \eqref{sumsplit},\eqref{zerosums}, and \eqref{B3values} leads to our values for $a_1, a_2$ and establishes the lemma.
\end{proof}

\begin{lemma}\label{lem:zerohalfhigh}
Let $m>0$ and $\sigma>0$. Then
\[
\sum_{\substack{ |\gamma| \ge H \\ 0<\beta < \sigma}}   \frac{x^{\beta-1} }{|\gamma|^{m+1}}
\le
2 x^{\sigma-1} \sum_{\substack{ \gamma \ge H }}   \frac{1}{\gamma^{m+1}}
= 2 b_m x^{\sigma-1},
\]
where $b_m={\bf B}_3(m,H)$.
Furthermore, using  $H = 3 \, 000 \, 175 \, 332 \, 800 $, we have
\begin{equation}
  \label{s3bm}
  b_m \leq  \begin{cases}
1.479621 \cdot 10^{-12} & \text{ if } m=1, \\
2.421686 \cdot 10^{-25} & \text{ if } m=2.
   \end{cases}
\end{equation}
\end{lemma}
\begin{proof}
The proof is a direct computation of ${\bf B}_3(m,H)$ for $m=1,2$.
\end{proof}
\begin{lemma}
\label{B4Bound}
Let $m\ge 0$, $2\le  U < V $, and $\s>1/2$.
Assume there exists $c_1(\s),c_2(\s),p(\sigma)$, and $q(\sigma)$ for which the zero-density bound \hyperref[ZDB]{(ZDB)} is satisfied for $N(\s,T) $.
Then,
\begin{equation}\label{def-RS}
 \sum_{\substack{U \le \gamma < V \\ \s \le \beta < 1}} \frac{1}{\gamma^{m+1}}
\le  {\bf B}_4(m,\s, U, V),
\text{ and }
 \sum_{\substack{U \le \gamma  \\ \s \le \beta < 1}} \frac{1}{\gamma^{m+1}} \le
{\bf B}_4(m,\s, U, \infty)
\end{equation}
where ${\bf B}_4(m,\s, U, V)$ and ${\bf B}_4(m,\s, U, \infty)$ are given by
\begin{equation}\label{def-B4}
\begin{split}
{\bf B}_4(m,\s, U, V)=
&c_1  \frac{ (\log V )^{q }}{V^{m+1-p }} + c_2  \frac{(\log V)^2}{V^{m+1}}
\\&+ \frac{c_1(m+1)}{(m+1-p)^{q+1} } \left( \Gamma(q+1,(m+1-p)\log U)  - \Gamma(q+1,(m+1-p)\log V)\right)
\\&+ \frac{c_2(m+1)}{(m+1)^3} \left(  \Gamma(3, (m+1)\log U)- \Gamma(3, (m+1)\log V)\right),
\end{split}
\end{equation}
and
\begin{equation}
 \label{B4infty}
{\bf B}_4(m,\s, U, \infty) :=
\frac{c_1(m+1)}{(m+1-p)^{q+1} } \Gamma(q+1,(m+1-p)\log U) + \frac{c_2(m+1)}{(m+1)^{3} } \Gamma(3,(m+1)\log U).
\end{equation}
In these definitions,  we write $c_1,c_2, p$, and $q$ respectively in place of $c_1(\sigma),c_2(\sigma),p(\sigma)$, and $q(\sigma)$.
\end{lemma}
Recall, the incomplete gamma function is defined in \eqref{incgamma}.
Explicit bounds for incomplete gamma functions may be found in \cite{Pin}.  We need a bound on  $ {\bf B}_4(m,\sigma_0,t_0^{4(m+1)},\infty)$. Note that we have $\Gamma(a, (m+1-b)\log T) \sim ( (m+1-b)\log T)^a T^{- (m+1-b)}$.

\begin{proof}
We begin by assuming that $U, V$ are not ordinates of the zeros of $\zeta(s)$.
We rewrite the sum using a Riemann-Stieltjes integral, then integrate by parts, and use the bound
$N(\s,y) \le \tilde{N}(\s,y)= c_1y^p (\log y)^q + c_2 (\log y)^2$ so that
\begin{equation}
\begin{split}
 \label{suminequality}
\sum_{\substack{U \le \gamma < V \\ \s \le \beta < 1}} \frac{1}{\gamma^{m+1}}
& = \int_{U}^{V} \frac{ dN(\sigma,y) }{y^{m+1}}
 = \frac{N(\s,V)}{V^{m+1}} -  \frac{N(\s,U)}{U^{m+1}} + (m+1) \int_{U}^{V}\frac{N(\s,y)}{y^{m+2}} dy \\
&  \le  \frac{\tilde N(\s,V)}{V^{m+1}} +(m+1) \int_{U}^{V}\frac{\tilde N(\s,y)}{y^{m+2}} dy
\\& =  c_1  \frac{ (\log V )^{q }}{V^{m+1-p }} + c_2  \frac{(\log V)^2}{V^{m+1}}
+ c_1(m+1)  \int_{U}^{V}  \frac{ (\log y )^{q } }{y^{m+2-p}} dy
+ c_2(m+1)  \int_{U}^{V} \frac{(\log y)^2 }{ y^{m+2}} dy \\
& \le c_1  \frac{ (\log V )^{q }}{V^{m+1-p }} + c_2  \frac{(\log V)^2}{V^{m+1}} \\
&+ (m+1) \Big(
 c_1 (J_{q,m+2-p}(U)-J_{q,m+2-p}(V)) + c_2 ( J_{2,m+2}(U)-J_{2,m+2}(V))
 \Big)
\end{split}
\end{equation}
where for  $a>0,b>1$
$$J_{a,b}(T)= \int_{T}^{\infty} \frac{(\log y)^a}{y^b} dy $$
Observe that by the change of variable $y =e^t$, we find that
\begin{equation}\label{def-Jab}
 J_{a,b}(T)=\int_{T}^{\infty} \frac{(\log y)^a}{y^b} dy =\frac{1}{(b-1)^{a+1} } \Gamma(a+1,(b-1)\log T)
\end{equation}
where $\Gamma(s,x)$ is the incomplete gamma function.
Combining this identity with \eqref{suminequality} establishes the first bound in \eqref{def-RS}.
We obtain the second bound in \eqref{def-RS} by letting $V\to\infty$
\begin{align*}
\sum_{\substack{U \le \gamma  \\ \s \le \beta < 1}} \frac{1}{\gamma^{m+1}}
& \le  \frac{c_1 (m+1)}{(m+1-p)^{q+1} } \Gamma(q+1,(m+1-p)\log T) + \frac{c_2(m+1)}{(m+1)^{3} } \Gamma(3,(m+1)\log T). \qedhere
\end{align*}
The general case when either of $U$ or $V$ equals an ordinate of a zero of $\zeta(s)$ can be obtained from this first case by a continuity argument.
\end{proof}
In the next lemma we provide initial bounds for the zero-counting sum $\sum_{\substack{ H \le |\gamma| < T \\ \beta \ge \sigma_0}} \frac{ x^{\beta-1} }{|\gamma|^{m+1}}$.
It is convenient to define the function
\begin{equation}
 \label{def-eps4B}
  \varepsilon_4(x,m,\s,(x_k)_{k=0}^{K})
= 2 \sum_{k=1}^{K-1}  \frac{x^{-\frac{1}{R\log (x_{k}) }}}{x_k^{m+1}} \Big(  \tilde{N}(\s, x_{k+1}) - \tilde{N}(\s, x_{k}) \Big)   + 2  \tilde{N}(\s, x_{1})\frac{x^{-\frac{1}{R(\log x_{0})}}}{x_0^{m+1}}
 \end{equation}
 where  $(x_k)_{k=0}^{K}$ an arbitrary strictly increasing sequence such that $x_K = T$.

\begin{proposition}
\label{prop:eps4bound}
Let $\sigma_0\ge 1/2$,  $u_0 =u_0(x)= {\rm max} \big( H, \exp\big( \sqrt{\frac{ \log x}{R(m+1)}} \big) \big)$, $t_0= t_0(x)=\exp\big( \sqrt{\frac{ \log x}{R(m+1)}}  \big)$,  and $T> t_0$.
 Then we have
\[
\sum_{\substack{ H \le |\gamma| < T \\ \beta \ge \sigma_0}} \frac{ x^{\beta-1} }{|\gamma|^{m+1}}  \le   2 \tilde N(\sigma_0,T)  \frac{x^{-\frac{1}{R\log u_0 }}}{u_0^{m+1}}
\le
2 \tilde N(\sigma_0,T)  \frac{x^{-\frac{1}{R\log t_0 }}}{t_0^{m+1}}.
\]
Further, suppose that  $(u_k)_{k=0}^{K}$ and $ (t_k)_{k=0}^{K}$ are strictly increasing sequences with  $u_0,t_0$ as above and $u_K=T$ and $t_K=T$
where $K \ge 1$, then we have
\begin{equation}\label{def-eps4}
\sum_{\substack{ H \le |\gamma| < T \\ \beta \ge \sigma_0}} \frac{ x^{\beta-1} }{|\gamma|^{m+1}}
 \le \varepsilon_4(x,m,\s_0,(u_k)_{k=0}^{K})
 \text{ and }
  \sum_{\substack{ H \le |\gamma| < T \\ \beta \ge \sigma_0}} \frac{ x^{\beta-1} }{|\gamma|^{m+1}}
 \le \varepsilon_4(x,m,\s_0,(t_k)_{k=0}^{K})
\end{equation}
where $ \varepsilon_4(x,m,\s,(x_k)_{k=0}^{K}) $ is defined in \eqref{def-eps4B}.

\end{proposition}
\begin{proof}
We generalize the proof of \cite[Proposition 3.9]{FKS} by replacing $\frac1t$ by $\frac1{t^{m+1}}$.
By the zero-free region \eqref{ZFR} with $R$ value given in \eqref{Rnumeric}
\[
   \sum_{\substack{ H \le |\gamma| < T \\ \beta \ge \sigma_0}} \frac{ x^{\beta-1} }{|\gamma|^{m+1}}  \le
     \sum_{\substack{ H \le |\gamma| < T \\ \beta \ge \sigma_0}} \frac{ x^{-\frac{1}{R \log \gamma}} }{|\gamma|^{m+1}}.
\]
Classical Riemann-Stieltjes integration followed by an integration by parts give:
\begin{equation}\label{ineq1}
 \sum_{\substack{ H \le |\gamma| < T \\ \beta \ge \sigma_0}} \frac{ x^{\beta-1} }{|\gamma|^{m+1}}  \le
2 \int_{H}^T  \frac{x^{-\frac{1}{R\log t }}}{t^{m+1}} dN(\sigma_0,t)
= 2\Big( \frac{x^{-\frac{1}{R\log T }}}{T^{m+1}}  N(\sigma_0,T) - \int_{H}^T N(\sigma_0,t) \frac{d}{dt}\Big( \frac{x^{-\frac{1}{R\log t}}}{t^{m+1}}
 \Big) \, dt \Big).
 \end{equation}
Note that
\[  \frac{d}{dt}\Big( \frac{x^{-\frac{1}{R\log t }}}{t^{m+1}} \Big) = \frac{x^{-\frac{1}{R\log t }}}{t^{m+2}} \Big(\frac{\log x}{R(\log t)^2} -(m+1) \Big)  \]
has a critical point at $t_0 = \exp\big( \sqrt{\frac{ \log x}{R(m+1)}} \big) $. This
becomes negative as soon as $t> t_0$.
Thus taking $u_0={\rm max} \big( H,t_0 \big) $, we have the bound
\begin{align*}
 - \int_{H}^T N(\sigma_0,t) \frac{d}{dt}\Big( \frac{x^{-\frac{1}{R\log t}}}{t^{m+1}} \Big)  \, dt
 &\le - N(\sigma_0,T)   \int_{u_0}^T \frac{d}{dt}\Big( \frac{x^{-\frac{1}{R\log t }}}{t^{m+1}} \Big) \, dt \\
 &=  N(\sigma_0,T)  \Big( \frac{x^{-\frac{1}{R\log u_0 }}}{u_0^{m+1}} - \frac{x^{-\frac{1}{R\log T }}}{T^{m+1}}\Big) \\
&\le  {\tilde N}(\sigma_0,T)  \frac{x^{-\frac{1}{R\log u_0 }}}{u_0^{m+1}}
\end{align*}
where we applied the inequality \eqref{bnd-ZDB}.
Using this last bound in \eqref{ineq1}  proves the first part of our proposition.
We refine this bound by splitting along the imaginary axis at the heights $u_k$, taking $u_0$ as above and $u_K = T$.
We can use the bound $N(\sigma_0,t) \le N(\sigma_0,u_{k+1})$ on each piece, so that
\begin{align*}
- \int_{H}^T N(\sigma_0,t) \frac{d}{dt}\Big( \frac{x^{-\frac{1}{R\log t}}}{t^{m+1}} \Big)  \, dt
& \le - \sum_{k=0}^{K-1} \int_{u_k}^{u_{k+1}}  N(\sigma_0,t) \frac{d}{dt}\Big( \frac{x^{-\frac{1}{R\log t}}}{t^{m+1}} \Big)  \, dt \\
& \le \sum_{k=0}^{K-1}
N(\sigma_0,u_{k+1}) \Big( \frac{x^{-\frac{1}{R\log u_k}}}{u_k^{m+1}} - \frac{x^{-\frac{1}{R\log u_{k+1}}}}{u_{k+1}^{m+1}} \Big)
\end{align*}
We note that the $(K-1)^{th}$ term is $N(\sigma_0,T) \Big( \frac{x^{-\frac{1}{R\log u_{K-1}}}}{u_{K-1}^{m+1}} - \frac{x^{-\frac{1}{R\log T}}}{T^{m+1}} \Big)  $, so that when combining the above with \eqref{ineq1}, and shifting the sum, we obtain
\[
\sum_{\substack{ H \le |\gamma| < T \\ \beta \ge \sigma_0}} \frac{ x^{\beta-1} }{|\gamma|^{m+1}}
\le   2 \Big( \sum_{k = 1}^{K-1} \tilde{N}(\sigma_0, u_k)\Big(  \frac{x^{-\frac{1}{R\log(u_{k-1})}}}{u_{k-1}^{m+1}} - \frac{x^{-\frac{1}{R\log u_{k} }}}{u_k^{m+1}} \Big)+ \frac{x^{-\frac{1}{R\log(u_{K-1})}}}{u_{K-1}^{m+1}}\tilde{N}(\sigma_0, T) \Big).
\]
By  ``telescoping" this last sum we are able to deduce the bound in \eqref{def-eps4}.
Observe that by following the same argument as above, establishes the proof for the sequence
$(t_k)_{k=0}^{K}$.
\end{proof}

\begin{proposition}\label{epsilon4-dec-simple}
Fix $K\ge 2$ and $c>1$, and set $t_0, T$, and $\sigma_1$ as functions of $x$ defined by
\begin{equation}\label{def-t0-T-sigma2}
t_0 = t_0(x)=\exp\Big(\sqrt{\frac{\log x}{R(m+1)}}\Big),\
T=T(x)=t_0^c,\ \text{and }\
\sigma_1 = \sigma_1(x)= 1-\frac{2(m+1)}{R\log t_0(x)} .
\end{equation}
We now make the specific choices of $t_k$ given by
 \begin{equation}
  \label{tkwk}
  t_k=  t_0^{w_k} \text{ where }
  w_k = 1 + \frac{k}{K}(c-1)
  \text{ for }
  0 \le k \le K.
 \end{equation}
Let
\begin{equation}
 \label{def-eps4B2}
  \varepsilon_4(x,m,\s_1(x),(t_k)_{k=0}^{K})
= 2 \sum_{k=1}^{K-1}  \frac{x^{-\frac{1}{R\log (t_{k}) }}}{t_k^{m+1}} \Big(  \tilde{N}(\s_1, t_{k+1}) - \tilde{N}(\s_1, t_{k}) \Big)   + 2  \tilde{N}(\s_1, t_{1})\frac{x^{-\frac{1}{R(\log t_{0})}}}{t_0^{m+1}}
 \end{equation}
where, by \eqref{def-Ntilde} we have
 \begin{equation}
   \label{tilNsig2}
 \tilde{N}(\sigma,x) = c_1 x^p (\log x)^q + c_2 (\log x)^2
\end{equation}
where we shall assume in \eqref{tilNsig2}
 \begin{equation}
   \label{pqchoice}
  p=p(\sigma) = \frac{8}{3}(1-\sigma) \text{ and }
  q=q(\sigma) = 5 -2 \sigma.
 \end{equation}
Then, as $x\to \infty$,
\begin{equation}\label{asymp-eps4}
  \varepsilon_4(x,m,\s_1(x),(t_k)_{k=0}^{K})
=(1+o(1)) C \frac{(\log t_0)^{3+\frac{4(m+1)}{R\log t_0}}}{t_0^{2(m+1)}}, \ \text{with}\
  C = 2 c_1   e^{\frac{16w_{1}}{3 R} } w_{1}^{3},
  \ \text{and}\ w_1 =1 + \frac{c-1}{K}.
\end{equation}
Moreover, both
$\varepsilon_4(x,m,\s_1(x),(t_k)_{k=0}^{K}) $
and
$   \frac{\varepsilon_4(x,m,\s_1(x),(t_k)_{k=0}^{K}) t_0^{2(m+1)}}{(\log t_0)^3} $
are decreasing in $x$ for $x>\exp(R(m+1)e^2)$.
\end{proposition}

\begin{proof}
Denoting $w_k = 1 + \frac{k}{K}(c-1)$, with $w_0=1$, we have the identities
\begin{align} \label{id1} \log x &= (m+1)R(\log t_0)^2, \\
\label{id2}
\log t_0 &= \sqrt{\frac{\log x}{(m+1)R}}, \\
\label{id3}
\log t_k  &= (\log t_0) \Big(1 + \frac{k}{K}(c-1)\Big)
= (\log t_0) w_k.
\end{align}
It follows from \eqref{id1} and \eqref{id3} that
\[
 \frac{\log x}{R \log t_k} = \frac{(m+1) R (\log t_0)^2}{R w_k (\log t_0)} = \frac{(m+1) (\log t_0)}{w_k}
\]
and thus
\begin{align*}
 \frac{x^{-\frac{1}{R\log t_{k} }} }{t_k^{(m+1)}}
= \frac{\exp \Big( \frac{(m+1) (\log t_0)}{w_k} \Big)}{ (t_0)^{w_k(m+1)}}
= t_{0}^{-(m+1)(w_k+1/w_k)}
= t_{0}^{-(m+1)(2 + \frac{ (w_k-1)^2}{w_k})}.
\end{align*}
In the case  $k=0$, we have
$\frac{x^{-\frac{1}{R(\log t_{0})}} }{t_0^{(m+1)}}   =  t_0^{-2(m+1)}$.
This allows us to rewrite \eqref{def-eps4B} as
\begin{equation}
\begin{split}
\label{eps4-1}
& \varepsilon_4(x,m,\s_1(x),(t_k)_{k=0}^{K})  \\
 & =
\frac{2}{t_0^{2(m+1)}}\sum_{k=1}^{K-1} t_0^{ - (m+1)\frac{(w_k-1)^2}{w_k} }
\Big(  \tilde{N}(\s_1, t_{k+1}) - \tilde{N}(\s_1, t_{k}) \Big)
+\frac{2}{t_0^{2(m+1)}} \tilde{N}(\s_1, t_{1}).
\end{split}
\end{equation}
Now by  \eqref{tilNsig2} we have
\begin{equation}
\begin{split}
  &  \tilde{N}(\s_1, t_{k+1}) - \tilde{N}(\s_1, t_{k})  \\
   &  = c_1 ( t_{k+1}^{p(\s_1)} (\log t_{k+1})^{q(\s_1)}
   -  t_{k}^{p(\s_1)} (\log t_{k})^{q(\s_1)} )
    + c_2 ( (\log t_{k+1})^2 - (\log t_k)^2)  \\
    & = (\log t_0)^{q(\s_1)} ( t_0^{w_{k+1} p(\s_1)} w_{k+1}^{q(\s_1)} -t_0^{w_k p(\s_1)} w_k^{q(\s_1)})
    + c_2(\log t_0)^2 (w_{k+1}^2-w_{k}^2) \\
\end{split}
\end{equation}
where we have used \eqref{id3} in the last line.  Next note that by \eqref{def-t0-T-sigma2} and \eqref{pqchoice}, we have
\[
 p(\s_1) = \frac{16(m+1)}{3R \log t_0} \text{ and }
 q(\s_1) = 3+ \frac{4(m+1)}{R \log t_0}
\]
and thus it follows that
\begin{equation}
\begin{split}
 \label{diffN}
  &  \tilde{N}(\s_1, t_{k+1}) - \tilde{N}(\s_1, t_{k})  \\
&= c_1 (\log t_0)^{3+\frac{4(m+1)}{R\log t_0}}  \Big( e^{\frac{16(m+1)w_{k+1}}{3 R} } w_{k+1}^{3+\frac{4(m+1)}{R\log t_0}} -  e^{\frac{16(m+1)w_k}{3 R } }  w_k ^{3+\frac{4(m+1)}{R\log t_0}}\Big)  + c_2(\log t_0)^2 (w_{k+1}^2-w_{k}^2).
\end{split}
\end{equation}
Thus \eqref{diffN} can be rewritten as
\begin{equation}\label{diffN1}
\tilde{N}(\s_1, t_{k+1}) - \tilde{N}(\s_1, t_{k})
=
C_{1,k} (\log t_0)^{3+\frac{4(m+1)}{R\log t_0}}
+ C_{2,k}  (\log t_0)^{2},
\end{equation}
with
\begin{equation}
C_{1,k} = c_1   \Big( e^{\frac{16(m+1)w_{k+1}}{3 R} } w_{k+1}^{3+\frac{4(m+1)}{R\log t_0}} -  e^{\frac{16(m+1)w_k}{3 R } }  w_k ^{3+\frac{4(m+1)}{R\log t_0}}\Big) ,\ \text{and}\
C_{2,k}= c_2 \big(w_{k+1}^2 - w_{k}^2\big).
\end{equation}
In addition, we have
\[
\tilde{N}(\s_1, t_{1})
=
c_1e^{\frac{16(m+1)w_{1} }{3 R} }w_1^{3+\frac{4(m+1)}{R\log t_0}}  (\log t_0)^{3+\frac{4(m+1)}{R\log t_0}}
+ c_2 w_{1}^2 (\log  t_0)^2 .
\]
Together with \eqref{diffN1}, we deduce that \eqref{eps4-1} can be simplified to
\begin{equation}\label{eps4-2}
\begin{split}
 \varepsilon_4(x,m,\s_1(x),(t_k)_{k=0}^{K})
=
 \frac{2(\log t_0)^3}{t_0^{2(m+1)}} \Bigg(\sum_{k=1}^{K-1} t_0^{ - (m+1)\frac{(w_k-1)^2}{w_k} }
\Big( C_{1,k} (\log t_0)^{\frac{4(m+1)}{R\log t_0}}
+ \frac{C_{2,k}}{  (\log t_0) } \Big)  \\
\qquad+ \left(c_1 (\log t_0)^{ \frac{4(m+1)}{R\log t_0}} e^{\frac{16(m+1)w_{1}}{3 R} } w_{1}^{3+\frac{4(m+1)}{R\log t_0}}
+ \frac{ c_2w_{1}^2 }{  (\log t_0) }
\right)\Bigg).
\end{split}
\end{equation}
Since
\[ t_0^{ -m \frac{(w_k-1)^2}{w_k} } ,\, (\log t_0)^{\frac{4m}{R \log t_0}} ,\, \frac1{\log t_0},\text{ and } w_{1}^{3+\frac{4m}{R\log t_0}}
\]
 all decrease with $t_0>e^e$, then
$\frac{\varepsilon_4(x,m,\s_1(x),(t_k)_{k=0}^{K})  t_0^{2(m+1)}}{(\log t_0)^3}$ and as well $ \varepsilon_4(x,m,\s_1(x),(t_k)_{k=0}^{K})  $ both decrease with $t_0$.
By our definition $t_0 = \exp\Big(\sqrt{\frac{\log x}{R(m+1)}}\Big)$, it follows that these each of these expressions decrease
for $x$ satisfying
\[
 x > 1, \, x > e^{R(m+1)e^2}, x > 1, \, x > 1.
\]
We deduce from \eqref{eps4-2} that
\[
\frac{\varepsilon_4(x,m,\s_1(x),(t_k)_{k=0}^{K})  t_0^{2(m+1)}}{(\log t_0)^{3+\frac{4}{R\log t_0}}}=
 2   \sum_{k=1}^{K-1} t_0^{ - m\frac{(w_k-1)^2}{w_k} }
\Big( C_{1,k}
+\frac{ C_{2,k} }{ (\log t_0)^{1+\frac{4}{R\log t_0}} }\Big)
 +2\Big(c_1   e^{\frac{16w_{1}}{3 R} } w_{1}^{3+\frac{4}{R\log t_0}}
+  \frac{c_2w_{1}^2 }{ (\log t_0)^{1+\frac{4}{R\log t_0}} }\Big).
\]
Letting $x\to \infty$, this gives the asymptotic announced in \eqref{asymp-eps4}.
\end{proof}

\begin{proposition}\label{prop:zerospast58}
Let $x_0 \ge 2$, $m\ge 0$, $\sigma_0>5/8$, and $K \ge 2$.  If
\[
     x \ge  x_0 \ge \exp \Big( \frac{4(m+1)^3}{R(1-\s_0)^2}\Big),
\]
then
\[  \sum_{ \beta> \sigma_0}     \frac{x^{\beta-1}}{|\gamma|^{m+1}}    \le     C_m(x_0,\s_0) (\log x)^{3/2} e^{-2\sqrt{(m+1)/R}  \sqrt{\log x} } \]
for some effectively computable constant $C_m(x_0,\s_0)$ given by

\begin{equation}
\begin{split}
  \label{Cmx0}
   C_m(x_0,\s_0)  & =  G_m(x_0,K)  \\
   & + \Big( {\bf B}_4(m,\sigma_0,H,\infty) t_0^{-2(m+1)^2}
  +  {\bf B}_4(m,\sigma_0,t_0^{4(m+1)},\infty) \Big)  ( \log x_0)^{-3/2} e^{2\sqrt{(m+1)/R}  \sqrt{(\log x_0)} }
  \end{split}
\end{equation}
where
\begin{equation}
 \eqref{Gmx0Kb}
G_m(x_0,K) :=  \frac{\varepsilon_4(x_0,m,\s_1(x_0),K,t_0(x_0)^{4(m+1)}) (t_0(x_0))^{2(m+1)}}{(\log t_0(x_0))^3}.
\end{equation}
\end{proposition}
Below is a table of values of $C_m(x_0,\sigma_0)$ with $K=10000$. Let $t(m,\s_0)=  \frac{4(m+1)^3}{R(1-\s_0)^2} $.

\begin{center}
\begin{tabular}{|c|c|c|c|c|}
\hline
$m$ & $\sigma$ & $\log t(m,\sigma_0)$ & $\log x_0$ & $C_m(x_0,\sigma_0)$ \\
\hline
1 & 0.65 & $46.926 \ldots$   & 50 & 0.0014464 \\
1 & 0.7 & $63.872\ldots$  & 70 & $1.4882\cdot 10^{-6}$ \\
1 & 0.8 &  $143.71\ldots$ & 150 & $1.6694\cdot 10^{-11}$ \\
1 & 0.9 &  $574.85\ldots$  & 600 & $5.0004\cdot 10^{-14}$ \\
2 & 0.65 &  $158.37\ldots$  & 160 & $9.4815\cdot 10^{-19}$ \\
2 & 0.7 &  $215.57\ldots$  & 220 & $1.5719\cdot 10^{-20}$ \\
2 & 0.8 &  $485.03\ldots$  & 500 & $7.0882\cdot 10^{-24}$ \\
2 & 0.9 &  $194013.23\ldots$  & 200000 & $6.1156\cdot 10^{-188}$ \\
\hline
\end{tabular}
\end{center}

\begin{proof}
We write
\[ \sigma_1=\sigma_1(x) = 1-\frac{2(m+1)}{R\log t_0},\quad t_0 = t_0(x)=\exp\Big(\sqrt{\frac{\log x}{R(m+1)}}\Big),\quad t_K=T(x)=t_0(x)^{4(m+1)}.
\]
and we choose $t_k$ as in \eqref{tkwk} of Proposition \ref{epsilon4-dec-simple}. That is,
\begin{equation}
  \label{tkx}
t_k(x)=  t_0(x)^{w_k}
\end{equation}
 where $w_k = 1 + \frac{k}{K}(c-1)$ for  $0 \le k \le K$ and $c=4(m+1)$.
Note that $\sigma_0 > \frac{5}{8}$ and assume that $\sigma_1(x) > \sigma_0$.  This condition is equivalent to
$x \ge \exp \Big( \frac{4(m+1)^3}{R(1-\s_0)^2}\Big)$.
We now decompose the zero sum as follows
\[  \sum_{ \beta> \sigma_0}   \frac{x^{\beta-1}}{|\gamma|^{(m+1)}}   =
 \sum_{\substack{ H \le |\gamma| < T \\ \sigma_1 \ge \beta \ge \sigma_0}}    \frac{x^{\beta-1}}{|\gamma|^{(m+1)}}
 + \sum_{\substack{ H \le |\gamma| < T \\ \beta \ge \sigma_1}}    \frac{x^{\beta-1}}{|\gamma|^{(m+1)}}
 + \sum_{\substack{  |\gamma| \ge T \\ \beta > \sigma_0}}    \frac{x^{\beta-1}}{|\gamma|^{(m+1)}}
 \]
 where we have used the fact that there are no zeros $\rho=\beta+i\gamma$ with $\beta > \sigma_0$ and $|\gamma| < H$
 by \cite[Theorem 1.1]{PlaTruH2020} (see \eqref{H}, \eqref{Hnumeric}).
 We now bound  the last three sums. First,
 \[ \sum_{\substack{ H \le |\gamma| < T \\ \sigma_1 \ge \beta \ge \sigma_0}}    \frac{x^{\beta-1}}{|\gamma|^{(m+1)}}   \le  \Bigg(\sum_{\substack{ H \le |\gamma|  \\ \beta \ge \sigma_0}}    \frac{1}{|\gamma|^{m+1}}\Bigg)  x^{\frac{-2(m+1)}{R\log t_0}}  \le  {\bf B}_4(m,\sigma_0,H,\infty) t_0^{-2(m+1)^2},  \]
 where we have invoked Proposition \ref{B4Bound} and we have used $ x^{-\frac{2(m+1)}{R\log t_0}}  =  t_0^{-2(m+1)^2}$.
Next, by combining Propositions \ref{prop:eps4bound} and \ref{epsilon4-dec-simple}, we have
\[  \sum_{\substack{ H \le |\gamma| < T \\ \beta \ge \sigma_1}}    \frac{x^{\beta-1}}{|\gamma|^{(m+1)}}
\le \varepsilon_4(x,m,\s_1(x),(t_k)_{k=0}^{K})
  \le \frac{\varepsilon_4(x,m,\s_1,K,T) t_0^{2(m+1)}}{(\log t_0)^3}  \frac{(\log t_0)^3}{t_0^{2(m+1)}} \]
where we recall that $\varepsilon_4(x,m,\s_1(x),(t_k)_{k=0}^{K})$  is defined in  \eqref{def-eps4B}.  From the second part of Proposition  \ref{epsilon4-dec-simple},
we know that $ \frac{\varepsilon_4(x,m,\s_1,K,T) t_0^{2(m+1)}}{(\log t_0)^3}$ is decreasing in $x$.
Thus we have
\begin{equation}
  \label{Gmx0Kb}
     \frac{\varepsilon_4(x,m,\s_1,K,T) t_0^{2(m+1)}}{(\log t_0)^3} \le G_m(x_0,K) :=  \frac{\varepsilon_4(x_0,m,\s_1,K,T') (t_0')^{2(m+1)}}{(\log t_0')^3}
\end{equation}
where  $t_k' := t_k(x_0)$ for $0 \le k \le K$ and $t_k(x)$ is defined in \eqref{tkx}.
In paricular, we have
\begin{equation}
  \label{t0prime}
  t_0' =  t_0(x_0)= \exp\Big(\sqrt{\frac{\log x_0}{R(m+1)}}\Big) \text{ and } T'= (t_0')^{4(m+1)}
\end{equation}
Thus we may evaluate this expression at $x=x_0$ to obtain:
\[
  \sum_{\substack{ H \le |\gamma| < T \\ \beta \ge \sigma_1}}    \frac{x^{\beta-1}}{|\gamma|^{(m+1)}}  \le
  G_m(x_0,K)  (\log t_0)^3t_0^{-2(m+1)}.
\]
We note that this will be the main term in the upper bound.
Finally, we apply again Proposition \ref{B4Bound} to obtain
\[ \sum_{\substack{  |\gamma| \ge T \\ \beta \ge \sigma_0}}    \frac{x^{\beta-1}}{|\gamma|^{(m+1)}}  \le  \sum_{\substack{  |\gamma| \ge T \\ \beta \ge \sigma_0}}    \frac{1}{|\gamma|^{m+1}}  \le   {\bf B}_4(m,\sigma_0,t_0^{4(m+1)},\infty).  \]
Combining all three terms it follows that
\begin{align*}
 \sum_{ \beta> \sigma_0}     \frac{x^{\beta-1}}{|\gamma|^{1+m}}   & \le
  {\bf B}_4(m,\sigma_0,H,\infty) t_0^{-2(m+1)^2} +   G_m(x_0,K)  (\log t_0)^3t_0^{-2(m+1)}
  +  {\bf B}_4(m,\sigma_0,t_0^{4(m+1)},\infty).
\end{align*}
Next, observe that
$(\log t_0)^3t_0^{-2(m+1)} = (\log x)^{3/2} e^{-2\sqrt{(m+1)/R}  \sqrt{\log x } }$.
By factoring out this expression, we obtain for $x \ge x_0$
\[
   \sum_{ \beta> \sigma_0}     \frac{x^{\beta-1}}{|\gamma|^{1+m}} \le  C_m(x_0,\s_0) (\log x)^{3/2} e^{-2\sqrt{(m+1)/R}  \sqrt{\log x } }
\]
where $C_m(x_0,\s_0)$ is defined in \ref{Cmx0}. Note that in computations we take $K=10000$.
\end{proof}
We also provide the alternate bound for $J_m(x)$.
\begin{lemma}
Let $x_0 \ge  2$. Let $\frac{5}{8} \le \s_0 < \s_1$ and $H \le T$.
Let $u_0 = {\rm max} \big( H, \exp\big( \sqrt{\frac{ \log x_0}{R(m+1)}} \big) \big)$, $u_K=T$, and
 $(u_k)_{k=0}^{K}$ is a strictly increasing sequence given.
\label{Jmbound}
We have the bound
\begin{equation}
\begin{split}
  \label{Jmxinterval}
 J_m(x)
 & \le a_m x_0^{-1/2}   +2b_m x_0^{\sigma_0-1}
 + x_0^{\s_1-1} {\bf B}_4(m,\sigma_0,H, \infty)
 +  \varepsilon_4(x,m,\s_1,(u_k)_{k=0}^{K})
 +   {\bf B}_4(m,\sigma_0,T,\infty) \text{ for } x \ge x_0.
\end{split}
\end{equation}

\end{lemma}
\begin{proof}
For $\s_0 < \s_1$ and $H \le T$, we have the following decomposition:
\begin{equation}
\begin{split}
  \label{Jmxineq}
J_m(x) & =\sum_{ \beta}     \frac{x^{\beta-1}}{|\gamma|^{m+1}}   \\
& = \sum_{|\gamma|<H }     \frac{x^{\beta-1}}{|\gamma|^{m+1}}
+  \sum_{ \substack{\beta \leq \sigma_0\\ |\gamma|>H} }     \frac{x^{\beta-1}}{|\gamma|^{m+1}}
 + \sum_{ \substack{ \s_0 < \beta \le \sigma_1\\  H \le |\gamma| \le T} }     \frac{x^{\beta-1}}{|\gamma|^{m+1}}
 +  \sum_{ \substack{\beta > \sigma_1\\  H \le |\gamma| \le T} }     \frac{x^{\beta-1}}{|\gamma|^{m+1}}
 +  \sum_{ \substack{\beta > \sigma_0\\   |\gamma| \ge T} }     \frac{x^{\beta-1}}{|\gamma|^{m+1}}  \\
 & \le  a_m x_0^{-1/2}   +2b_m x_0^{\sigma_0-1}
 + x_0^{\s_1-1} {\bf B}_4(m,\sigma_0,H, \infty)
 +   \sum_{ \substack{\beta > \sigma_1\\  H \le |\gamma| \le T} }     \frac{x^{\beta-1}}{|\gamma|^{m+1}}
 +  \sum_{ \substack{\beta > \sigma_0\\   |\gamma| \ge T} }     \frac{x^{\beta-1}}{|\gamma|^{m+1}}
 \end{split}
 \end{equation}
 where Lemmas  \ref{lem:zerohalflow},   \ref{lem:zerohalfhigh}, and \ref{B4Bound}
 have been applied to the first three terms respectively.
 For the fourth term we have
 \begin{equation}
    \sum_{ \substack{\beta > \sigma_1\\  H \le |\gamma| \le T} }     \frac{x^{\beta-1}}{|\gamma|^{m+1}}
    \le  \sum_{ \substack{\beta > \sigma_1\\  H \le |\gamma| \le T} }     \frac{x_0^{\beta-1}}{|\gamma|^{m+1}}
    \le
     \varepsilon_4(x_0,m,\s_1,(u_k)_{k=0}^{K})
    ,
 \end{equation}
 where $u_0 = {\rm max} \big( H, \exp\big( \sqrt{\frac{ \log c}{R(m+1)}} \big) \big)$, $u_K=T$, and  we have applied equation \eqref{def-eps4} of Proposition \ref{prop:eps4bound}.
 For the final term, we have by Lemma
 \ref{B4Bound}
\[ \sum_{\substack{  |\gamma| \ge T \\ \beta \ge \sigma_0}}    \frac{x^{\beta-1}}{|\gamma|^{(m+1)}}  \le  \sum_{\substack{  |\gamma| \ge T \\ \beta \ge \sigma_0}}    \frac{1}{|\gamma|^{m+1}}  \le   {\bf B}_4(m,\sigma_0,T,\infty).  \]
 Inserting these last two bounds in \eqref{Jmxineq} completes the proof of \eqref{Jmxinterval}.

\end{proof}
We now are able to deduce  \cref{JmLemma}.
\begin{proof}[Proof of Theorem \ref{JmLemma}]
(i)  Fix $\sigma_0 > 5/8$.  Recall that  $t(m,\s_0) = \frac{4(m+1)^3}{R(1-\s_0)^2}$ (see \eqref{tmsig}).
Note that $\s_0 \le \s_1(x)$ if and only if $x \ge \exp(t(m,\s_0))$.
We write
\[
J_m(x)=
\sum_{ \beta}     \frac{x^{\beta-1}}{|\gamma|^{m+1}}
= \sum_{|\gamma|<H }     \frac{x^{\beta-1}}{|\gamma|^{m+1}}
+  \sum_{ \substack{\beta \leq \sigma_0\\ |\gamma|>H} }     \frac{x^{\beta-1}}{|\gamma|^{m+1}}
 +  \sum_{ \substack{\beta > \sigma_0\\  |\gamma|>H} }     \frac{x^{\beta-1}}{|\gamma|^{m+1}} .  \]
Combining Lemmas \ref{lem:zerohalflow}, \ref{lem:zerohalfhigh}, and Proposition \ref{prop:zerospast58} we obtain for $\s_1(x) \ge \s_0 > \frac{5}{8}$ and $x \ge x_0$
\begin{align*}
 J_m(x)
  & \le    a_m x^{-1/2}   +2b_m x^{\sigma_0-1} +    C_m(x_0,\s_0) (\log x)^{3/2} e^{-2\sqrt{(m+1)/R}  \sqrt{\log x} } \\
 & \le A_m(x_0,\s_0)  (\log x)^{3/2} e^{-2\sqrt{(m+1)/R}  \sqrt{\log x} }
 \end{align*}
 where
\begin{equation}
   \label{Amx}
 A_m(x_0,\s_0) = C_m(x_0,\s_0)
   +
   a_m g_{\tfrac{1}{2},-\tfrac{3}{2},2 \sqrt{\tfrac{m+1}{R}}}(x_0)
   +b_m g_{1-\sigma_0,-\tfrac{3}{2},2 \sqrt{\tfrac{m+1}{R}}(x_0) }
  \end{equation}
 where $g_{a,b,c}(x)=x^{-a}(\log x)^b \exp(c\sqrt{\log x})$ (see \eqref{gabcx}). We shall use this bound for $x_0 \ge e^{t(m,\s_0)}$.  \\
  (ii)
By \eqref{Jmxinterval} we have
\begin{equation}
  \label{JmxIbound}
  J_m(x) < \varepsilon(c) :=
 a_m c^{-1/2}   +2b_m c^{\sigma_0-1}
 + c^{\s_1-1} {\bf B}_4(m,\sigma_0,H, \infty)
 +  \varepsilon_4(c,m,\s_1,K,T )
 +   {\bf B}_4(m,\sigma_0,T,\infty)   \text{ for all } x  \ge c.
\end{equation}
Now choose $x_1 > x_0$ and split up the interval $[x_0, x_1]$ into many subintervals, say
$[x_0,x_1]= \sum_{j=1}^{N} I_j$
\footnote{In the numerical computations we chose the parameters $\sigma_0=0.9,\sigma_1=0.99$
and we chose the intervals to be of the form $[x_0 e^j, x_0e^{j+1}]$ with $0 \le j \le J$. }
 with
\begin{equation}
 x_0 =a_0 < b_0 = a_1 < b_1 < \cdots b_{N-1}=a_N < b_N=x_1.
\end{equation}
On each subinterval $I_j=[a_j,b_j]$, we determine by \eqref{JmxIbound} a constant
$\varepsilon_j= \varepsilon_j(I_j)$
so that
\begin{equation}
J_m(x) < \varepsilon_j  \text{ for } x \in I_j
\end{equation}
and this implies
\begin{equation}
  J_m(x) <  \frac{\varepsilon_j (\log b_j)^{\ell}}{(\log x)^{\ell}} \text{ for } x \in I_j.
\end{equation}
It follows that for $\ell \ge 0$,
\begin{equation}
  J_m(x) < \Big(  \max_{j=1, \ldots, N}  \varepsilon_j (\log b_j)^{\ell} \Big) (\log x)^{-\ell} \text{ for } x \in [x_0,x_1].
\end{equation}
Now by the first part of the lemma, namely \eqref{Jmbd1}, we have for $x \ge x_1$
\[
  J_m(x) \le
 \frac{A_m(x_1,\s_0)}{(\log x)^{\ell}}
  g_{0,\frac{3}{2}+\ell,-2 \sqrt{\frac{m+1}{R}}}(x)
  \le \frac{A_m(x_1,\s_0) \max_{x \ge x_1}( g_{0,\frac{3}{2}+\ell,-2 \sqrt{\frac{m+1}{R}}} (x_1)) }{(\log x)^{\ell}}.
\]
Therefore
\begin{equation}
 J_m(x) \le  \frac{ \widetilde{\delta}_{m,\ell}(x_0)}{(\log x)^{\ell}}  \text{ for } x \ge x_0
\end{equation}
where
\begin{equation}
  \label{tildedeltamlx0}
    \widetilde{\delta}_{m,\ell}(x_0) = \max \Big( \max_{j=1, \ldots, N}  \varepsilon_j (\log b_j)^{\ell} , A_m(x_1.\s_0)
    \max_{x \ge x_1}( g_{0,\tfrac{3}{2}+\ell,-2 \sqrt{\tfrac{m+1}{R}}} (x_1))\Big). \qedhere
\end{equation}
\end{proof}
We now prove Lemma \ref{E2Lemma}, which is a simple consequence of Lemma \ref{JmLemma}.
\begin{proof}[Proof of Lemma \ref{E2Lemma}]
(i)
Note that since $\mathcal{E}_2(x) =  \frac{\log x +1}{ (\log  x)^2}  J_1(x) + \frac{\log x +4}{(\log  x)^3}  J_2(x)$
it follows from Lemma \ref{JmLemma} that for $x \ge x_0$
\begin{align*}
   \mathcal{E}_2(x) & \le \Big(\tfrac{1}{\log x} +\tfrac{1}{(\log x)^2} \Big)
   A_1(x_0,\s_0) (\log x)^{B}
 \exp(-C \sqrt{2} \sqrt{\log x}) \\
   & + \Big( \tfrac{1}{(\log x)^2}+ \tfrac{4}{(\log x)^3}\Big)  A_2(x_0,\s_0) (\log x)^{B}
 \exp(-C \sqrt{3} \sqrt{\log x})
 \end{align*}
 for $x \ge x_0$
and thus
\begin{align*}
     \mathcal{E}_2(x) & \le (\log x)^{B-1}  \exp(-C \sqrt{2} \sqrt{\log x}) \\
     & \times
     \Big(  A_1(x_0,\s_0)
     (1+ \tfrac{1}{\log x_0}) + A_2(x_0,\s_0)   (1+ \tfrac{4}{\log x_0}) (\log x)^{-1}
      \exp((C \sqrt{2}-C \sqrt{3}) \sqrt{\log x})
     \Big) \\
     & \le A'(x_0,\s_0)(\log x)^{B-1}  \exp(-C \sqrt{2} \sqrt{\log x})
 \end{align*}
where
\begin{equation}
  \label{Ap}
 A'(x_0,\s_0) =  A_1(x_0,\s_0)
     (1+ \tfrac{1}{\log x_0}) + A_2(x_0,\s_0)   (1+ \tfrac{4}{\log x_0}) (\log x_0)^{-1}
      \exp((C(\sqrt{2}-\sqrt{3})) \sqrt{\log x_0}).
\end{equation}
(ii) Let $\ell \ge 1$ and $x_0 \ge 2$.   Then we have by \eqref{Jmbd2} that $ J_m(x) \le \frac{ \widetilde{\delta}_{m,\ell-1}(x_0)}{(\log x)^{\ell-1}} $.  It follows from \eqref{E2}
$\mathcal{E}_2(x) \le \frac{\delta_{\ell}(x_0)}{(\log x)^{\ell}}$  for  $x \ge x_0$, where
\begin{equation}
  \label{deltalx0}
  \delta_{\ell}(x_0)=
  \widetilde{\delta}_{1,\ell-1}(x_0) \Big( 1 + \tfrac{1}{\log x_0} \Big)
 +\widetilde{\delta}_{2,\ell-1}(x_0)  \Big( \tfrac{1}{\log x_0} + \tfrac{4}{(\log x_0)^2} \Big).
 \qedhere
\end{equation}
\end{proof}

\section{Proof of Theorem \ref{bnd:sum_logp}
and
\ref{bnd:sum_lambdan}
}. \label{section:othersums}

In this section, we will study the Mertens sums in \eqref{Mertens2} and \eqref{Mertens3}.  The strategy for
bounding them is similar to that of $\lambda(x)$. The following result is an analogue of \eqref{keyinequalityB}.
\begin{proposition} \label{logpoverpfirst}
(i) (Explicit formula) For $x \ge 2$,
\begin{equation}
  \label{Upsexplicit}
  \Upsilon(x) = \log x -M' + \frac{\vartheta(x)-x}{x} - \sum_{\rho} \frac{x^{\rho}}{\rho (\rho-1)} + \int_{x}^{\infty} \frac{\psi(t)-\vartheta(t)}{t^2} \, dt -\frac{B(x)}{x}
\end{equation}
where
\begin{equation}
  \label{B}
 B(x) =   -\frac{x}{2} \log \Big( \frac{x+1}{x-1} \Big) - \frac12 \log \Big( 1- \frac{1}{x^2} \Big) - \log 2\pi +1.
 \footnote{
    In \cite[p. 13]{Van}, there are typos in the formula for $B(x)$. The $\frac12$ coefficient for $\log ( 1- x^{-2} )$ is missing and also the sign on $\log(2 \pi)$ is incorrect. However,
    the stated bound $|B(x)| \le 1.884$ in that article still holds.
}
\end{equation}
(ii) (Explicit inequality)
For $x \ge 2$, we have
\begin{equation}\label{bnd:sum_logp2}
|
\Upsilon(x) -\log x + M' | \le \frac{|\vartheta(x) - x|}{x} + J_1(x)   +
\eta_1' r_1(x) + \eta_2' r_2(x) +  \log(2 \pi) x^{-1}
 \end{equation}
 where $\eta_1', \eta_2'$ are constants defined in \eqref{etajconstants} and $r_1(x),r_2(x)$ are functions
 defined in \eqref{r1r2}. \\
 (iii) (Exponential decay) Let $x_0 \ge 2$ and $\s_0 > \frac{5}{8}$. There exists a positive constant $D'=D'(x_0,\s_0)$ such that if
 $x \ge x_0 \ge  \exp \Big( \frac{32}{R(1-\s_0)^2}\Big)$, then
 \begin{equation}\label{sum_logp:ImplicitExpDecay}
  | \Upsilon(x) - \log x + M' | \le \frac{|\vartheta(x) - x|}{x} +
D'(x_0,\s_0) (\log x)^{B}
 \exp(-C \sqrt{2} \sqrt{\log x})
 \end{equation}
 where $D'(x_0,\s_0)$ is given by \eqref{Dp}.
\end{proposition}

\begin{proof}
(i)
By \cite[Equation (4.21)]{RS62}, we have
\begin{equation}
  \label{Upsps}
 \Upsilon(x) = \log x - M' + \frac{\vartheta(x) - x}{x} - \int_x^{\infty} \frac{\vartheta(t) - t}{t^2}dt.
\end{equation}
Subtracting and adding
 $\psi(t)$ we have
\begin{equation}
  \label{intsplit}
 - \int_x^{\infty} \frac{\vartheta(t) - t}{y^2}dt =
 - \int_x^{\infty} \frac{\psi(t) - t}{t^2}dt+
 \int_x^{\infty} \frac{\psi(t) - \vartheta(t)}{t^2}dt.
 \end{equation}
By an application of \eqref{explicitg} \cite[Lemma 4]{RamSau}  with $g(t) = t^{-2}$ the first integral is
\begin{equation}
  \label{explicitg0}
 \int_x^{\infty} \frac{\psi(t)-t}{t^2}dt =
\sum_{\rho} \frac{x^{\rho-1}}{\rho (\rho-1)}
 +\widetilde{B}(x)
 \end{equation}
where
\begin{equation}
\widetilde{B}(x) = - \log(2 \pi) \int_{x}^{\infty} t^{-2} \, dt
-\frac12
\int_x^\infty \frac{\log(1-t^{-2})}{t^2}\,dt.
\end{equation}
A direct calculation gives  $\widetilde{B}(x) = B(x)/x$ where
\begin{equation}
\label{tildeB}
 \widetilde{B}(x) = \frac{B(x)}{x}
 \text{ where }
 B(x) = -\frac{x}{2} \log \Big( \frac{x+1}{x-1} \Big) - \frac12 \log \Big( 1- \frac{1}{x^2} \Big) - \log 2\pi +1.
 \end{equation}
Therefore, combining \eqref{Upsps}, \eqref{intsplit}, \eqref{explicitg0}, and \eqref{tildeB} we obtain
\eqref{Upsexplicit}. \\
(ii)
By Proposition \ref{psi-theta:ExplicitCor}, for $x \ge 2$
\begin{equation}
\begin{split}
  \label{firstintbd}
 \int_x^{\infty} \frac{\psi(t)-\vartheta(t)}{t^2} dt &\le \int_x^{\infty} \frac{ \eta_1' (t^\frac12+ t^\frac13 + t^\frac15)
 + \eta_2' (t^\frac14+t^\frac16 +t^\frac{1}{10}  )  }{t^2} dt  \\
    & = \eta_1' ( 2x^{-\frac{1}{2}}+ \tfrac{3}{2} x^{-\tfrac{2}{3}}+ \tfrac{5}{4} x^{-\frac{4}{5}})
    + \eta_2'( \tfrac{4}{3}x^{-\frac{3}{4}}+ \tfrac{6}{5} x^{-\frac{5}{6}}+ \tfrac{10}{9} x^{-\frac{9}{10}})  \\
    & = \eta_1' r_1(x) + \eta_2' r_2(x)
\end{split}
\end{equation}
where we recall $r_j(x)$ are defined in \eqref{r1r2}.
 Next, note that for $x \ge 2$,
 \begin{equation}
   \label{sumbd}
  \Big| \sum_{\rho} \frac{x^{\rho-1}}{\rho (\rho-1)}  \Big| \le \sum_{\rho} \frac{x^{\beta-1}}{|\rho(\rho-1)|} \le  \sum_{\rho} \frac{x^{\beta-1}}{|\gamma|^2}
  = J_1(x).
\end{equation}
Finally, it may be checked that $B(x) < 0$, and that it decreases on $[2,\infty)$ to
$-\log(2 \pi)$. Therefore, for $x \ge 2$
\begin{equation}
  \label{Bbd}
   |\widetilde{B}(x)| =
|B(x)| x^{-1} \le \log(2 \pi) x^{-1}.
\end{equation}
By the explicit formula \eqref{Upsexplicit}, the triangle inequality, and then
applying the inequalities \eqref{firstintbd}, \eqref{sumbd}, and \eqref{Bbd},
 we arrive at \eqref{bnd:sum_logp2}.
\\
(iii) Let $x_0 \ge 2$. By
\eqref{bnd:sum_logp2} and  an application of Lemma \eqref{JmLemma} (ii) to $J_1(x)$ it follows that there exists $A_1(x_0,\sigma_0) >0$ such that
\begin{equation}
\begin{split}
\label{logp:bndd}
  | \Upsilon(x) - \log x + M' | & \le \frac{|\vartheta(x) - x|}{x} +  J_1(x)   +
  \eta_1' r_1(x) + \eta_2' r_2(x) +   \log(2 \pi) x^{-1}\\
  & \le  \frac{|\vartheta(x) - x|}{x} +  A_1(x_0,\sigma_0) (\log x)^{B}
 \exp(-C \sqrt{2} \sqrt{\log x})+
 \eta_1' r_1(x) + \eta_2' r_2(x) +  \log(2 \pi) x^{-1}.
\end{split}
\end{equation}
Observe that
\begin{equation}
  \label{E3epansion}
    \eta_1' r_1(x) + \eta_2' r_2(x) +  \log(2 \pi) x^{-1} = \sum_{j \in J} \alpha_j x^{-j}
\end{equation}
where
\begin{equation}
 \label{Jalphaj}
J =  \left\{\frac12,\frac23, \frac45,\frac34, \frac56, \frac{9}{10}, 1\right\}
\text{ and }
 \{ \alpha_j \}_{j \in J} =   \left\{
2 \eta_1' , \tfrac{3}{2}  \eta_1' , \tfrac{5}{4}  \eta_1',  \tfrac{4}{3} \eta_2' , \tfrac{6}{5} \eta_2',  \tfrac{10}{9} \eta_2',
\log(2 \pi)
  \right\}.
 \end{equation}
By factoring $(\log x)^{B}
 \exp(-C \sqrt{2} \sqrt{\log x})$ it follows that
\begin{align*}
 | \Upsilon(x) - \log x + M' |  & \le  \frac{|\vartheta(x) - x|}{x}+
(\log x)^{B}
 \exp(-C \sqrt{2} \sqrt{\log x})
 \Big(
 A_1(x_0,\s_0) + \sum_{j \in J} \alpha_j g_{j,-B,C \sqrt{2}}(x)
  \Big)
\end{align*}
where we recall the functions $g_{j,-B,C \sqrt{2}}(x)$ are defined in \eqref{gabcx}. By Lemma \ref{dec_funcs} (iii) all of these functions are decreasing
for $x \ge x_0 \ge 2$.  Thus, setting
\begin{equation}
 \label{Dpsigma}
 D'(x_0,\s_0) := A_1(x_0,\s_0) + \sum_{j \in J} \alpha_j g(j,-B,C \sqrt{2},x_0)
\end{equation}
it follows that for $x_0 \ge 2$, $\sigma_0 > \frac{5}{8}$ that
\begin{align*}
  | \Upsilon(x) - \log x + M' |  & \le  \frac{|\vartheta(x) - x|}{x}+
D'(x_0,\s_0) (\log x)^{B}
 \exp(-C \sqrt{2} \sqrt{\log x})
 \text{ for all } x \ge x_0
\end{align*}
we obtain the desired inequality \eqref{sum_logp:ImplicitExpDecay}.
\end{proof}

We will now prove Theorem \ref{bnd:sum_logp}.
\begin{proof}[Proof of Theorem  \ref{bnd:sum_logp}]
(i)  Let  $x_0  \ge 2$ and $\ell \ge 1$. By Theorem \ref{thetathm} equation
\eqref{Epsithetalogbd}
$|(\vartheta(x)-x)/x  | \le \frac{\eta_{\ell}(x_0)}{(\log x)^{\ell}}$
for $x \ge x_0$ and by Lemma \ref{JmLemma} (ii)
$J_1(x) \le \frac{ \widetilde{\delta}_{1,\ell}(x_0)}{(\log x)^{\ell}}$
for  $x \ge x_0$.  Thus, for $x \ge x_0$,
\begin{align*}
    | \Upsilon(x) - \log x + M' | & \le   \frac{ \eta_{\ell}(x_0) + \widetilde{\delta}_{1,\ell}(x_0)}{(\log x)^{\ell}}
    + \eta_1' r_1(x) + \eta_2' r_2(x) +  \log(2 \pi) x^{-1}
   \le \frac{\mathcal{B}_{\ell}(x_0)}{(\log x)^{\ell}}
\end{align*}
where
\begin{equation}
  \label{Bell}
 \mathcal{B}_{\ell}(x_0) =   \eta_{\ell}(x_0) + \widetilde{\delta}_{1,\ell}(x_0) + \sum_{j \in J} \alpha_j   \max_{x \ge x_0}( g_{j,\ell,0}(x))
\end{equation}
and we recall that $\eta_{\ell}(x_0)$ is defined by \eqref{Epsithetalogbd}, $\widetilde{\delta}_{1,\ell}(x_0)$ is given by \eqref{tildedeltamlx0},  and
$J$ and $\alpha_j$ are given in \eqref{Jalphaj}.  \\
(ii)
Let $\varphi(x) = 9.2204(\log x)^{3/2}
 \exp(- C \sqrt{\log x})$ with $C= 0.84768$ given in \eqref{ImplicitExpDecaypsitheta}.
 Note we have the classic bounds of Rosser-Schoenfeld \cite[Theorem 6, Corollary, p.70]{RS62}
\[
 | \Upsilon(x) - \log x + M' |  \le
  \begin{cases}
  \frac{1}{2 (\log x)} & \text{ for } x \ge 319, \\
  \frac{1}{(\log x)} & \text{ for } x \ge 32.
  \end{cases}
\]
Observe that $\frac{1}{2 (\log x)}< \varphi(x)$ for $319 \le x \le e^{464}$ and
 $\frac{1}{ (\log x)}< \varphi(x)$ for $2 \le x \le 319$.  It follows that
 \begin{equation}
     | \Upsilon(x) - \log x + M' |\le  9.2204(\log x)^{3/2}
 \exp(- C \sqrt{\log x}) \text{ for } 32 \le x \le e^{464}.
 \end{equation}
Thus, it suffices to establish a bound for $| \Upsilon(x) - \log x + M' | $ for $x \ge e^{464}$.
Select parameters $\sigma_{j}$  such that
 $$0.65=\sigma_1 < \cdots < \sigma_{r}=0.9$$
and define $y_i = \exp(t(1,\sigma_i))$.  Note that
$y_1 < y_2 < \cdots < y_r$ and $y_1 = \exp(46.92685996 )$.
 Now set
 \begin{equation}
  \label{Dp}
 D'(x_0)
 = \begin{cases}
  D'(x_0,\s_j) & \text{ if } x_0 \in [\sigma_j,\sigma_{j+1} ] \text{ for some } 1 \le j \le r-1 \\
  D'(x_0,\s_r) &  \text{ for }  x_0 \ge s_r. \\
 \end{cases}
\end{equation}
 It follows from that from Proposition \ref{logpoverpfirst} (ii), for $x \ge x_0 \ge e^{50}$,
  \begin{align*}
   | \Upsilon(x) - \log x + M' | &  \le  A_{\vartheta}(x_0) (\log x)^{B}
 \exp(-C\sqrt{\log x}) +D'(x_0) (\log x)^{B}
 \exp(-C \sqrt{2} \sqrt{\log x}) \\
 & \le     (\log x)^{B}
 \exp(-C\sqrt{\log x}) ( A_{\vartheta}(x_0)+ D'(x_0)  \exp((C-C \sqrt{2}) \sqrt{\log x})) \\
 & \le  A_{\Upsilon}(x_0)  (\log x)^{B}
 \exp(-C\sqrt{\log x})
 \end{align*}
 where
 \begin{equation}
   \label{Aupsilon}
    A_{\Upsilon}(x_0)   = A_{\vartheta}(x_0) +   D'(x_0)  \exp((-C(\sqrt{2}- 1)) \sqrt{\log x_0})) .
 \end{equation}
 Note that in the actual computations of $D'(x_0)$ we choose the values
 $\s_1=0.65, \s_2=0.7, \s_3=0.8, \s_4=0.9$.
\end{proof}

Next we consider the third and final sum over primes we are interested in studying.
\begin{proposition}
(i) (Explicit formula) For $x \ge 2$
\begin{equation}
  \label{tildepsiexplicit}
  \widetilde{\psi}(x) =  \log x -\gamma + \frac{\psi(x)-x}{x} - \sum_{\rho} \frac{x^{\rho}}{\rho(\rho-1)} - \frac{B(x)}{x}
\end{equation}
where $B$ is defined in \eqref{B}. \\
(ii) (Explicit inequality)
 For $x \ge 2$ we have
 \begin{equation}\label{psitilde:bnd1}
   |  \widetilde{\psi}(x) -\log x + \gamma | \le \frac{|\psi(x)-x|}{x}  + J_1(x) +  \log(2 \pi) x^{-1}.
 \end{equation}
 (iii)  (Exponential decay)
 Let $x_0 \ge 1$ and $\s_0 > \frac{5}{8}$.  There exists a positive constant $D''=D''(x_0,\s_0)$
 such that if
 $x \ge x_0 \ge  \exp \Big( \frac{32}{R(1-\s_0)^2}\Big)$, then
 \begin{equation}\label{sum_logp:ImplicitExpDecay}
   | \widetilde{\psi}(x) -\log x + \gamma | \le \frac{|\psi(x) - x|}{x} +
   D''(x_0,\s_0) (\log x)^{B}
 \exp(-C\sqrt{2} \sqrt{\log x})
 \end{equation}
 where $D''(x_0,\s_0)$ is defined in \eqref{Dppsigma}.
\end{proposition}

\begin{proof}
(i)
By partial summation
 \begin{align*}
    \widetilde{\psi}(x)
		= \frac{\psi(x)-x}{x} - \gamma + \log x - \int_x^{\infty} \frac{\psi(t)-t}{t^2} dt.
 \end{align*}
 The identity then follows from \eqref{explicitg0} and \eqref{tildeB}. \\
 (ii) The inequality \eqref{psitilde:bnd1} follows from the explicit formula
 \eqref{tildepsiexplicit}, the triangle inequality,  along with the inequalities
 \eqref{sumbd} and \eqref{Bbd} already established in Proposition
 \eqref{logpoverpfirst}. \\
 (iii)  Let $\s_0 > \frac{5}{8}$.
 By Lemma \ref{JmLemma} (ii),  it follows that for $x \ge x_0$
 \begin{equation}
 \begin{split}
    \label{J1bd}
  J_1(x) +  \log(2 \pi) x^{-1} & \le A_1(x_0, \s_0) (\log x)^{B} \exp(-C \sqrt{2} \sqrt{\log x}) + \log(2 \pi) x^{-1} \\
  & =    (\log x)^{B} \exp(-C \sqrt{2} \sqrt{\log x}) \Big( A_1(x_0, \s_0) +   \log(2 \pi)
  g_{1,-B,C \sqrt{2}}(x)   \Big)  \\
  & \le    D''(x_0,\s_0) (\log x)^{B} \exp(-C \sqrt{2} \sqrt{\log x})
  \end{split}
  \end{equation}
  where
  \begin{equation}
  \label{Dppsigma}
  D''(x_0,\s_0) = A_1(x_0,\s_0)  +   \log(2 \pi)   g_{1,-B,C \sqrt{2}}(x_0).
 \end{equation}
 In the last inequality, we have made use of Lemma \ref{dec_funcs} (iii),
 which asserts that $g_{1,-B,C \sqrt{2}}(x)$ decreases for $x \ge x_0 \ge 2$.  This establishes \eqref{sum_logp:ImplicitExpDecay}.
\end{proof}

We now prove Theorem \ref{bnd:sum_lambdan}.
\begin{proof}[Proof of Theorem \ref{bnd:sum_lambdan}]
(i) The proof is very similar to  the proof of Theorem  \ref{bnd:sum_logp}] (i).  Let $x_0 \ge 2$ and $\ell \ge 1$, then
recall that \eqref{Epsithetalogbd} gives   $|\psi(x)-x|x^{-1} \le \frac{\widetilde{\eta_{\ell}}(x_0)}{(\log x)^{\ell}}$ for $x \ge x_0$
and thus
\begin{align*}
  |  \widetilde{\psi}(x) -\log x + \gamma | & \le
     \frac{\widetilde{\eta_{\ell}}(x_0)+\widetilde{\delta}_{1,\ell}(x_0)}{(\log x)^{\ell}}+ \log(2 \pi)x^{-1} \\
     & = \frac{1}{(\log x)^{\ell}} (\widetilde{\eta_{\ell}}(x_0)+\widetilde{\delta}_{1,\ell}(x_0) + \log(2 \pi)
     g_{1,\ell,0}(x) )
     \le \frac{\mathcal{C}_{\ell}(x_0)}{(\log x)^{\ell}}
\end{align*}
where
\begin{equation}
 \label{Cell}
    \mathcal{C}_{\ell}(x_0) =\widetilde{\eta_{\ell}}(x_0)+\widetilde{\delta}_{1,\ell}(x_0) +
     \log(2 \pi) \max_{x \ge x_0}(g_{1,\ell,0}(x)).
\end{equation}
Note that by
 by Lemma \ref{dec_funcs} (ii), in the case $x_0 > e^{\ell}$,
 we have $ \max_{x \ge x_0}(g_{1,\ell,0}(x))= g_{1,\ell,0}(x_0)$. \\
(ii)  Observe that for $x \ge 2$
\[
  \widetilde{\psi}(x) -\Upsilon(x) = \sum_{\substack{p^k \le x \\ p \text{ prime }, k \ge 2}} \frac{\log p}{p^k}
  \le  C_0 := \sum_{\substack{p^k  \\ p \text{ prime }, k \ge 2}} \frac{\log p}{p^k} < 0.75537.
\]
This implies
\[
   |  \widetilde{\psi}(x) -\log x + \gamma | \le  | \Upsilon(x) - \log x + M' | +C_0
   \le
  \begin{cases}
  \frac{1}{2 (\log x)}  +C_0 & \text{ for } x \ge 319, \\
  \frac{1}{(\log x)} + C_0 & \text{ for } x \ge 32.
  \end{cases}
\]
Observe that $\frac{1}{2 (\log x)} +C_0< \varphi(x)$ for $319 \le x \le e^{115}$ and
 $\frac{1}{ (\log x)} + C_0< \varphi(x)$ for $32 \le x \le 319$.  It follows that
 \begin{equation}
    |  \widetilde{\psi}(x) -\log x + \gamma | \le  9.2204(\log x)^{3/2}
 \exp(- C \sqrt{\log x}) \text{ for } 2 \le x \le e^{115}.
 \end{equation}
Thus, it suffices to establish a bound for $ |  \widetilde{\psi}(x) -\log x + \gamma | $ for $x \ge e^{115}$.
The proof is very similar to the proof of Theorem  \ref{bnd:sum_logp}] (ii).
Select parameters $\sigma_{j}$  such that
 $$0.65=\sigma_1 < \cdots < \sigma_{r}=0.9$$
and define $y_i = \exp(t(1,\sigma_i))$.  Note that
$y_1 < y_2 < \cdots < y_r$ and $y_1 = \exp(46.92685996 )$.
 Now set
 \begin{equation}
  \label{Dpp}
 D''(x_0)
 = \begin{cases}
  D''(x_0,\s_j) & \text{ if } x_0 \in [y_j,y_{j+1} ] \text{ for some } 1 \le j \le r-1, \\
  D''(x_0,\s_r) &  \text{ for }  x_0 \ge y_r. \\
 \end{cases}
\end{equation}
Following the same steps we have for $x \ge x_0 \ge e^{50}$,
\begin{align*}
    |  \widetilde{\psi}(x) -\log x + \gamma |  &
    \le  A_{\psi}(x_0) (\log x)^{B}
 \exp(-C\sqrt{\log x}) +D''(x_0) (\log x)^{B}
 \exp(-C \sqrt{2} \sqrt{\log x}) \\
 & \le     (\log x)^{B}
 \exp(-C\sqrt{\log x}) ( A_{\psi}(x_0)+ D''(x_0)  \exp((C-C \sqrt{2}) \sqrt{\log x})) \\
 & \le  A_{\widetilde{\psi}}(x_0)   (\log x)^{B}
 \exp(-C\sqrt{\log x}),
\end{align*}
for $x \ge x_0$
where
\begin{equation}
 \label{Awidetildepsi}
   A_{\widetilde{\psi}}(x_0)  = A_{\psi}(x_0) +   D''(x_0)  \exp(-(C(\sqrt{2}- 1)) \sqrt{\log x_0}))
\end{equation}
since the function within the brackets is decreasing.
Note that in the actual computations of $D''(x_0)$ we choose the values
 $\s_1=0.65, \s_2=0.7, \s_3=0.8, \s_4=0.9$.
\end{proof}

\section{Proof of Theorems \ref{prod1} and \ref{prod2}: Products of primes}\label{sec:ProdPrimes}  \label{products}

In this section, we give improved bounds for the Mertens products
based on our improved bounds for $   \lambda(x)$.
\begin{lemma}
Let
\begin{equation}
 \label{S}
S =  \sum_{p > x} \Big( \log \Big(1 - \frac{1}{p} \Big) + \frac{1}{p} \Big).
\end{equation}
   If $\vartheta(x) < Zx$ for all $x >0$, then
\begin{equation}
   \label{Sbound}
   -\frac{Z}{(x-1) \log x} < S < 0.
\end{equation}
In particular, since $\vartheta(x) < Z_1 x$ (see \cite[Corollary 2.1]{BKLNW})
with $Z_1 = 1+1.93378 \cdot 10^{-8}$, we have
\begin{equation}
   \label{Sboundb}
   -\frac{Z_1}{(x-1) \log x} < S < 0.
\end{equation}
\end{lemma}
\begin{proof} From the Maclaurin series of $\log$, we have
\begin{equation}
 S = - \sum_{n=2}^{\infty} \frac{1}{n} \sum_{x < p} \frac{1}{p^n}
\end{equation}
and it follows that $S <0$.
By partial summation,
\begin{equation}
   \sum_{x < p} \frac{1}{p^n} = - \frac{\vartheta(x)}{x^n \log x} + \int_{x}^{\infty} \frac{ \vartheta(y)}{y^{n+1}}
    \frac{1 + n \log y}{(\log y)^2} \, dy
    \le Z \int_{x}^{\infty} \frac{1}{y^{n}}
    \frac{1 + n \log y}{(\log y)^2} \, dy,
\end{equation}
where we dropped the first term since $\vartheta(x) \ge 0$ and used the bound $\vartheta(x) < Zx$ for all $x >0$.
Lemma 9 of  \cite{RS62} with $a=1$, gives a bound for the last integral, so that
$\sum_{x < p} \frac{1}{p^n} \le  \frac{Zn}{n-1}  \frac{x^{1-n}}{\log x}$.  It follows that
\begin{equation}
  -S \le \sum_{n=2}^{\infty} \frac{1}{n}  \cdot   \frac{Zn}{n-1}  \frac{x^{1-n}}{\log x}
  < \frac{Z}{\log x} \sum_{n=1}^{\infty} x^{1-n} = \frac{Z}{(x-1) \log x}
\end{equation}
and the result is established.
\end{proof}

\begin{theorem} \label{firstprodthm}
 Let $\ell \in \{1,2,3,4,5\}$ and $x_0 \ge 2$. \\
 (i) For $x \ge x_0$, we have
 \begin{equation}\label{bnd:prod1}
  \frac{e^{-\gamma}}{\log x} \exp \Big( -\frac{\mathcal{A}_{\ell}(x_0)}{\log ^{\ell} x} \Big) \le \prod_{p \le x} \Big( 1 - \frac{1}{p} \Big) \le \frac{e^{-\gamma}}{\log x} \exp \Big( \frac{\mathcal{A}_{\ell}(x_0)}{\log ^{\ell} x} + \frac{Z_1}{(x-1)\log x} \Big)
 \end{equation}
where $\mathcal{A}_{\ell}(x_0)$ is defined in Theorem \ref{1overpthm}. \\
(ii)    For $x \ge x_0$, we have
 \begin{equation}
  \label{bnd:prod2}
 \frac{e^{-\gamma}}{\log x} \Big( 1 - \frac{\mathcal{A}_{\ell}(x_0)}{\log ^{\ell} x} \Big) \le \prod_{p \le x} \Big( 1 - \frac{1}{p} \Big) \le \frac{e^{-\gamma}}{\log x} \Big( 1 + \frac{\mathcal{A}'_{\ell}(x_0)}{\log ^{\ell} x} \Big)
 \end{equation}
where
\begin{equation}\label{defn:Atildem}
 \mathcal{A}'_{\ell}(x_0) =  \Big( \mathcal{A}_{\ell} +
 Z_1 \max_{x \ge x_0} \frac{(\log x)^{{\ell}-1}}{x-1}
 \Big)  \Big(\frac{e^{m}-1}{m}\Big)
 \text{ and }
 m = \frac{\mathcal{A}_{\ell}(x_0)}{\log ^{\ell} x_0} + \frac{Z_1}{(x_0-1)\log x_0}.
\end{equation}
\end{theorem}
\begin{proof}
(i)  We have
 \begin{align*}
   \Big| \lambda(x) - \log \log x - M \Big|
	      &= \Big| -\sum_{p \le x} \log \Big(1-\frac{1}{p} \Big) - \sum_{p > x} \Big(\log \Big( 1- \frac{1}{p} \Big) + \frac{1}{p} \Big) - \log \log x - \gamma \Big|.
 \end{align*}
By applying Theorem \ref{1overpthm} to the left hand side, for $x \ge x_0$,
we obtain
\begin{equation}\label{prod1:MidStep}
 \Big| \sum_{p \le x}  \log \Big(1-\frac{1}{p} \Big) + S + \log \log x + \gamma \Big| \le \frac{\mathcal{A}_{\ell}(x_0)}{\log ^{\ell} x}
\end{equation}
where $S$ is defined in \eqref{S}.
By rearranging this and then exponentiating, we obtain
\begin{align*}
 \exp \Big( \frac{-\mathcal{A}_{\ell}(x_0)}{\log^{\ell} x} -S \Big) \frac{e^{-\gamma}}{\log x} \le \prod_{p \le x} \Big( 1 - \frac{1}{p} \Big) \le \exp \Big( \frac{\mathcal{A}_{\ell}(x_0)}{\log^{\ell} x} -S \Big) \frac{e^{-\gamma}}{\log x}.
\end{align*}
An application of inequality \eqref{Sboundb} establishes \eqref{bnd:prod1}. \\
(ii)
 We begin with the upper bound. By the Taylor series expansion,  it follows that
 $
 e^u \le 1 + u ( \frac{e^{X} -1}{X} )$, for  $0 \le u \le X$.
 We deduce that for $x \ge x_0$,
 \begin{align*}
 \exp \Big( \frac{\mathcal{A}_{\ell}(x_0)}{\log ^{\ell} x} + \frac{Z_1}{(x-1)\log x} \Big)
  & \le 1+ \Big(\frac{\mathcal{A}_{\ell}(x_0)}{\log ^{\ell} x} + \frac{Z_1}{(x-1)\log x} \Big) \Big(\frac{e^m -1}{m} \Big) \\
  & \le 1+ \frac{1}{\log ^{\ell} x} \Big( \mathcal{A}_{\ell}(x_0) +
  Z_1 \cdot \max_{x \ge x_0} \frac{(\log x)^{{\ell}-1}}{x-1}
  \Big) \Big(\frac{e^m -1}{m} \Big),
 \end{align*}
 where $m$ is defined in \eqref{defn:Atildem}.
 The lower bound in \eqref{bnd:prod2} follows from the lower bound \eqref{bnd:prod1} along with
 the inequality $e^{-u} \ge 1-u$ for $u \ge 0$.
\end{proof}
We now establish a result for the products $\prod_{p \le x} \frac{p}{p-1}$.
\begin{theorem}
  Let $\ell \in \{1,2,3,4,5\}$.  \\
  (i) For $x \ge x_0 \ge 2 $, we have
  \begin{equation}\label{bnd:prod3}
   e^\gamma \log x \exp \Big(-\frac{\mathcal{A}_\ell(x_0)}{\log ^{\ell} x} - \frac{Z_1}{(x-1)\log x} \Big) \le \prod_{p\le x}\frac{p}{p-1} \le e^\gamma \log x \exp \Big(\frac{\mathcal{A}_{\ell}(x_0)}{\log ^{\ell} x} \Big)
  \end{equation}
 where $\mathcal{A}_{\ell}(x_0)$ is defined in \eqref{1overpthm}. \\
 (ii) For $x \ge x_0$, we have
 \begin{equation}
  \label{bnd:prod4}
  e^\gamma \log x \Big( 1 - \frac{\mathcal{D}_{\ell}(x_0)}{\log^{\ell} x} \Big) \le \prod_{p \le x} \frac{p}{p-1}\le e^\gamma \log x \Big( 1 + \frac{\mathcal{D}_{\ell}'(x_0)}{\log^{\ell} x} \Big).
 \end{equation}
 where,
 \begin{equation}
 \begin{split}
 \label{defn:Dellx0}
 \mathcal{D}_{\ell}(x_0) & = \Big( \mathcal{A}_{\ell}(x_0) +
 Z_1 \cdot \max_{x \ge x_0} \frac{(\log x)^{{\ell}-1}}{x-1} \Big) \Big( \frac{e^{m}- 1}{m} \Big),
  \, \,
  m = - \frac{\mathcal{A}_{\ell}(x_0)}{\log^{\ell} x_0} -\frac{Z_1}{(x_0-1) \log x_0}, \\
\end{split}
\end{equation}
and
\begin{equation}
\begin{split}
   \label{defn:Dellprimex0}
  \mathcal{D}_{\ell}'(x_0) & = (\log x_0)^{\ell} \Big( \exp \Big( \frac{\mathcal{A}_{\ell}(x_0)}{\log^{\ell} x_0} \Big) - 1 \Big).
 \end{split}
 \end{equation}
 \end{theorem}
\begin{proof}
(i)  This follows immediately from \eqref{bnd:prod1}. \\
(ii) The upper bound in \eqref{bnd:prod4} follows from the upper bound for $e^u$ used in Theorem  \ref{firstprodthm} part (ii).
For the lower bound, note that
$e^{u}  \ge 1+u ( \frac{e^c -1}{c} )$  for   $c \le u \le 0$ and thus
  \begin{equation}
   \label{explowbd}
  \exp \Big( -\frac{\mathcal{A}_{\ell}(x_0)}{\log^{\ell} x} - \frac{Z_1}{(x -1)\log x} \Big) > 1- \Big( \frac{\mathcal{A}_{\ell}(x_0)}{\log^{\ell} x} +\frac{Z_1}{(x -1)\log x} \Big) \Big( \frac{e^{c_2} -1}{c_2} \Big)
 \end{equation}
 for $c_2= - \Big( \frac{\mathcal{A}_{\ell}(x_0)}{\log^{\ell} x_0} + \frac{Z_1}{(x_0-1)\log x_0} \Big) = -m$.  Observe that, for $x \ge x_0$
 \[
   \Big( \frac{\mathcal{A}_{\ell}(x_0)}{\log^{\ell} x} +\frac{Z_1}{(x -1)\log x} \Big) \le \frac{1}{(\log x)^{\ell}} \Big( \mathcal{A}_{\ell}(x_0) +
   Z_1 \cdot \max_{x \ge x_0} \frac{(\log x)^{{\ell}-1}}{x-1}  \Big).
 \]
 Inserting this into \eqref{explowbd} establishes the lower bound in
 \eqref{bnd:prod3}.
\end{proof}

\section{Numerical Confirmation} \label{section:numerics}

In this section we provide two useful lemmas that will be employed in  Tables \ref{MkmktableThm2},   \ref{MkmktableThm3}, and  \ref{MkmktableThm4}.
In those tables, we compute numerical bounds for the functions $\lambda(x), \Upsilon(x)$, and $\widetilde{\psi}(x)$ for $x \le e^{30}$.
Tables \ref{GRHtable} - \ref{MkmktableThm4} were computed by evaluating the expression at every prime up to the bound iteratively. An unexpected power interruption after several weeks of computation led to the chosen cutoff points.
Essentially, these lemmas allow us to compute the relevant error terms at  endpoints on each interval of consecutive prime numbers $[p_n, p_{n+1})$.
\begin{lemma}\label{1pNumericalLemma}
 (i) Let $\ell \in \{1,2,3,4,5\}$ and let $p_n$ be any prime greater than $\exp(\sqrt{\ell})$. For $x \in [p_n, p_{n+1})$,
 \begin{equation}\label{1pNumericalBound}
  | \lambda(x) - \log\log x - M | \le \frac{ \alpha_n }{\log^\ell x}
 \end{equation}
 where $\alpha_n = (\log p_n)^{\ell} \max \Big(  | \sum_{p\le p_n} \frac1p - \log\log p_n - M|,   | \sum_{p\le p_n} \frac1p - \log\log p_{n+1} - M| \Big)$. \\
 (ii) Let $\ell \in \{1,2,3,4,5\}$ and let $p_n$ be any prime greater than $32$. For $x \in [p_n, p_{n+1})$,
 \begin{equation}\label{logpNumericalBound}
  | \Upsilon(x) - \log x + M' | \le \frac{ \beta_n }{\log^{\ell} x}
 \end{equation}
 where $\beta_n =  (\log  p_n)^{\ell}  \max \Big( |\sum_{p\le p_n} \frac{\log p}{p} - \log p_n + M'|,  |\sum_{p\le p_n} \frac{\log p}{p} - \log p_{n+1} + M'|  \Big) $.
\end{lemma}

\begin{proof}
We only provide the proof of (i) as (ii) is very similar.
Let $x \in [ p_n, p_{n+1} )$ and  $f(x) := (\log x)^{\ell} ( \sum_{p \le x}\frac{1}{p} - \log\log x -M)$.
Let $C_n = \sum_{p \le x}\frac{1}{p} - M = \sum_{p \le p_n}\frac{1}{p} - M $, so that $f(x) = (\log x)^{\ell} (C-\log\log x)$ for $x \in [ p_n, p_{n+1} )$.
We have  $f'(x) = \frac{(\log x)^{{\ell}-1}}{x}( \ell C - \ell\log\log x -1 )$. It follows that $f$ is non-increasing as long as
\begin{equation}
\ell C - \ell\log\log x -1 \le 0
\Longleftrightarrow
  x \ge \exp ( \exp ( C-\tfrac{1}{\ell} )).
\end{equation}
Rosser and Schoenfeld \cite{RS62} showed
$\sum_{p\le x}\frac{1}{p} -\log\log x - M < \frac{1}{\log^2 x}$  for $x >1$.
It follows that
 \begin{equation}
  \exp \Big( \exp \Big( C_n -\frac{1}{\ell} \Big) \Big)
 <
   \exp \Big( \exp \Big( \frac{1}{\log^2 x} - \frac{1}{\ell} \Big) \log x \Big).
 \end{equation}
 If   $x \ge \exp(\sqrt{\ell})$, then  $ \exp \Big( \frac{1}{\log^2 x} - \frac{1}{\ell} \Big) \le 1$ and it follows that
 $ \exp ( \exp ( C-\tfrac{1}{\ell} )) \le 1$ as desired.
 Thus for any $\ell$, $f(x)$ is decreasing for $x \ge 10$. Thus $f(p_{n+1}-\delta) \le f(x) \le  f(p_n)$ for $x \in [p_n,p_{n+1}-\delta]$, for any $\delta >0$.
 Thus $|f(x)| \le \max( |f(p_n)|, |f(p_{n+1}-\delta)|)$, for any $\delta >0$.  Thus on $[p_n, p_{n+1})$, we have  $|f(x)| \le \max( |f(p_n)|, \lim_{\delta \to 0}|f(p_{n+1}-\delta)|)$
  which yields \eqref{1pNumericalBound}.
\end{proof}

\section{Concluding remarks}
This paper uses bounds on error terms in the prime number theorem of the exponential-type and log-type.
However, other types of bounds exists. For instance
for  $x_0 \ge 2$, there exist positive constants $a,b$, and $c$ such that
\begin{equation}
     \label{thetaexplicit3}
    \text{ for } x \ge x_0, \quad   \Big|\frac{\vartheta(x)-x}{x}  \Big| \le
    a   (\log x)^{b}  \exp(-c (\log x)^{3/5}(\log\log x )^{-1/5}).
\end{equation}
Bounds of this form were recently given in \cite{JY}. In this article we do not establish an error term for Merten sums of the type in  \eqref{thetaexplicit3}.
With additional work to bound the additional error terms in the explicit formula such bounds could be obtained.

Additionally, there are tighter bounds for small values of $x$:
\begin{equation}
     \label{thetaexplicit4}
    \text{ for }  599 \le x \le \exp(60.8), \quad   \Big|\frac{\vartheta(x)-x}{x}  \Big| \le
    \frac{\log^2 x}{8\pi \sqrt{x} }.
\end{equation}
These estimates have first proven by B\"uthe \cite{But16} and recently been updated by Johnston  \cite[Proposition 3.1, p. 1310]{J}.
We note that such bounds could likely be used to obtain similarly shaped bounds on Mertens sums for relatively small $x$.

\section*{Acknowledgements}
The authors used ChatGptPlus in the formatting of the tables in the appendix.   Thank-you to L\'{e}on Ng-Kadiri for help with the rounding of numbers in Tables 4-6.

\bibliographystyle{abbrv}
\bibliography{mybib}

\pagebreak
\appendix

\section{Tables of values}
\label{Section:Tables}

\begin{table}[h]
\caption{Values for $A_{\psi}(x_0)$, $A_{\vartheta}(x_0)$,  $A_{\lambda}(x_0)$,
 $A_{\Upsilon}(x_0)$,
 $A_{\widetilde{\psi}}(x_0)$.
}
 \begin{tabular}{|c|c|}
\hline
 $b=\log(x_0)$ &  $A_{\psi}(x_0)$/$A_{\vartheta}(x_0)$/  $A_{\lambda}(x_0)$/
 $A_{\Upsilon}(x_0)$/
 $A_{\widetilde{\psi}}(x_0)$\\
\hline
$2$ & 9.2203 \\
$200\,000$   &    9.235416396  \\
$300\,000$   &    8.962915751 \\
$400\,000$   &     8.792027714 \\
$500\,000$   &      8.670751841  \\
$600\,000$   &     8.578303214  \\
$700\,000$   &     8.504472502   \\
$800\,000$   &       8.443521723  \\
$900\,000$   &      8.391956205  \\
$10^{6}$   &      8.347494489  \\
$10^{7}$   &     7.669199615   \\
$10^{8}$   &    7.327728754  \\
$10^{9}$   &   7.147859862  \\
\hline

\end{tabular}
\label{ApsiUpsilonwpsi}
\end{table}

\pagebreak

\newpage

\begin{table}
\caption{: For each $k$ and all $x>x_0$ we have $|\theta(x)-x| < \frac{\eta_\ell(x_0)x}{\log(x)^k}$
}\label{table:eta}
\centering
\begin{tabular}{|c|c|c|c|c|c|c|}
\hline
$\log(x_0)$ & $\eta_0(x_0)$ & $\eta_1(x_0)$ & $\eta_2(x_0)$ & $\eta_3(x_0)$ & $\eta_4(x_0)$ & $\eta_5(x_0)$ \\
\hline
 $20$ & $8.8530\cdot10^{-5}$ & 0.0017795 & 0.035767 & 0.71892 & 28.964 & 60042 \\
 $21$ & $5.3697\cdot10^{-5}$ & 0.0011330 & 0.023907 & 0.50442 & 28.964 & 60042 \\
 $22$ & $3.2569\cdot10^{-5}$ & 0.00071976 & 0.015907 & 0.35154 & 28.964 & 60042 \\
 $23$ & $1.9754\cdot10^{-5}$ & 0.00045632 & 0.010541 & 0.24350 & 28.964 & 60042 \\
 $24$ & $1.1982\cdot10^{-5}$ & 0.00028875 & 0.0069589 & 0.16771 & 28.964 & 60042 \\
 $25$ & $7.2670\cdot10^{-6}$ & 0.00018241 & 0.0045783 & 0.11492 & 28.964 & 60042 \\
 $26$ & $4.4077\cdot10^{-6}$ & 0.00011504 & 0.0030026 & 0.078367 & 28.964 & 60042 \\
 $27$ & $2.6734\cdot10^{-6}$ & $7.2449\cdot10^{-5}$ & 0.0019634 & 0.053207 & 28.964 & 60042 \\
 $28$ & $1.6215\cdot10^{-6}$ & $4.5564\cdot10^{-5}$ & 0.0012804 & 0.035978 & 28.964 & 60042 \\
 $29$ & $9.8348\cdot10^{-7}$ & $2.8620\cdot10^{-5}$ & 0.00083282 & 0.024236 & 28.964 & 60042 \\
 $30$ & $5.9651\cdot10^{-7}$ & $1.7955\cdot10^{-5}$ & 0.00054045 & 0.016268 & 28.964 & 60042 \\
 $35$ & $4.8965\cdot10^{-8}$ & $1.7187\cdot10^{-6}$ & $6.0325\cdot10^{-5}$ & 0.013972 & 28.964 & 60042 \\
 $40$ & $1.9537\cdot10^{-8}$ & $8.5536\cdot10^{-7}$ & $3.7465\cdot10^{-5}$ & 0.013972 & 28.964 & 60042 \\
 $45$ & $1.1075\cdot10^{-8}$ & $4.9892\cdot10^{-7}$ & $2.2477\cdot10^{-5}$ & 0.013972 & 28.964 & 60042 \\
 $50$ & $1.1336\cdot10^{-9}$ & $5.6736\cdot10^{-8}$ & $6.7399\cdot10^{-6}$ & 0.013972 & 28.964 & 60042 \\
 $70$ & $2.3059\cdot10^{-12}$ & $3.2513\cdot10^{-9}$ & $6.7399\cdot10^{-6}$ & 0.013972 & 28.964 & 60042 \\
 $100$ & $2.0097\cdot10^{-12}$ & $3.2513\cdot10^{-9}$ & $6.7399\cdot10^{-6}$ & 0.013972 & 28.964 & 60042 \\
 $200$ & $1.7684\cdot10^{-12}$ & $3.2513\cdot10^{-9}$ & $6.7399\cdot10^{-6}$ & 0.013972 & 28.964 & 60042 \\
 $300$ & $1.6930\cdot10^{-12}$ & $3.2513\cdot10^{-9}$ & $6.7399\cdot10^{-6}$ & 0.013972 & 28.964 & 60042 \\
 $400$ & $1.6560\cdot10^{-12}$ & $3.2513\cdot10^{-9}$ & $6.7399\cdot10^{-6}$ & 0.013972 & 28.964 & 60042 \\
 $500$ & $1.6341\cdot10^{-12}$ & $3.2513\cdot10^{-9}$ & $6.7399\cdot10^{-6}$ & 0.013972 & 28.964 & 60042 \\
 $700$ & $1.6092\cdot10^{-12}$ & $3.2513\cdot10^{-9}$ & $6.7399\cdot10^{-6}$ & 0.013972 & 28.964 & 60042 \\
 $1000$ & $1.5907\cdot10^{-12}$ & $3.2513\cdot10^{-9}$ & $6.7399\cdot10^{-6}$ & 0.013972 & 28.964 & 60042 \\
 $1500$ & $1.5763\cdot10^{-12}$ & $3.2513\cdot10^{-9}$ & $6.7399\cdot10^{-6}$ & 0.013972 & 28.964 & 60042 \\
 $2000$ & $1.5692\cdot10^{-12}$ & $3.2513\cdot10^{-9}$ & $6.7399\cdot10^{-6}$ & 0.013972 & 28.964 & 60042 \\
 $2500$ & $1.0690\cdot10^{-13}$ & $2.6726\cdot10^{-10}$ & $6.6819\cdot10^{-7}$ & 0.0016706 & 4.1769 & 10443 \\
 $3000$ & $4.8408\cdot10^{-15}$ & $1.4524\cdot10^{-11}$ & $4.3573\cdot10^{-8}$ & 0.00013073 & 0.39221 & 1176.7 \\
 $3500$ & $2.2451\cdot10^{-16}$ & $7.8580\cdot10^{-13}$ & $2.7505\cdot10^{-9}$ & $9.6272\cdot10^{-6}$ & 0.033697 & 117.95 \\
 $4000$ & $1.0742\cdot10^{-17}$ & $4.2967\cdot10^{-14}$ & $1.7188\cdot10^{-10}$ & $6.8754\cdot10^{-7}$ & 0.0027503 & 11.002 \\
 $4500$ & $5.3660\cdot10^{-19}$ & $2.4148\cdot10^{-15}$ & $1.0868\cdot10^{-11}$ & $4.8905\cdot10^{-8}$ & 0.00022008 & 0.99040 \\
 $5000$ & $2.9189\cdot10^{-20}$ & $1.4595\cdot10^{-16}$ & $7.2977\cdot10^{-13}$ & $3.6490\cdot10^{-9}$ & $1.8246\cdot10^{-5}$ & 0.091232 \\
 $6000$ & $1.2652\cdot10^{-22}$ & $7.5913\cdot10^{-19}$ & $4.5549\cdot10^{-15}$ & $2.7331\cdot10^{-11}$ & $1.6399\cdot10^{-7}$ & 0.00098396 \\
 $7000$ & $8.3050\cdot10^{-25}$ & $5.8137\cdot10^{-21}$ & $4.0697\cdot10^{-17}$ & $2.8489\cdot10^{-13}$ & $1.9943\cdot10^{-9}$ & $1.3961\cdot10^{-5}$ \\
 $8000$ & $7.5937\cdot10^{-27}$ & $6.0751\cdot10^{-23}$ & $4.8601\cdot10^{-19}$ & $3.8882\cdot10^{-15}$ & $3.1106\cdot10^{-11}$ & $2.4885\cdot10^{-7}$ \\
 $9000$ & $8.9977\cdot10^{-29}$ & $8.0984\cdot10^{-25}$ & $7.2889\cdot10^{-21}$ & $6.5604\cdot10^{-17}$ & $5.9047\cdot10^{-13}$ & $5.3145\cdot10^{-9}$ \\
 $10000$ & $1.3352\cdot10^{-30}$ & $1.3352\cdot10^{-26}$ & $1.3353\cdot10^{-22}$ & $1.3354\cdot10^{-18}$ & $1.3354\cdot10^{-14}$ & $1.3355\cdot10^{-10}$ \\
\hline
\end{tabular}
\end{table}

\newpage

\begin{table}
\caption{For each $k$ and all $x>x_0$ we have
$|\psi(x)-x| < \frac{\widetilde{\eta_{\ell}}(x_0)x}{\log(x)^k}$.}
\label{table:etatilde}
\centering
\begin{tabular}{|c|c|c|c|c|c|c|}
\hline
$\log(x_0)$ & $\widetilde{\eta}_0(x_0)$ & $\widetilde{\eta}_1(x_0)$ & $\widetilde{\eta}_2(x_0)$ & $\widetilde{\eta}_3(x_0)$ & $\widetilde{\eta}_4(x_0)$ & $\widetilde{\eta}_5(x_0)$ \\
\hline
 $20$ & $4.2676\cdot10^{-5}$ & 0.00085779 & 0.017242 & 0.34656 & 28.964 & 60042 \\
 $21$ & $2.5885\cdot10^{-5}$ & 0.00054616 & 0.011524 & 0.24316 & 28.964 & 60042 \\
 $22$ & $1.5700\cdot10^{-5}$ & 0.00034697 & 0.0076679 & 0.16946 & 28.964 & 60042 \\
 $23$ & $9.5223\cdot10^{-6}$ & 0.00021997 & 0.0050812 & 0.11738 & 28.964 & 60042 \\
 $24$ & $5.7756\cdot10^{-6}$ & 0.00013920 & 0.0033546 & 0.080844 & 28.964 & 60042 \\
 $25$ & $3.5031\cdot10^{-6}$ & $8.7927\cdot10^{-5}$ & 0.0022070 & 0.055395 & 28.964 & 60042 \\
 $26$ & $2.1248\cdot10^{-6}$ & $5.5455\cdot10^{-5}$ & 0.0014474 & 0.037777 & 28.964 & 60042 \\
 $27$ & $1.2888\cdot10^{-6}$ & $3.4924\cdot10^{-5}$ & 0.00094644 & 0.025649 & 28.964 & 60042 \\
 $28$ & $7.8164\cdot10^{-7}$ & $2.1964\cdot10^{-5}$ & 0.00061719 & 0.017343 & 28.964 & 60042 \\
 $29$ & $4.7409\cdot10^{-7}$ & $1.3796\cdot10^{-5}$ & 0.00040147 & 0.013972 & 28.964 & 60042 \\
 $30$ & $2.8755\cdot10^{-7}$ & $8.6553\cdot10^{-6}$ & 0.00026053 & 0.013972 & 28.964 & 60042 \\
 $35$ & $2.3604\cdot10^{-8}$ & $8.4150\cdot10^{-7}$ & $3.6858\cdot10^{-5}$ & 0.013972 & 28.964 & 60042 \\
 $40$ & $1.9220\cdot10^{-8}$ & $8.4150\cdot10^{-7}$ & $3.6858\cdot10^{-5}$ & 0.013972 & 28.964 & 60042 \\
 $45$ & $1.0906\cdot10^{-8}$ & $4.9130\cdot10^{-7}$ & $2.2133\cdot10^{-5}$ & 0.013972 & 28.964 & 60042 \\
 $50$ & $1.1197\cdot10^{-9}$ & $5.6041\cdot10^{-8}$ & $6.7399\cdot10^{-6}$ & 0.013972 & 28.964 & 60042 \\
 $70$ & $2.3053\cdot10^{-12}$ & $3.2513\cdot10^{-9}$ & $6.7399\cdot10^{-6}$ & 0.013972 & 28.964 & 60042 \\
 $100$ & $2.0097\cdot10^{-12}$ & $3.2513\cdot10^{-9}$ & $6.7399\cdot10^{-6}$ & 0.013972 & 28.964 & 60042 \\
 $200$ & $1.7684\cdot10^{-12}$ & $3.2513\cdot10^{-9}$ & $6.7399\cdot10^{-6}$ & 0.013972 & 28.964 & 60042 \\
 $300$ & $1.6930\cdot10^{-12}$ & $3.2513\cdot10^{-9}$ & $6.7399\cdot10^{-6}$ & 0.013972 & 28.964 & 60042 \\
 $400$ & $1.6560\cdot10^{-12}$ & $3.2513\cdot10^{-9}$ & $6.7399\cdot10^{-6}$ & 0.013972 & 28.964 & 60042 \\
 $500$ & $1.6341\cdot10^{-12}$ & $3.2513\cdot10^{-9}$ & $6.7399\cdot10^{-6}$ & 0.013972 & 28.964 & 60042 \\
 $700$ & $1.6092\cdot10^{-12}$ & $3.2513\cdot10^{-9}$ & $6.7399\cdot10^{-6}$ & 0.013972 & 28.964 & 60042 \\
 $1000$ & $1.5907\cdot10^{-12}$ & $3.2513\cdot10^{-9}$ & $6.7399\cdot10^{-6}$ & 0.013972 & 28.964 & 60042 \\
 $1500$ & $1.5763\cdot10^{-12}$ & $3.2513\cdot10^{-9}$ & $6.7399\cdot10^{-6}$ & 0.013972 & 28.964 & 60042 \\
 $2000$ & $1.5692\cdot10^{-12}$ & $3.2513\cdot10^{-9}$ & $6.7399\cdot10^{-6}$ & 0.013972 & 28.964 & 60042 \\
 $2500$ & $1.0690\cdot10^{-13}$ & $2.6726\cdot10^{-10}$ & $6.6819\cdot10^{-7}$ & 0.0016706 & 4.1769 & 10443 \\
 $3000$ & $4.8408\cdot10^{-15}$ & $1.4524\cdot10^{-11}$ & $4.3573\cdot10^{-8}$ & 0.00013073 & 0.39221 & 1176.7 \\
 $3500$ & $2.2451\cdot10^{-16}$ & $7.8580\cdot10^{-13}$ & $2.7505\cdot10^{-9}$ & $9.6272\cdot10^{-6}$ & 0.033697 & 117.95 \\
 $4000$ & $1.0742\cdot10^{-17}$ & $4.2967\cdot10^{-14}$ & $1.7188\cdot10^{-10}$ & $6.8754\cdot10^{-7}$ & 0.0027503 & 11.002 \\
 $4500$ & $5.3660\cdot10^{-19}$ & $2.4148\cdot10^{-15}$ & $1.0868\cdot10^{-11}$ & $4.8905\cdot10^{-8}$ & 0.00022008 & 0.99040 \\
 $5000$ & $2.9189\cdot10^{-20}$ & $1.4595\cdot10^{-16}$ & $7.2977\cdot10^{-13}$ & $3.6490\cdot10^{-9}$ & $1.8246\cdot10^{-5}$ & 0.091232 \\
 $6000$ & $1.2652\cdot10^{-22}$ & $7.5913\cdot10^{-19}$ & $4.5549\cdot10^{-15}$ & $2.7331\cdot10^{-11}$ & $1.6399\cdot10^{-7}$ & 0.00098396 \\
 $7000$ & $8.3050\cdot10^{-25}$ & $5.8137\cdot10^{-21}$ & $4.0697\cdot10^{-17}$ & $2.8489\cdot10^{-13}$ & $1.9943\cdot10^{-9}$ & $1.3961\cdot10^{-5}$ \\
 $8000$ & $7.5937\cdot10^{-27}$ & $6.0751\cdot10^{-23}$ & $4.8601\cdot10^{-19}$ & $3.8882\cdot10^{-15}$ & $3.1106\cdot10^{-11}$ & $2.4885\cdot10^{-7}$ \\
 $9000$ & $8.9977\cdot10^{-29}$ & $8.0984\cdot10^{-25}$ & $7.2889\cdot10^{-21}$ & $6.5604\cdot10^{-17}$ & $5.9047\cdot10^{-13}$ & $5.3145\cdot10^{-9}$ \\
 $10000$ & $1.3352\cdot10^{-30}$ & $1.3352\cdot10^{-26}$ & $1.3353\cdot10^{-22}$ & $1.3354\cdot10^{-18}$ & $1.3354\cdot10^{-14}$ & $1.3355\cdot10^{-10}$ \\
\hline
\end{tabular}
\end{table}
\newpage

\begin{table}[htbp]
\caption{Values of $\mathcal{A}_{\ell}(x_0)$ for $\ell\in\{1,2,3,4,5\}$, in $|\lambda(x)-\log\log x-M|\leq \dfrac{\mathcal{A}_{\ell}(x_0)}{(\log x)^{\ell}}$, valid for $x\geq x_0$, where $\mathcal{A}_{\ell}(x_0)$ is given by (76).}
\label{Altable}
\centering
\begin{tabular}{|c|c|c|c|c|c|}
\hline
$\log x_0$ & $\mathcal{A}_1(x_0)$ & $\mathcal{A}_2(x_0)$ & $\mathcal{A}_3(x_0)$ & $\mathcal{A}_4(x_0)$ & $\mathcal{A}_5(x_0)$ \\
\hline
$20.000$ & $1.8928 \cdot 10^{-4}$ & $3.7966 \cdot 10^{-3}$ & $7.6155 \cdot 10^{-2}$ & $1.5277 \cdot 10^{0}$ & $4.5159 \cdot 10^{1}$ \\
$25.000$ & $1.5302 \cdot 10^{-5}$ & $3.8345 \cdot 10^{-4}$ & $9.6091 \cdot 10^{-3}$ & $2.4081 \cdot 10^{-1}$ & $3.2114 \cdot 10^{1}$ \\
$30.000$ & $1.2469 \cdot 10^{-6}$ & $3.7480 \cdot 10^{-5}$ & $1.1267 \cdot 10^{-3}$ & $3.3868 \cdot 10^{-2}$ & $2.9492 \cdot 10^{1}$ \\
$35.000$ & $1.0194 \cdot 10^{-7}$ & $3.5739 \cdot 10^{-6}$ & $1.2530 \cdot 10^{-4}$ & $1.6248 \cdot 10^{-2}$ & $2.9044 \cdot 10^{1}$ \\
$45.000$ & $1.1429 \cdot 10^{-8}$ & $5.1487 \cdot 10^{-7}$ & $2.3195 \cdot 10^{-5}$ & $1.4005 \cdot 10^{-2}$ & $2.8965 \cdot 10^{1}$ \\
$50.000$ & $1.1626 \cdot 10^{-9}$ & $5.8187 \cdot 10^{-8}$ & $6.8125 \cdot 10^{-6}$ & $1.3976 \cdot 10^{-2}$ & $2.8964 \cdot 10^{1}$ \\
$70.000$ & $2.3101 \cdot 10^{-12}$ & $3.2516 \cdot 10^{-9}$ & $6.7399 \cdot 10^{-6}$ & $1.3972 \cdot 10^{-2}$ & $2.8964 \cdot 10^{1}$ \\
$100.00$ & $2.0100 \cdot 10^{-12}$ & $3.2513 \cdot 10^{-9}$ & $6.7399 \cdot 10^{-6}$ & $1.3972 \cdot 10^{-2}$ & $2.8964 \cdot 10^{1}$ \\
$200.00$ & $1.7684 \cdot 10^{-12}$ & $3.2513 \cdot 10^{-9}$ & $6.7399 \cdot 10^{-6}$ & $1.3972 \cdot 10^{-2}$ & $2.8964 \cdot 10^{1}$ \\
$300.00$ & $1.6930 \cdot 10^{-12}$ & $3.2513 \cdot 10^{-9}$ & $6.7399 \cdot 10^{-6}$ & $1.3972 \cdot 10^{-2}$ & $2.8964 \cdot 10^{1}$ \\
$400.00$ & $1.6560 \cdot 10^{-12}$ & $3.2513 \cdot 10^{-9}$ & $6.7399 \cdot 10^{-6}$ & $1.3972 \cdot 10^{-2}$ & $2.8964 \cdot 10^{1}$ \\
$500.00$ & $1.6341 \cdot 10^{-12}$ & $3.2513 \cdot 10^{-9}$ & $6.7399 \cdot 10^{-6}$ & $1.3972 \cdot 10^{-2}$ & $2.8964 \cdot 10^{1}$ \\
$700.00$ & $1.6092 \cdot 10^{-12}$ & $3.2513 \cdot 10^{-9}$ & $6.7399 \cdot 10^{-6}$ & $1.3972 \cdot 10^{-2}$ & $2.8964 \cdot 10^{1}$ \\
$1000.0$ & $1.5907 \cdot 10^{-12}$ & $3.2513 \cdot 10^{-9}$ & $6.7399 \cdot 10^{-6}$ & $1.3972 \cdot 10^{-2}$ & $2.8964 \cdot 10^{1}$ \\
$1500.0$ & $1.5763 \cdot 10^{-12}$ & $3.2513 \cdot 10^{-9}$ & $6.7399 \cdot 10^{-6}$ & $1.3972 \cdot 10^{-2}$ & $2.8964 \cdot 10^{1}$ \\
$2000.0$ & $1.5692 \cdot 10^{-12}$ & $3.2513 \cdot 10^{-9}$ & $6.7399 \cdot 10^{-6}$ & $1.3972 \cdot 10^{-2}$ & $2.8964 \cdot 10^{1}$ \\
$2500.0$ & $1.0690 \cdot 10^{-13}$ & $2.6726 \cdot 10^{-10}$ & $6.6819 \cdot 10^{-7}$ & $1.6706 \cdot 10^{-3}$ & $4.1769 \cdot 10^{0}$ \\
$3000.0$ & $4.8408 \cdot 10^{-15}$ & $1.4524 \cdot 10^{-11}$ & $4.3573 \cdot 10^{-8}$ & $1.3073 \cdot 10^{-4}$ & $3.9221 \cdot 10^{-1}$ \\
$3500.0$ & $2.2451 \cdot 10^{-16}$ & $7.8580 \cdot 10^{-13}$ & $2.7505 \cdot 10^{-9}$ & $9.6272 \cdot 10^{-6}$ & $3.3697 \cdot 10^{-2}$ \\
$4000.0$ & $1.0742 \cdot 10^{-17}$ & $4.2967 \cdot 10^{-14}$ & $1.7188 \cdot 10^{-10}$ & $6.8754 \cdot 10^{-7}$ & $2.7503 \cdot 10^{-3}$ \\
$4500.0$ & $5.3660 \cdot 10^{-19}$ & $2.4148 \cdot 10^{-15}$ & $1.0868 \cdot 10^{-11}$ & $4.8905 \cdot 10^{-8}$ & $2.2008 \cdot 10^{-4}$ \\
$5000.0$ & $2.9189 \cdot 10^{-20}$ & $1.4595 \cdot 10^{-16}$ & $7.2977 \cdot 10^{-13}$ & $3.6490 \cdot 10^{-9}$ & $1.8246 \cdot 10^{-5}$ \\
$6000.0$ & $1.2652 \cdot 10^{-22}$ & $7.5913 \cdot 10^{-19}$ & $4.5549 \cdot 10^{-15}$ & $2.7331 \cdot 10^{-11}$ & $1.6399 \cdot 10^{-7}$ \\
$7000.0$ & $8.3050 \cdot 10^{-25}$ & $5.8137 \cdot 10^{-21}$ & $4.0697 \cdot 10^{-17}$ & $2.8489 \cdot 10^{-13}$ & $1.9943 \cdot 10^{-9}$ \\
$8000.0$ & $7.5937 \cdot 10^{-27}$ & $6.0751 \cdot 10^{-23}$ & $4.8601 \cdot 10^{-19}$ & $3.8882 \cdot 10^{-15}$ & $3.1106 \cdot 10^{-11}$ \\
$9000.0$ & $8.9977 \cdot 10^{-29}$ & $8.0984 \cdot 10^{-25}$ & $7.2889 \cdot 10^{-21}$ & $6.5604 \cdot 10^{-17}$ & $5.9047 \cdot 10^{-13}$ \\
$10000$ & $1.3352 \cdot 10^{-30}$ & $1.3352 \cdot 10^{-26}$ & $1.3353 \cdot 10^{-22}$ & $1.3354 \cdot 10^{-18}$ & $1.3354 \cdot 10^{-14}$ \\
\hline
\end{tabular}
\end{table}
\clearpage
\begin{table}[htbp]
\caption{Values of $\mathcal{B}_{\ell}(x_0)$ for $\ell\in\{1,2,3,4,5\}$, in $|\Upsilon(x)-\log x+M'|\leq \dfrac{\mathcal{B}_{\ell}(x_0)}{(\log x)^{\ell}}$, valid for $x\geq x_0$, where $\mathcal{B}_{\ell}(x_0)$ is given by (149).}
\label{Bltable}
\centering
\begin{tabular}{|c|c|c|c|c|c|}
\hline
$\log x_0$ & $\mathcal{B}_1(x_0)$ & $\mathcal{B}_2(x_0)$ & $\mathcal{B}_3(x_0)$ & $\mathcal{B}_4(x_0)$ & $\mathcal{B}_5(x_0)$ \\
\hline
$20.000$ & $3.7003 \cdot 10^{-3}$ & $7.4229 \cdot 10^{-2}$ & $1.4891 \cdot 10^{0}$ & $4.4386 \cdot 10^{1}$ & $6.0350 \cdot 10^{4}$ \\
$25.000$ & $3.7571 \cdot 10^{-4}$ & $9.4153 \cdot 10^{-3}$ & $2.3596 \cdot 10^{-1}$ & $3.1993 \cdot 10^{1}$ & $6.0117 \cdot 10^{4}$ \\
$30.000$ & $3.6849 \cdot 10^{-5}$ & $1.1077 \cdot 10^{-3}$ & $3.3299 \cdot 10^{-2}$ & $2.9475 \cdot 10^{1}$ & $6.0057 \cdot 10^{4}$ \\
$35.000$ & $3.5223 \cdot 10^{-6}$ & $1.2350 \cdot 10^{-4}$ & $1.6185 \cdot 10^{-2}$ & $2.9041 \cdot 10^{1}$ & $6.0044 \cdot 10^{4}$ \\
$45.000$ & $5.1452 \cdot 10^{-7}$ & $2.3179 \cdot 10^{-5}$ & $1.4004 \cdot 10^{-2}$ & $2.8965 \cdot 10^{1}$ & $6.0042 \cdot 10^{4}$ \\
$50.000$ & $5.8159 \cdot 10^{-8}$ & $6.8111 \cdot 10^{-6}$ & $1.3976 \cdot 10^{-2}$ & $2.8964 \cdot 10^{1}$ & $6.0042 \cdot 10^{4}$ \\
$70.000$ & $3.2516 \cdot 10^{-9}$ & $6.7399 \cdot 10^{-6}$ & $1.3972 \cdot 10^{-2}$ & $2.8964 \cdot 10^{1}$ & $6.0042 \cdot 10^{4}$ \\
$100.00$ & $3.2513 \cdot 10^{-9}$ & $6.7399 \cdot 10^{-6}$ & $1.3972 \cdot 10^{-2}$ & $2.8964 \cdot 10^{1}$ & $6.0042 \cdot 10^{4}$ \\
$200.00$ & $3.2513 \cdot 10^{-9}$ & $6.7399 \cdot 10^{-6}$ & $1.3972 \cdot 10^{-2}$ & $2.8964 \cdot 10^{1}$ & $6.0042 \cdot 10^{4}$ \\
$300.00$ & $3.2513 \cdot 10^{-9}$ & $6.7399 \cdot 10^{-6}$ & $1.3972 \cdot 10^{-2}$ & $2.8964 \cdot 10^{1}$ & $6.0042 \cdot 10^{4}$ \\
$400.00$ & $3.2513 \cdot 10^{-9}$ & $6.7399 \cdot 10^{-6}$ & $1.3972 \cdot 10^{-2}$ & $2.8964 \cdot 10^{1}$ & $6.0042 \cdot 10^{4}$ \\
$500.00$ & $3.2513 \cdot 10^{-9}$ & $6.7399 \cdot 10^{-6}$ & $1.3972 \cdot 10^{-2}$ & $2.8964 \cdot 10^{1}$ & $6.0042 \cdot 10^{4}$ \\
$700.00$ & $3.2513 \cdot 10^{-9}$ & $6.7399 \cdot 10^{-6}$ & $1.3972 \cdot 10^{-2}$ & $2.8964 \cdot 10^{1}$ & $6.0042 \cdot 10^{4}$ \\
$1000.0$ & $3.2513 \cdot 10^{-9}$ & $6.7399 \cdot 10^{-6}$ & $1.3972 \cdot 10^{-2}$ & $2.8964 \cdot 10^{1}$ & $6.0042 \cdot 10^{4}$ \\
$1500.0$ & $3.2513 \cdot 10^{-9}$ & $6.7399 \cdot 10^{-6}$ & $1.3972 \cdot 10^{-2}$ & $2.8964 \cdot 10^{1}$ & $6.0042 \cdot 10^{4}$ \\
$2000.0$ & $3.2513 \cdot 10^{-9}$ & $6.7399 \cdot 10^{-6}$ & $1.3972 \cdot 10^{-2}$ & $2.8964 \cdot 10^{1}$ & $6.0042 \cdot 10^{4}$ \\
$2500.0$ & $2.6726 \cdot 10^{-10}$ & $6.6819 \cdot 10^{-7}$ & $1.6706 \cdot 10^{-3}$ & $4.1769 \cdot 10^{0}$ & $1.0443 \cdot 10^{4}$ \\
$3000.0$ & $1.4524 \cdot 10^{-11}$ & $4.3573 \cdot 10^{-8}$ & $1.3073 \cdot 10^{-4}$ & $3.9221 \cdot 10^{-1}$ & $1.1767 \cdot 10^{3}$ \\
$3500.0$ & $7.8580 \cdot 10^{-13}$ & $2.7505 \cdot 10^{-9}$ & $9.6272 \cdot 10^{-6}$ & $3.3697 \cdot 10^{-2}$ & $1.1795 \cdot 10^{2}$ \\
$4000.0$ & $4.2967 \cdot 10^{-14}$ & $1.7188 \cdot 10^{-10}$ & $6.8754 \cdot 10^{-7}$ & $2.7503 \cdot 10^{-3}$ & $1.1002 \cdot 10^{1}$ \\
$4500.0$ & $2.4148 \cdot 10^{-15}$ & $1.0868 \cdot 10^{-11}$ & $4.8905 \cdot 10^{-8}$ & $2.2008 \cdot 10^{-4}$ & $9.9040 \cdot 10^{-1}$ \\
$5000.0$ & $1.4595 \cdot 10^{-16}$ & $7.2977 \cdot 10^{-13}$ & $3.6490 \cdot 10^{-9}$ & $1.8246 \cdot 10^{-5}$ & $9.1232 \cdot 10^{-2}$ \\
$6000.0$ & $7.5913 \cdot 10^{-19}$ & $4.5549 \cdot 10^{-15}$ & $2.7331 \cdot 10^{-11}$ & $1.6399 \cdot 10^{-7}$ & $9.8396 \cdot 10^{-4}$ \\
$7000.0$ & $5.8137 \cdot 10^{-21}$ & $4.0697 \cdot 10^{-17}$ & $2.8489 \cdot 10^{-13}$ & $1.9943 \cdot 10^{-9}$ & $1.3961 \cdot 10^{-5}$ \\
$8000.0$ & $6.0751 \cdot 10^{-23}$ & $4.8601 \cdot 10^{-19}$ & $3.8882 \cdot 10^{-15}$ & $3.1106 \cdot 10^{-11}$ & $2.4885 \cdot 10^{-7}$ \\
$9000.0$ & $8.0984 \cdot 10^{-25}$ & $7.2889 \cdot 10^{-21}$ & $6.5604 \cdot 10^{-17}$ & $5.9047 \cdot 10^{-13}$ & $5.3145 \cdot 10^{-9}$ \\
$10000$ & $1.3352 \cdot 10^{-26}$ & $1.3353 \cdot 10^{-22}$ & $1.3354 \cdot 10^{-18}$ & $1.3354 \cdot 10^{-14}$ & $1.3355 \cdot 10^{-10}$ \\
\hline
\end{tabular}
\end{table}
\clearpage
\begin{table}[htbp]
\caption{Values of $\mathcal{C}_{\ell}(x_0)$ for $\ell\in\{1,2,3,4,5\}$, in $|\widetilde{\psi}(x)-\log x+\gamma|\leq \dfrac{\mathcal{C}_{\ell}(x_0)}{(\log x)^{\ell}}$, valid for $x\geq x_0$, where $\mathcal{C}_{\ell}(x_0)$ is given by (158).}
\label{Cltable}
\centering
\begin{tabular}{|c|c|c|c|c|c|}
\hline
$\log x_0$ & $\mathcal{C}_1(x_0)$ & $\mathcal{C}_2(x_0)$ & $\mathcal{C}_3(x_0)$ & $\mathcal{C}_4(x_0)$ & $\mathcal{C}_5(x_0)$ \\
\hline
$20.000$ & $9.0197 \cdot 10^{-4}$ & $1.8170 \cdot 10^{-2}$ & $3.6604 \cdot 10^{-1}$ & $2.9373 \cdot 10^{1}$ & $6.0050 \cdot 10^{4}$ \\
$25.000$ & $9.2409 \cdot 10^{-5}$ & $2.3235 \cdot 10^{-3}$ & $5.8425 \cdot 10^{-2}$ & $2.9043 \cdot 10^{1}$ & $6.0044 \cdot 10^{4}$ \\
$30.000$ & $9.0939 \cdot 10^{-6}$ & $2.7412 \cdot 10^{-4}$ & $1.4394 \cdot 10^{-2}$ & $2.8977 \cdot 10^{1}$ & $6.0042 \cdot 10^{4}$ \\
$35.000$ & $8.8332 \cdot 10^{-7}$ & $3.8363 \cdot 10^{-5}$ & $1.4026 \cdot 10^{-2}$ & $2.8966 \cdot 10^{1}$ & $6.0042 \cdot 10^{4}$ \\
$45.000$ & $4.9166 \cdot 10^{-7}$ & $2.2150 \cdot 10^{-5}$ & $1.3973 \cdot 10^{-2}$ & $2.8964 \cdot 10^{1}$ & $6.0042 \cdot 10^{4}$ \\
$50.000$ & $5.6074 \cdot 10^{-8}$ & $6.7416 \cdot 10^{-6}$ & $1.3972 \cdot 10^{-2}$ & $2.8964 \cdot 10^{1}$ & $6.0042 \cdot 10^{4}$ \\
$70.000$ & $3.2515 \cdot 10^{-9}$ & $6.7399 \cdot 10^{-6}$ & $1.3972 \cdot 10^{-2}$ & $2.8964 \cdot 10^{1}$ & $6.0042 \cdot 10^{4}$ \\
$100.00$ & $3.2513 \cdot 10^{-9}$ & $6.7399 \cdot 10^{-6}$ & $1.3972 \cdot 10^{-2}$ & $2.8964 \cdot 10^{1}$ & $6.0042 \cdot 10^{4}$ \\
$200.00$ & $3.2513 \cdot 10^{-9}$ & $6.7399 \cdot 10^{-6}$ & $1.3972 \cdot 10^{-2}$ & $2.8964 \cdot 10^{1}$ & $6.0042 \cdot 10^{4}$ \\
$300.00$ & $3.2513 \cdot 10^{-9}$ & $6.7399 \cdot 10^{-6}$ & $1.3972 \cdot 10^{-2}$ & $2.8964 \cdot 10^{1}$ & $6.0042 \cdot 10^{4}$ \\
$400.00$ & $3.2513 \cdot 10^{-9}$ & $6.7399 \cdot 10^{-6}$ & $1.3972 \cdot 10^{-2}$ & $2.8964 \cdot 10^{1}$ & $6.0042 \cdot 10^{4}$ \\
$500.00$ & $3.2513 \cdot 10^{-9}$ & $6.7399 \cdot 10^{-6}$ & $1.3972 \cdot 10^{-2}$ & $2.8964 \cdot 10^{1}$ & $6.0042 \cdot 10^{4}$ \\
$700.00$ & $3.2513 \cdot 10^{-9}$ & $6.7399 \cdot 10^{-6}$ & $1.3972 \cdot 10^{-2}$ & $2.8964 \cdot 10^{1}$ & $6.0042 \cdot 10^{4}$ \\
$1000.0$ & $3.2513 \cdot 10^{-9}$ & $6.7399 \cdot 10^{-6}$ & $1.3972 \cdot 10^{-2}$ & $2.8964 \cdot 10^{1}$ & $6.0042 \cdot 10^{4}$ \\
$1500.0$ & $3.2513 \cdot 10^{-9}$ & $6.7399 \cdot 10^{-6}$ & $1.3972 \cdot 10^{-2}$ & $2.8964 \cdot 10^{1}$ & $6.0042 \cdot 10^{4}$ \\
$2000.0$ & $3.2513 \cdot 10^{-9}$ & $6.7399 \cdot 10^{-6}$ & $1.3972 \cdot 10^{-2}$ & $2.8964 \cdot 10^{1}$ & $6.0042 \cdot 10^{4}$ \\
$2500.0$ & $2.6726 \cdot 10^{-10}$ & $6.6819 \cdot 10^{-7}$ & $1.6706 \cdot 10^{-3}$ & $4.1769 \cdot 10^{0}$ & $1.0443 \cdot 10^{4}$ \\
$3000.0$ & $1.4524 \cdot 10^{-11}$ & $4.3573 \cdot 10^{-8}$ & $1.3073 \cdot 10^{-4}$ & $3.9221 \cdot 10^{-1}$ & $1.1767 \cdot 10^{3}$ \\
$3500.0$ & $7.8580 \cdot 10^{-13}$ & $2.7505 \cdot 10^{-9}$ & $9.6272 \cdot 10^{-6}$ & $3.3697 \cdot 10^{-2}$ & $1.1795 \cdot 10^{2}$ \\
$4000.0$ & $4.2967 \cdot 10^{-14}$ & $1.7188 \cdot 10^{-10}$ & $6.8754 \cdot 10^{-7}$ & $2.7503 \cdot 10^{-3}$ & $1.1002 \cdot 10^{1}$ \\
$4500.0$ & $2.4148 \cdot 10^{-15}$ & $1.0868 \cdot 10^{-11}$ & $4.8905 \cdot 10^{-8}$ & $2.2008 \cdot 10^{-4}$ & $9.9040 \cdot 10^{-1}$ \\
$5000.0$ & $1.4595 \cdot 10^{-16}$ & $7.2977 \cdot 10^{-13}$ & $3.6490 \cdot 10^{-9}$ & $1.8246 \cdot 10^{-5}$ & $9.1232 \cdot 10^{-2}$ \\
$6000.0$ & $7.5913 \cdot 10^{-19}$ & $4.5549 \cdot 10^{-15}$ & $2.7331 \cdot 10^{-11}$ & $1.6399 \cdot 10^{-7}$ & $9.8396 \cdot 10^{-4}$ \\
$7000.0$ & $5.8137 \cdot 10^{-21}$ & $4.0697 \cdot 10^{-17}$ & $2.8489 \cdot 10^{-13}$ & $1.9943 \cdot 10^{-9}$ & $1.3961 \cdot 10^{-5}$ \\
$8000.0$ & $6.0751 \cdot 10^{-23}$ & $4.8601 \cdot 10^{-19}$ & $3.8882 \cdot 10^{-15}$ & $3.1106 \cdot 10^{-11}$ & $2.4885 \cdot 10^{-7}$ \\
$9000.0$ & $8.0984 \cdot 10^{-25}$ & $7.2889 \cdot 10^{-21}$ & $6.5604 \cdot 10^{-17}$ & $5.9047 \cdot 10^{-13}$ & $5.3145 \cdot 10^{-9}$ \\
$10000$ & $1.3352 \cdot 10^{-26}$ & $1.3353 \cdot 10^{-22}$ & $1.3354 \cdot 10^{-18}$ & $1.3354 \cdot 10^{-14}$ & $1.3355 \cdot 10^{-10}$ \\
\hline
\end{tabular}
\end{table}
\clearpage
\begin{table}[htbp]
\caption{Values for $\mathcal{A}'_{\ell}(x_0)$ for $\ell\in\{1,2,3,4,5\}$, in $\displaystyle \prod_{p\leq x}\left(1-\frac{1}{p}\right)\leq \frac{e^{-\gamma}}{\log x}\left(1+\frac{\mathcal{A}'_{\ell}(x_0)}{(\log x)^{\ell}}\right)$, valid for $x\geq x_0$, where $\mathcal{A}'_{\ell}(x_0)$ is given by (170).}
\label{Aellptable}
\centering
\begin{tabular}{|c|c|c|c|c|c|}
\hline
$\log x_0$ & $\mathcal{A}'_1(x_0)$ & $\mathcal{A}'_2(x_0)$ & $\mathcal{A}'_3(x_0)$ & $\mathcal{A}'_4(x_0)$ & $\mathcal{A}'_5(x_0)$ \\
\hline
$20.000$ & $1.8928 \cdot 10^{-4}$ & $3.7966 \cdot 10^{-3}$ & $7.6156 \cdot 10^{-2}$ & $1.5277 \cdot 10^{0}$ & $4.5160 \cdot 10^{1}$ \\
$25.000$ & $1.5302 \cdot 10^{-5}$ & $3.8345 \cdot 10^{-4}$ & $9.6091 \cdot 10^{-3}$ & $2.4081 \cdot 10^{-1}$ & $3.2114 \cdot 10^{1}$ \\
$30.000$ & $1.2469 \cdot 10^{-6}$ & $3.7480 \cdot 10^{-5}$ & $1.1267 \cdot 10^{-3}$ & $3.3868 \cdot 10^{-2}$ & $2.9492 \cdot 10^{1}$ \\
$35.000$ & $1.0194 \cdot 10^{-7}$ & $3.5739 \cdot 10^{-6}$ & $1.2530 \cdot 10^{-4}$ & $1.6248 \cdot 10^{-2}$ & $2.9044 \cdot 10^{1}$ \\
$45.000$ & $1.1429 \cdot 10^{-8}$ & $5.1487 \cdot 10^{-7}$ & $2.3195 \cdot 10^{-5}$ & $1.4005 \cdot 10^{-2}$ & $2.8965 \cdot 10^{1}$ \\
$50.000$ & $1.1626 \cdot 10^{-9}$ & $5.8187 \cdot 10^{-8}$ & $6.8125 \cdot 10^{-6}$ & $1.3976 \cdot 10^{-2}$ & $2.8964 \cdot 10^{1}$ \\
$70.000$ & $2.3101 \cdot 10^{-12}$ & $3.2516 \cdot 10^{-9}$ & $6.7399 \cdot 10^{-6}$ & $1.3972 \cdot 10^{-2}$ & $2.8964 \cdot 10^{1}$ \\
$100.00$ & $2.0100 \cdot 10^{-12}$ & $3.2513 \cdot 10^{-9}$ & $6.7399 \cdot 10^{-6}$ & $1.3972 \cdot 10^{-2}$ & $2.8964 \cdot 10^{1}$ \\
$200.00$ & $1.7684 \cdot 10^{-12}$ & $3.2513 \cdot 10^{-9}$ & $6.7399 \cdot 10^{-6}$ & $1.3972 \cdot 10^{-2}$ & $2.8964 \cdot 10^{1}$ \\
$300.00$ & $1.6930 \cdot 10^{-12}$ & $3.2513 \cdot 10^{-9}$ & $6.7399 \cdot 10^{-6}$ & $1.3972 \cdot 10^{-2}$ & $2.8964 \cdot 10^{1}$ \\
$400.00$ & $1.6560 \cdot 10^{-12}$ & $3.2513 \cdot 10^{-9}$ & $6.7399 \cdot 10^{-6}$ & $1.3972 \cdot 10^{-2}$ & $2.8964 \cdot 10^{1}$ \\
$500.00$ & $1.6341 \cdot 10^{-12}$ & $3.2513 \cdot 10^{-9}$ & $6.7399 \cdot 10^{-6}$ & $1.3972 \cdot 10^{-2}$ & $2.8964 \cdot 10^{1}$ \\
$700.00$ & $1.6092 \cdot 10^{-12}$ & $3.2513 \cdot 10^{-9}$ & $6.7399 \cdot 10^{-6}$ & $1.3972 \cdot 10^{-2}$ & $2.8964 \cdot 10^{1}$ \\
$1000.0$ & $1.5907 \cdot 10^{-12}$ & $3.2513 \cdot 10^{-9}$ & $6.7399 \cdot 10^{-6}$ & $1.3972 \cdot 10^{-2}$ & $2.8964 \cdot 10^{1}$ \\
$1500.0$ & $1.5763 \cdot 10^{-12}$ & $3.2513 \cdot 10^{-9}$ & $6.7399 \cdot 10^{-6}$ & $1.3972 \cdot 10^{-2}$ & $2.8964 \cdot 10^{1}$ \\
$2000.0$ & $1.5692 \cdot 10^{-12}$ & $3.2513 \cdot 10^{-9}$ & $6.7399 \cdot 10^{-6}$ & $1.3972 \cdot 10^{-2}$ & $2.8964 \cdot 10^{1}$ \\
$2500.0$ & $1.0690 \cdot 10^{-13}$ & $2.6726 \cdot 10^{-10}$ & $6.6819 \cdot 10^{-7}$ & $1.6706 \cdot 10^{-3}$ & $4.1769 \cdot 10^{0}$ \\
$3000.0$ & $4.8408 \cdot 10^{-15}$ & $1.4524 \cdot 10^{-11}$ & $4.3573 \cdot 10^{-8}$ & $1.3073 \cdot 10^{-4}$ & $3.9221 \cdot 10^{-1}$ \\
$3500.0$ & $2.2451 \cdot 10^{-16}$ & $7.8580 \cdot 10^{-13}$ & $2.7505 \cdot 10^{-9}$ & $9.6272 \cdot 10^{-6}$ & $3.3697 \cdot 10^{-2}$ \\
$4000.0$ & $1.0742 \cdot 10^{-17}$ & $4.2967 \cdot 10^{-14}$ & $1.7188 \cdot 10^{-10}$ & $6.8754 \cdot 10^{-7}$ & $2.7503 \cdot 10^{-3}$ \\
$4500.0$ & $5.3660 \cdot 10^{-19}$ & $2.4148 \cdot 10^{-15}$ & $1.0868 \cdot 10^{-11}$ & $4.8905 \cdot 10^{-8}$ & $2.2008 \cdot 10^{-4}$ \\
$5000.0$ & $2.9189 \cdot 10^{-20}$ & $1.4595 \cdot 10^{-16}$ & $7.2977 \cdot 10^{-13}$ & $3.6490 \cdot 10^{-9}$ & $1.8246 \cdot 10^{-5}$ \\
$6000.0$ & $1.2652 \cdot 10^{-22}$ & $7.5913 \cdot 10^{-19}$ & $4.5549 \cdot 10^{-15}$ & $2.7331 \cdot 10^{-11}$ & $1.6399 \cdot 10^{-7}$ \\
$7000.0$ & $8.3050 \cdot 10^{-25}$ & $5.8137 \cdot 10^{-21}$ & $4.0697 \cdot 10^{-17}$ & $2.8489 \cdot 10^{-13}$ & $1.9943 \cdot 10^{-9}$ \\
$8000.0$ & $7.5937 \cdot 10^{-27}$ & $6.0751 \cdot 10^{-23}$ & $4.8601 \cdot 10^{-19}$ & $3.8882 \cdot 10^{-15}$ & $3.1106 \cdot 10^{-11}$ \\
$9000.0$ & $8.9977 \cdot 10^{-29}$ & $8.0984 \cdot 10^{-25}$ & $7.2889 \cdot 10^{-21}$ & $6.5604 \cdot 10^{-17}$ & $5.9047 \cdot 10^{-13}$ \\
$10000$ & $1.3352 \cdot 10^{-30}$ & $1.3352 \cdot 10^{-26}$ & $1.3353 \cdot 10^{-22}$ & $1.3354 \cdot 10^{-18}$ & $1.3354 \cdot 10^{-14}$ \\
\hline
\end{tabular}
\end{table}
\clearpage

\begin{table}[htbp]
\caption{Values for $\mathcal{D}_{\ell}(x_0)$ for $\ell\in\{1,2,3,4,5\}$, in $\displaystyle e^{\gamma}\log x\left(1-\frac{\mathcal{D}_{\ell}(x_0)}{(\log x)^{\ell}}\right)\leq \prod_{p\leq x}\frac{p}{p-1}$, valid for $x\geq x_0$, where $\mathcal{D}_{\ell}(x_0)$ is given by (174).}
\label{Delltable}
\centering
\begin{tabular}{|c|c|c|c|c|c|}
\hline
$\log x_0$ & $\mathcal{D}_1(x_0)$ & $\mathcal{D}_2(x_0)$ & $\mathcal{D}_3(x_0)$ & $\mathcal{D}_4(x_0)$ & $\mathcal{D}_5(x_0)$ \\
\hline
$20.000$ & $1.8927 \cdot 10^{-4}$ & $3.7965 \cdot 10^{-3}$ & $7.6155 \cdot 10^{-2}$ & $1.5276 \cdot 10^{0}$ & $4.5158 \cdot 10^{1}$ \\
$25.000$ & $1.5301 \cdot 10^{-5}$ & $3.8344 \cdot 10^{-4}$ & $9.6090 \cdot 10^{-3}$ & $2.4080 \cdot 10^{-1}$ & $3.2113 \cdot 10^{1}$ \\
$30.000$ & $1.2468 \cdot 10^{-6}$ & $3.7479 \cdot 10^{-5}$ & $1.1266 \cdot 10^{-3}$ & $3.3867 \cdot 10^{-2}$ & $2.9491 \cdot 10^{1}$ \\
$35.000$ & $1.0193 \cdot 10^{-7}$ & $3.5738 \cdot 10^{-6}$ & $1.2529 \cdot 10^{-4}$ & $1.6247 \cdot 10^{-2}$ & $2.9043 \cdot 10^{1}$ \\
$45.000$ & $1.1428 \cdot 10^{-8}$ & $5.1486 \cdot 10^{-7}$ & $2.3194 \cdot 10^{-5}$ & $1.4004 \cdot 10^{-2}$ & $2.8964 \cdot 10^{1}$ \\
$50.000$ & $1.1625 \cdot 10^{-9}$ & $5.8186 \cdot 10^{-8}$ & $6.8124 \cdot 10^{-6}$ & $1.3975 \cdot 10^{-2}$ & $2.8963 \cdot 10^{1}$ \\
$70.000$ & $2.3100 \cdot 10^{-12}$ & $3.2515 \cdot 10^{-9}$ & $6.7398 \cdot 10^{-6}$ & $1.3971 \cdot 10^{-2}$ & $2.8963 \cdot 10^{1}$ \\
$100.00$ & $2.0099 \cdot 10^{-12}$ & $3.2512 \cdot 10^{-9}$ & $6.7398 \cdot 10^{-6}$ & $1.3971 \cdot 10^{-2}$ & $2.8963 \cdot 10^{1}$ \\
$200.00$ & $1.7683 \cdot 10^{-12}$ & $3.2512 \cdot 10^{-9}$ & $6.7398 \cdot 10^{-6}$ & $1.3971 \cdot 10^{-2}$ & $2.8963 \cdot 10^{1}$ \\
$300.00$ & $1.6929 \cdot 10^{-12}$ & $3.2512 \cdot 10^{-9}$ & $6.7398 \cdot 10^{-6}$ & $1.3971 \cdot 10^{-2}$ & $2.8963 \cdot 10^{1}$ \\
$400.00$ & $1.6559 \cdot 10^{-12}$ & $3.2512 \cdot 10^{-9}$ & $6.7398 \cdot 10^{-6}$ & $1.3971 \cdot 10^{-2}$ & $2.8963 \cdot 10^{1}$ \\
$500.00$ & $1.6340 \cdot 10^{-12}$ & $3.2512 \cdot 10^{-9}$ & $6.7398 \cdot 10^{-6}$ & $1.3971 \cdot 10^{-2}$ & $2.8963 \cdot 10^{1}$ \\
$700.00$ & $1.6091 \cdot 10^{-12}$ & $3.2512 \cdot 10^{-9}$ & $6.7398 \cdot 10^{-6}$ & $1.3971 \cdot 10^{-2}$ & $2.8963 \cdot 10^{1}$ \\
$1000.0$ & $1.5906 \cdot 10^{-12}$ & $3.2512 \cdot 10^{-9}$ & $6.7398 \cdot 10^{-6}$ & $1.3971 \cdot 10^{-2}$ & $2.8963 \cdot 10^{1}$ \\
$1500.0$ & $1.5762 \cdot 10^{-12}$ & $3.2512 \cdot 10^{-9}$ & $6.7398 \cdot 10^{-6}$ & $1.3971 \cdot 10^{-2}$ & $2.8963 \cdot 10^{1}$ \\
$2000.0$ & $1.5691 \cdot 10^{-12}$ & $3.2512 \cdot 10^{-9}$ & $6.7398 \cdot 10^{-6}$ & $1.3971 \cdot 10^{-2}$ & $2.8963 \cdot 10^{1}$ \\
$2500.0$ & $1.0689 \cdot 10^{-13}$ & $2.6725 \cdot 10^{-10}$ & $6.6818 \cdot 10^{-7}$ & $1.6705 \cdot 10^{-3}$ & $4.1768 \cdot 10^{0}$ \\
$3000.0$ & $4.8407 \cdot 10^{-15}$ & $1.4523 \cdot 10^{-11}$ & $4.3572 \cdot 10^{-8}$ & $1.3072 \cdot 10^{-4}$ & $3.9220 \cdot 10^{-1}$ \\
$3500.0$ & $2.2450 \cdot 10^{-16}$ & $7.8579 \cdot 10^{-13}$ & $2.7504 \cdot 10^{-9}$ & $9.6271 \cdot 10^{-6}$ & $3.3696 \cdot 10^{-2}$ \\
$4000.0$ & $1.0741 \cdot 10^{-17}$ & $4.2966 \cdot 10^{-14}$ & $1.7187 \cdot 10^{-10}$ & $6.8753 \cdot 10^{-7}$ & $2.7502 \cdot 10^{-3}$ \\
$4500.0$ & $5.3659 \cdot 10^{-19}$ & $2.4147 \cdot 10^{-15}$ & $1.0867 \cdot 10^{-11}$ & $4.8904 \cdot 10^{-8}$ & $2.2007 \cdot 10^{-4}$ \\
$5000.0$ & $2.9188 \cdot 10^{-20}$ & $1.4594 \cdot 10^{-16}$ & $7.2976 \cdot 10^{-13}$ & $3.6489 \cdot 10^{-9}$ & $1.8245 \cdot 10^{-5}$ \\
$6000.0$ & $1.2651 \cdot 10^{-22}$ & $7.5912 \cdot 10^{-19}$ & $4.5548 \cdot 10^{-15}$ & $2.7330 \cdot 10^{-11}$ & $1.6398 \cdot 10^{-7}$ \\
$7000.0$ & $8.3049 \cdot 10^{-25}$ & $5.8136 \cdot 10^{-21}$ & $4.0696 \cdot 10^{-17}$ & $2.8488 \cdot 10^{-13}$ & $1.9942 \cdot 10^{-9}$ \\
$8000.0$ & $7.5936 \cdot 10^{-27}$ & $6.0750 \cdot 10^{-23}$ & $4.8600 \cdot 10^{-19}$ & $3.8881 \cdot 10^{-15}$ & $3.1105 \cdot 10^{-11}$ \\
$9000.0$ & $8.9976 \cdot 10^{-29}$ & $8.0983 \cdot 10^{-25}$ & $7.2888 \cdot 10^{-21}$ & $6.5603 \cdot 10^{-17}$ & $5.9046 \cdot 10^{-13}$ \\
$10000$ & $1.3351 \cdot 10^{-30}$ & $1.3351 \cdot 10^{-26}$ & $1.3352 \cdot 10^{-22}$ & $1.3353 \cdot 10^{-18}$ & $1.3353 \cdot 10^{-14}$ \\
\hline
\end{tabular}
\end{table}

\clearpage
\begin{table}[htbp]
\caption{Values for $\mathcal{D}'_{\ell}(x_0)$ for $\ell\in\{1,2,3,4,5\}$, in $\displaystyle \prod_{p\leq x}\frac{p}{p-1}\leq e^{\gamma}\log x\left(1+\frac{\mathcal{D}'_{\ell}(x_0)}{(\log x)^{\ell}}\right)$, valid for $x\geq x_0$, where $\mathcal{D}'_{\ell}(x_0)$ is given by (175).}
\label{Dellptable}
\centering
\begin{tabular}{|c|c|c|c|c|c|}
\hline
$\log x_0$ & $\mathcal{D}'_1(x_0)$ & $\mathcal{D}'_2(x_0)$ & $\mathcal{D}'_3(x_0)$ & $\mathcal{D}'_4(x_0)$ & $\mathcal{D}'_5(x_0)$ \\
\hline
$20.000$ & $1.8928 \cdot 10^{-4}$ & $3.7966 \cdot 10^{-3}$ & $7.6155 \cdot 10^{-2}$ & $1.5277 \cdot 10^{0}$ & $4.5159 \cdot 10^{1}$ \\
$25.000$ & $1.5302 \cdot 10^{-5}$ & $3.8345 \cdot 10^{-4}$ & $9.6091 \cdot 10^{-3}$ & $2.4081 \cdot 10^{-1}$ & $3.2114 \cdot 10^{1}$ \\
$30.000$ & $1.2469 \cdot 10^{-6}$ & $3.7480 \cdot 10^{-5}$ & $1.1267 \cdot 10^{-3}$ & $3.3868 \cdot 10^{-2}$ & $2.9492 \cdot 10^{1}$ \\
$35.000$ & $1.0194 \cdot 10^{-7}$ & $3.5739 \cdot 10^{-6}$ & $1.2530 \cdot 10^{-4}$ & $1.6248 \cdot 10^{-2}$ & $2.9044 \cdot 10^{1}$ \\
$45.000$ & $1.1429 \cdot 10^{-8}$ & $5.1487 \cdot 10^{-7}$ & $2.3195 \cdot 10^{-5}$ & $1.4005 \cdot 10^{-2}$ & $2.8965 \cdot 10^{1}$ \\
$50.000$ & $1.1626 \cdot 10^{-9}$ & $5.8187 \cdot 10^{-8}$ & $6.8125 \cdot 10^{-6}$ & $1.3976 \cdot 10^{-2}$ & $2.8964 \cdot 10^{1}$ \\
$70.000$ & $2.3101 \cdot 10^{-12}$ & $3.2516 \cdot 10^{-9}$ & $6.7399 \cdot 10^{-6}$ & $1.3972 \cdot 10^{-2}$ & $2.8964 \cdot 10^{1}$ \\
$100.00$ & $2.0100 \cdot 10^{-12}$ & $3.2513 \cdot 10^{-9}$ & $6.7399 \cdot 10^{-6}$ & $1.3972 \cdot 10^{-2}$ & $2.8964 \cdot 10^{1}$ \\
$200.00$ & $1.7684 \cdot 10^{-12}$ & $3.2513 \cdot 10^{-9}$ & $6.7399 \cdot 10^{-6}$ & $1.3972 \cdot 10^{-2}$ & $2.8964 \cdot 10^{1}$ \\
$300.00$ & $1.6930 \cdot 10^{-12}$ & $3.2513 \cdot 10^{-9}$ & $6.7399 \cdot 10^{-6}$ & $1.3972 \cdot 10^{-2}$ & $2.8964 \cdot 10^{1}$ \\
$400.00$ & $1.6560 \cdot 10^{-12}$ & $3.2513 \cdot 10^{-9}$ & $6.7399 \cdot 10^{-6}$ & $1.3972 \cdot 10^{-2}$ & $2.8964 \cdot 10^{1}$ \\
$500.00$ & $1.6341 \cdot 10^{-12}$ & $3.2513 \cdot 10^{-9}$ & $6.7399 \cdot 10^{-6}$ & $1.3972 \cdot 10^{-2}$ & $2.8964 \cdot 10^{1}$ \\
$700.00$ & $1.6092 \cdot 10^{-12}$ & $3.2513 \cdot 10^{-9}$ & $6.7399 \cdot 10^{-6}$ & $1.3972 \cdot 10^{-2}$ & $2.8964 \cdot 10^{1}$ \\
$1000.0$ & $1.5907 \cdot 10^{-12}$ & $3.2513 \cdot 10^{-9}$ & $6.7399 \cdot 10^{-6}$ & $1.3972 \cdot 10^{-2}$ & $2.8964 \cdot 10^{1}$ \\
$1500.0$ & $1.5763 \cdot 10^{-12}$ & $3.2513 \cdot 10^{-9}$ & $6.7399 \cdot 10^{-6}$ & $1.3972 \cdot 10^{-2}$ & $2.8964 \cdot 10^{1}$ \\
$2000.0$ & $1.5692 \cdot 10^{-12}$ & $3.2513 \cdot 10^{-9}$ & $6.7399 \cdot 10^{-6}$ & $1.3972 \cdot 10^{-2}$ & $2.8964 \cdot 10^{1}$ \\
$2500.0$ & $1.0690 \cdot 10^{-13}$ & $2.6726 \cdot 10^{-10}$ & $6.6819 \cdot 10^{-7}$ & $1.6706 \cdot 10^{-3}$ & $4.1769 \cdot 10^{0}$ \\
$3000.0$ & $4.8408 \cdot 10^{-15}$ & $1.4524 \cdot 10^{-11}$ & $4.3573 \cdot 10^{-8}$ & $1.3073 \cdot 10^{-4}$ & $3.9221 \cdot 10^{-1}$ \\
$3500.0$ & $2.2451 \cdot 10^{-16}$ & $7.8580 \cdot 10^{-13}$ & $2.7505 \cdot 10^{-9}$ & $9.6272 \cdot 10^{-6}$ & $3.3697 \cdot 10^{-2}$ \\
$4000.0$ & $1.0742 \cdot 10^{-17}$ & $4.2967 \cdot 10^{-14}$ & $1.7188 \cdot 10^{-10}$ & $6.8754 \cdot 10^{-7}$ & $2.7503 \cdot 10^{-3}$ \\
$4500.0$ & $5.3660 \cdot 10^{-19}$ & $2.4148 \cdot 10^{-15}$ & $1.0868 \cdot 10^{-11}$ & $4.8905 \cdot 10^{-8}$ & $2.2008 \cdot 10^{-4}$ \\
$5000.0$ & $2.9189 \cdot 10^{-20}$ & $1.4595 \cdot 10^{-16}$ & $7.2977 \cdot 10^{-13}$ & $3.6490 \cdot 10^{-9}$ & $1.8246 \cdot 10^{-5}$ \\
$6000.0$ & $1.2652 \cdot 10^{-22}$ & $7.5913 \cdot 10^{-19}$ & $4.5549 \cdot 10^{-15}$ & $2.7331 \cdot 10^{-11}$ & $1.6399 \cdot 10^{-7}$ \\
$7000.0$ & $8.3050 \cdot 10^{-25}$ & $5.8137 \cdot 10^{-21}$ & $4.0697 \cdot 10^{-17}$ & $2.8489 \cdot 10^{-13}$ & $1.9943 \cdot 10^{-9}$ \\
$8000.0$ & $7.5937 \cdot 10^{-27}$ & $6.0751 \cdot 10^{-23}$ & $4.8601 \cdot 10^{-19}$ & $3.8882 \cdot 10^{-15}$ & $3.1106 \cdot 10^{-11}$ \\
$9000.0$ & $8.9977 \cdot 10^{-29}$ & $8.0984 \cdot 10^{-25}$ & $7.2889 \cdot 10^{-21}$ & $6.5604 \cdot 10^{-17}$ & $5.9047 \cdot 10^{-13}$ \\
$10000$ & $1.3352 \cdot 10^{-30}$ & $1.3352 \cdot 10^{-26}$ & $1.3353 \cdot 10^{-22}$ & $1.3354 \cdot 10^{-18}$ & $1.3354 \cdot 10^{-14}$ \\
\hline
\end{tabular}
\end{table}
\clearpage

\begin{table}[h!]
   \caption{
   $|\lambda(x) - \log \log x - M | \le
  A_{\lambda}(x_0) (\log x)^{3/2}
 \exp(- C \sqrt{\log x})$ for all $x \ge x_0$ where
 $ A_{\lambda}(x_0)-A_{\vartheta}(x_0)=A''(x_0)\exp((-C(\sqrt{2}- 1)) \sqrt{\log x_0})$, and $C = 0.84768 $.
 Below are table of values of $A''(x_0)= A''(x_0,\sigma)$ where $ A''(x_0,\sigma)$ is defined in \eqref{Appsigma}.}
 \label{Apptable}
\centering
\begin{tabular}{|c|c|c|}
\hline
$\log x_0$ & $\sigma$ & $A''(x_0)=A''(x_0,\sigma)$ \\
\hline
200 & 0.65 & $3.1007\cdot 10^{-12}$ \\
300 & 0.70 & $3.8278\cdot 10^{-13}$ \\
400 & 0.70 & $1.3869\cdot 10^{-13}$ \\
500 & 0.80 & $7.5214\cdot 10^{-14}$ \\
700 & 0.80 & $3.7423\cdot 10^{-14}$ \\
1000 & 0.80 & $2.1377\cdot 10^{-14}$ \\
1500 & 0.80 & $1.1242\cdot 10^{-14}$ \\
2000 & 0.80 & $5.8062\cdot 10^{-15}$ \\
2500 & 0.80 & $2.6915\cdot 10^{-15}$ \\
3000 & 0.80 & $1.1101\cdot 10^{-15}$ \\
3500 & 0.80 & $4.1068\cdot 10^{-16}$ \\
4000 & 0.80 & $1.3787\cdot 10^{-16}$ \\
4500 & 0.80 & $4.2464\cdot 10^{-17}$ \\
5000 & 0.80 & $1.2118\cdot 10^{-17}$ \\
6000 & 0.80 & $8.1148\cdot 10^{-19}$ \\
7000 & 0.80 & $4.3876\cdot 10^{-20}$ \\
8000 & 0.80 & $2.0130\cdot 10^{-21}$ \\
9000 & 0.80 & $8.2488\cdot 10^{-23}$ \\
10000 & 0.80 & $3.6840\cdot 10^{-24}$ \\
200000 & 0.90 & $1.0903\cdot 10^{-114}$ \\
300000 & 0.90 & $7.2925\cdot 10^{-141}$ \\
400000 & 0.90 & $6.2924\cdot 10^{-163}$ \\
500000 & 0.90 & $2.2967\cdot 10^{-182}$ \\
600000 & 0.90 & $6.1476\cdot 10^{-200}$ \\
700000 & 0.90 & $4.2624\cdot 10^{-216}$ \\
800000 & 0.90 & $3.8858\cdot 10^{-231}$ \\
900000 & 0.90 & $2.9084\cdot 10^{-245}$ \\
1000000 & 0.90 & $1.2685\cdot 10^{-258}$ \\
10000000 & 0.90 & $1.5790\cdot 10^{-821}$ \\
100000000 & 0.90 & $1.6128\cdot 10^{-2601}$ \\
1000000000 & 0.90 & $2.6127\cdot 10^{-8230}$ \\
\hline
\end{tabular}
\end{table}

\begin{table}[h!]
\caption{$
|\Upsilon(x) - \log x +M' | \le
  A_{\Upsilon}(x_0) (\log x)^{3/2}
 \exp(-0.8746 \sqrt{\log x})$ for all  $x \ge x_0$
 where  $A_{\Upsilon}(x_0)   = A_{\vartheta}(x_0) +   D'(x_0)  \exp((-C(\sqrt{2}- 1)) \sqrt{\log x_0}))$, and $C = 0.84768 $.
Below are table of values of $ D'(x_0)=D'(x_0,\sigma)$ where $ D'(x_0,\sigma)$ defined in \eqref{Dpsigma}.
}
 \label{Dptable}
\centering

\begin{tabular}{|c|c|c|}
\hline
$\log x_0$ & $\sigma$ & $D'(x_0)=D'(x_0,\sigma)$ \\
\hline
50 & 0.65 & 0.0014464 \\
70 & 0.70 & $1.4882\cdot 10^{-6}$ \\
100 & 0.70 & $6.4187\cdot 10^{-9}$ \\
200 & 0.80 & $2.6925\cdot 10^{-12}$ \\
300 & 0.80 & $3.8151\cdot 10^{-13}$ \\
400 & 0.80 & $1.3835\cdot 10^{-13}$ \\
500 & 0.80 & $7.5064\cdot 10^{-14}$ \\
700 & 0.90 & $3.7370\cdot 10^{-14}$ \\
1000 & 0.90 & $2.1356\cdot 10^{-14}$ \\
1500 & 0.90 & $1.1235\cdot 10^{-14}$ \\
2000 & 0.90 & $5.8032\cdot 10^{-15}$ \\
2500 & 0.90 & $2.6904\cdot 10^{-15}$ \\
3000 & 0.90 & $1.1097\cdot 10^{-15}$ \\
3500 & 0.90 & $4.1057\cdot 10^{-16}$ \\
4000 & 0.90 & $1.3784\cdot 10^{-16}$ \\
4500 & 0.90 & $4.2454\cdot 10^{-17}$ \\
5000 & 0.90 & $1.2115\cdot 10^{-17}$ \\
6000 & 0.90 & $8.1135\cdot 10^{-19}$ \\
7000 & 0.90 & $4.3870\cdot 10^{-20}$ \\
8000 & 0.90 & $2.0127\cdot 10^{-21}$ \\
9000 & 0.90 & $8.2479\cdot 10^{-23}$ \\
10000 & 0.90 & $3.6836\cdot 10^{-24}$ \\
200000 & 0.90 & $1.0903\cdot 10^{-114}$ \\
300000 & 0.90 & $7.2925\cdot 10^{-141}$ \\
400000 & 0.90 & $6.2923\cdot 10^{-163}$ \\
500000 & 0.90 & $2.2967\cdot 10^{-182}$ \\
600000 & 0.90 & $6.1475\cdot 10^{-200}$ \\
700000 & 0.90 & $4.2624\cdot 10^{-216}$ \\
800000 & 0.90 & $3.8858\cdot 10^{-231}$ \\
900000 & 0.90 & $2.9084\cdot 10^{-245}$ \\
1000000 & 0.90 & $1.2685\cdot 10^{-258}$ \\
10000000 & 0.90 & $1.5790\cdot 10^{-821}$ \\
100000000 & 0.90 & $1.6128\cdot 10^{-2601}$ \\
1000000000 & 0.90 & $2.6127\cdot 10^{-8230}$ \\
\hline
\end{tabular}
\end{table}

\begin{table}[h!]
\caption{$| \widetilde{\psi}(x) - \log x - \gamma | \le
  A_{\widetilde{\psi}}(x_0)(\log x)^{1/2}
 \exp(-0.8746 \sqrt{\log x})$ or all  $x \ge x_0$
 where $A_{\widetilde{\psi}}(x_0)  = A_{\psi}(x_0) +   D''(x_0)  \exp((-C(\sqrt{2}- 1)) \sqrt{\log x_0}))$, and $C = 0.84768 $.
Below are table of values of $ D''(x_0)=D''(x_0,\sigma)$ where $ D''(x_0,\sigma)$ defined in \eqref{Dppsigma}.}
 \label{Dpptable}
\centering
\begin{tabular}{|c|c|c|}
\hline
$\log x_0$ & $\sigma$ & $D''(x_0)=D''(x_0,\sigma)$ \\
\hline
50 & 0.65 & 0.0014464 \\
70 & 0.70 & $1.4882\cdot 10^{-6}$ \\
100 & 0.70 & $6.4187\cdot 10^{-9}$ \\
200 & 0.80 & $2.6925\cdot 10^{-12}$ \\
300 & 0.80 & $3.8151\cdot 10^{-13}$ \\
400 & 0.80 & $1.3835\cdot 10^{-13}$ \\
500 & 0.80 & $7.5064\cdot 10^{-14}$ \\
700 & 0.90 & $3.7369\cdot 10^{-14}$ \\
1000 & 0.90 & $2.1356\cdot 10^{-14}$ \\
1500 & 0.90 & $1.1235\cdot 10^{-14}$ \\
2000 & 0.90 & $5.8032\cdot 10^{-15}$ \\
2500 & 0.90 & $2.6904\cdot 10^{-15}$ \\
3000 & 0.90 & $1.1097\cdot 10^{-15}$ \\
3500 & 0.90 & $4.1057\cdot 10^{-16}$ \\
4000 & 0.90 & $1.3784\cdot 10^{-16}$ \\
4500 & 0.90 & $4.2454\cdot 10^{-17}$ \\
5000 & 0.90 & $1.2115\cdot 10^{-17}$ \\
6000 & 0.90 & $8.1135\cdot 10^{-19}$ \\
7000 & 0.90 & $4.3870\cdot 10^{-20}$ \\
8000 & 0.90 & $2.0127\cdot 10^{-21}$ \\
9000 & 0.90 & $8.2479\cdot 10^{-23}$ \\
10000 & 0.90 & $3.6836\cdot 10^{-24}$ \\
200000 & 0.90 & $1.0903\cdot 10^{-114}$ \\
300000 & 0.90 & $7.2925\cdot 10^{-141}$ \\
400000 & 0.90 & $6.2923\cdot 10^{-163}$ \\
500000 & 0.90 & $2.2967\cdot 10^{-182}$ \\
600000 & 0.90 & $6.1475\cdot 10^{-200}$ \\
700000 & 0.90 & $4.2624\cdot 10^{-216}$ \\
800000 & 0.90 & $3.8858\cdot 10^{-231}$ \\
900000 & 0.90 & $2.9084\cdot 10^{-245}$ \\
1000000 & 0.90 & $1.2685\cdot 10^{-258}$ \\
10000000 & 0.90 & $1.5790\cdot 10^{-821}$ \\
100000000 & 0.90 & $1.6128\cdot 10^{-2601}$ \\
1000000000 & 0.90 & $2.6127\cdot 10^{-8230}$ \\
\hline
\end{tabular}
\end{table}

\begin{table}[h]
\caption{Tight Bounds on  $\psi(x)$, $\theta(x)$, and the Mertens sums \eqref{Mertens1},\eqref{Mertens2}, and \eqref{Mertens3}. For all $e^{b_1} \leq x \leq e^{b_2}$ we have respectively the bounds
                $\frac{\ell}{\sqrt{x}} < x -\psi(x) < \frac{L}{\sqrt{x}}$,   $\frac{\ell}{\sqrt{x}} < x -\theta(x) < \frac{L}{\sqrt{x}}$,
                 $\frac{\ell}{\log(x)\sqrt{x}} < \log(\log(x)) + M -\lambda(x) < \frac{L}{\log(x)\sqrt{x}}$,
                  $\frac{\ell}{\log(x)\sqrt{x}} < \log(x) - M' - \Upsilon(x) < \frac{L}{\log(x)\sqrt{x}}$,
                  and
                    $\frac{\ell}{\log(x)\sqrt{x}} < \log(x) -\gamma - \widetilde{\psi}(x) < \frac{L}{\log(x)\sqrt{x}}$.}
\begin{tabular}{|cc|cc|cc|cc|cc|cc|}
\hline
&&\multicolumn{2}{c|}{$\psi$}&\multicolumn{2}{c|}{$\theta$}&\multicolumn{2}{c|}{$\lambda $}&\multicolumn{2}{c|}{$ \Upsilon$}&\multicolumn{2}{c|}{$\widetilde{\psi}$}\\
\hline
$b_1$ & $b_2$ &  $\ell$ & L & $\ell$ & L & $\ell$ & L&  $\ell$ & L &  $\ell$ & L \\
\hline
1 &    2 &      0.362 &      1.415 &     0.624 &      1.435 &    -1.282 &     -0.162 &    -1.851 &     -0.293 &    -0.393 &      0.426 \\
   2 &    3 &     -0.061 &      0.956 &     0.668 &      1.705 &    -1.464 &     -0.318 &    -1.908 &     -0.820 &    -0.470 &      0.424 \\
   3 &    4 &     -0.164 &      0.779 &     0.787 &      1.816 &    -1.424 &     -0.423 &    -1.837 &     -0.814 &    -0.482 &      0.474 \\
   4 &    5 &     -0.413 &      0.765 &     0.557 &      1.832 &    -1.756 &     -0.423 &    -2.062 &     -0.770 &    -0.582 &      0.582 \\
   5 &    6 &     -0.507 &      0.771 &     0.692 &      1.948 &    -1.624 &     -0.367 &    -1.886 &     -0.639 &    -0.641 &      0.644 \\
   6 &    7 &     -0.464 &      0.592 &     0.796 &      1.902 &    -1.485 &     -0.380 &    -1.731 &     -0.619 &    -0.541 &      0.509 \\
   7 &    8 &     -0.550 &      0.804 &     0.618 &      2.053 &    -1.626 &     -0.196 &    -1.837 &     -0.411 &    -0.586 &      0.759 \\
   8 &    9 &     -0.519 &      0.673 &     0.723 &      1.898 &    -1.474 &     -0.333 &    -1.654 &     -0.526 &    -0.538 &      0.639 \\
   9 &   10 &     -0.530 &      0.716 &     0.646 &      1.938 &    -1.511 &     -0.220 &    -1.682 &     -0.381 &    -0.549 &      0.696 \\
  10 &   11 &     -0.711 &      0.637 &     0.456 &      1.818 &    -1.688 &     -0.306 &    -1.847 &     -0.462 &    -0.722 &      0.634 \\
  11 &   12 &     -0.582 &      0.609 &     0.572 &      1.769 &    -1.535 &     -0.347 &    -1.682 &     -0.491 &    -0.595 &      0.597 \\
  12 &   13 &     -0.681 &      0.712 &     0.429 &      1.833 &    -1.643 &     -0.230 &    -1.771 &     -0.360 &    -0.688 &      0.715 \\
  13 &   14 &     -0.616 &      0.639 &     0.497 &      1.772 &    -1.534 &     -0.286 &    -1.653 &     -0.407 &    -0.609 &      0.645 \\
  14 &   15 &     -0.597 &      0.531 &     0.492 &      1.642 &    -1.529 &     -0.385 &    -1.645 &     -0.500 &    -0.601 &      0.534 \\
  15 &   16 &     -0.676 &      0.618 &     0.407 &      1.707 &    -1.626 &     -0.320 &    -1.736 &     -0.430 &    -0.690 &      0.609 \\
  16 &   17 &     -0.660 &      0.612 &     0.411 &      1.679 &    -1.602 &     -0.326 &    -1.704 &     -0.430 &    -0.670 &      0.611 \\
  17 &   18 &     -0.716 &      0.751 &     0.344 &      1.808 &    -1.647 &     -0.182 &    -1.745 &     -0.278 &    -0.716 &      0.752 \\
  18 &   19 &     -0.791 &      0.568 &     0.260 &      1.616 &    -1.717 &     -0.369 &    -1.809 &     -0.460 &    -0.788 &      0.562 \\
  19 &   20 &     -0.728 &      0.581 &     0.316 &      1.616 &    -1.658 &     -0.342 &    -1.748 &     -0.429 &    -0.726 &      0.594 \\
  20 &   21 &     -0.683 &      0.678 &     0.349 &      1.716 &    -1.625 &     -0.269 &    -1.707 &     -0.352 &    -0.691 &      0.667 \\
  21 &   22 &     -0.592 &      0.590 &     0.435 &      1.619 &    -1.523 &     -0.334 &    -1.603 &     -0.413 &    -0.588 &      0.599 \\
  22 &   23 &     -0.607 &      0.689 &     0.420 &      1.713 &    -1.530 &     -0.250 &    -1.608 &     -0.326 &    -0.596 &      0.685 \\
  23 &   24 &     -0.639 &      0.725 &     0.383 &      1.747 &    -1.580 &     -0.212 &    -1.655 &     -0.287 &    -0.645 &      0.723 \\
  24 &   25 &     -0.614 &      0.671 &     0.404 &      1.688 &    -1.557 &     -0.253 &    -1.628 &     -0.325 &    -0.619 &      0.683 \\
  25 &   26 &     -0.680 &      0.736 &     0.335 &      1.751 &    -1.629 &     -0.208 &    -1.698 &     -0.276 &    -0.690 &      0.730 \\
  26 &   27 &     -0.754 &      0.657 &     0.258 &      1.669 &    -1.685 &     -0.285 &    -1.752 &     -0.351 &    -0.745 &      0.655 \\
  27 &   28 &     -0.626 &      0.637 &     0.385 &      1.647 &    -1.560 &     -0.306 &    -1.626 &     -0.370 &    -0.620 &      0.635 \\
  28 &   29 &     -0.681 &      0.690 &     0.328 &      1.699 &    -1.622 &     -0.259 &    -1.685 &     -0.321 &    -0.680 &      0.683 \\
  \hline
\end{tabular}
\label{GRHtable}
\end{table}

\begin{table}[h]
\caption{Tight Bounds on  $\psi(x)$, $\theta(x)$, and the Mertens sums \eqref{Mertens1},\eqref{Mertens2}, and \eqref{Mertens3}. For all $e^{b_1} \leq x \leq e^{b_2}$ we have respectively the bounds
                $\frac{\ell\log(\log(\log(x)))^2}{\sqrt{x}} < x -\psi(x) < \frac{L\log(\log(\log(x)))^2}{\sqrt{x}}$,   $\frac{\ell\log(\log(\log(x)))^2}{\sqrt{x}} < x -\theta(x) < \frac{L\log(\log(\log(x)))^2}{\sqrt{x}}$,
                 $\frac{\ell\log(\log(\log(x)))^2}{\log(x)\sqrt{x}} < \log(\log(x)) + M - \lambda(x) < \frac{L\log(\log(\log(x)))^2}{\log(x)\sqrt{x}}$,
                  $\frac{\ell\log(\log(\log(x)))^2}{\log(x)\sqrt{x}} < \log(x) - M' - \Upsilon(x)< \frac{L\log(\log(\log(x)))^2}{\log(x)\sqrt{x}}$,
                  and
                    $\frac{\ell\log(\log(\log(x)))^2}{\log(x)\sqrt{x}} < \log(x) -\gamma - \widetilde{\psi}(x) < \frac{L\log(\log(\log(x)))^2}{\log(x)\sqrt{x}}$.
}
\label{Conjtable}
\begin{tabular}{|cc|cc|cc|cc|cc|cc|}
\hline
&&\multicolumn{2}{c|}{$\psi$}&\multicolumn{2}{c|}{$\theta$}&\multicolumn{2}{c|}{$\lambda $}&\multicolumn{2}{c|}{$ \Upsilon $}&\multicolumn{2}{c|}{$\widetilde{\psi}$}\\
\hline
$b_1$ & $b_2$ &  $\ell$ & L & $\ell$ & L & $\ell$ & L&  $\ell$ & L &  $\ell$ & L \\
\hline
4 &    5 &     -2.128 &      6.230 &     2.875 &     15.809 &   -10.376 &     -2.383 &   -13.262 &     -3.661 &    -3.184 &      4.197 \\
   5 &    6 &     -1.943 &      2.812 &     2.301 &      8.071 &    -6.040 &     -1.293 &    -7.105 &     -2.073 &    -2.457 &      2.349 \\
   6 &    7 &     -1.286 &      1.579 &     2.041 &      5.147 &    -4.172 &     -1.015 &    -4.863 &     -1.654 &    -1.520 &      1.357 \\
   7 &    8 &     -1.144 &      1.716 &     1.286 &      4.384 &    -3.429 &     -0.420 &    -3.894 &     -0.878 &    -1.219 &      1.620 \\
   8 &    9 &     -0.864 &      1.235 &     1.205 &      3.483 &    -2.509 &     -0.605 &    -2.837 &     -0.899 &    -0.896 &      1.172 \\
   9 &   10 &     -0.822 &      1.098 &     0.963 &      2.963 &    -2.365 &     -0.321 &    -2.632 &     -0.556 &    -0.859 &      1.064 \\
  10 &   11 &     -1.012 &      0.882 &     0.601 &      2.517 &    -2.404 &     -0.425 &    -2.630 &     -0.640 &    -1.028 &      0.878 \\
  11 &   12 &     -0.760 &      0.785 &     0.693 &      2.282 &    -2.006 &     -0.448 &    -2.198 &     -0.634 &    -0.777 &      0.770 \\
  12 &   13 &     -0.778 &      0.822 &     0.491 &      2.118 &    -1.878 &     -0.265 &    -2.024 &     -0.416 &    -0.786 &      0.825 \\
  13 &   14 &     -0.654 &      0.718 &     0.528 &      1.991 &    -1.671 &     -0.308 &    -1.807 &     -0.435 &    -0.647 &      0.718 \\
  14 &   15 &     -0.623 &      0.552 &     0.496 &      1.707 &    -1.603 &     -0.400 &    -1.725 &     -0.518 &    -0.631 &      0.555 \\
  15 &   16 &     -0.679 &      0.620 &     0.409 &      1.712 &    -1.634 &     -0.321 &    -1.745 &     -0.431 &    -0.694 &      0.611 \\
  16 &   17 &     -0.624 &      0.585 &     0.389 &      1.604 &    -1.516 &     -0.312 &    -1.612 &     -0.411 &    -0.634 &      0.584 \\
  17 &   18 &     -0.654 &      0.681 &     0.314 &      1.639 &    -1.504 &     -0.165 &    -1.593 &     -0.252 &    -0.654 &      0.682 \\
  18 &   19 &     -0.690 &      0.494 &     0.226 &      1.403 &    -1.497 &     -0.321 &    -1.577 &     -0.400 &    -0.687 &      0.488 \\
  19 &   20 &     -0.624 &      0.486 &     0.271 &      1.353 &    -1.422 &     -0.286 &    -1.499 &     -0.359 &    -0.623 &      0.497 \\
  20 &   21 &     -0.554 &      0.563 &     0.283 &      1.423 &    -1.317 &     -0.219 &    -1.384 &     -0.287 &    -0.560 &      0.550 \\
  21 &   22 &     -0.469 &      0.464 &     0.344 &      1.273 &    -1.206 &     -0.262 &    -1.270 &     -0.324 &    -0.466 &      0.471 \\
  22 &   23 &     -0.476 &      0.531 &     0.329 &      1.319 &    -1.200 &     -0.193 &    -1.262 &     -0.251 &    -0.468 &      0.528 \\
  23 &   24 &     -0.489 &      0.553 &     0.293 &      1.333 &    -1.209 &     -0.162 &    -1.267 &     -0.219 &    -0.494 &      0.552 \\
  24 &   25 &     -0.456 &      0.495 &     0.299 &      1.245 &    -1.155 &     -0.187 &    -1.208 &     -0.239 &    -0.460 &      0.504 \\
  25 &   26 &     -0.494 &      0.534 &     0.243 &      1.270 &    -1.183 &     -0.151 &    -1.233 &     -0.200 &    -0.501 &      0.530 \\
  26 &   27 &     -0.535 &      0.465 &     0.183 &      1.180 &    -1.195 &     -0.201 &    -1.243 &     -0.248 &    -0.529 &      0.463 \\
  27 &   28 &     -0.440 &      0.442 &     0.270 &      1.144 &    -1.097 &     -0.212 &    -1.143 &     -0.257 &    -0.436 &      0.441 \\
  28 &   29 &     -0.467 &      0.475 &     0.225 &      1.171 &    -1.113 &     -0.178 &    -1.156 &     -0.221 &    -0.467 &      0.471 \\
  \hline
\end{tabular}
\end{table}

\begin{landscape}
\begin{table}[h]
\setlength{\tabcolsep}{2pt}
\caption{ For all $e^{b_1} \leq x \leq e^{b_2}$ we have $\frac{m_kx}{(\log(x)^k} < x -\theta(x) < \frac{M_kx}{(\log(x)^k}$. Note in the given range $x-\theta(x)$ is always positive.}
                \footnotesize
\begin{tabular}{cccccccc}
 $b_1$ & $b_2$ & $(m_0,M_0)$ & $(m_1,M_1)$& $(m_2,M_2)$& $(m_3,M_3)$& $(m_4,M_4)$& $(m_5,M_5)$ \\
 $   1$ & $   2$& $(2.44\cdot 10^{-2},2.53)$& $(2.93\cdot 10^{-2},5.05)$& $(3.51\cdot 10^{-2},10.10)$& $(4.21\cdot 10^{-2},20.20)$& $(5.06\cdot 10^{-2},43.40)$& $(6.07\cdot 10^{-2},146.00)$\\
$   2$ & $   3$& $(9.75\cdot 10^{-2},2.73)$& $(4.18\cdot 10^{-1},6.97)$& $(1.79,20.60)$& $(7.28,61.60)$& $(27.40,185.00)$& $(80.70,554.00)$\\
$   3$ & $   4$& $(5.24\cdot 10^{-2},2.28)$& $(2.48\cdot 10^{-1},6.84)$& $(1.17,21.00)$& $(5.54,71.90)$& $(26.20,271.00)$& $(123.00,1080.00)$\\
$   4$ & $   5$& $(8.88\cdot 10^{-2},1.05)$& $(4.38\cdot 10^{-1},4.20)$& $(2.16,16.80)$& $(10.60,67.20)$& $(52.70,269.00)$& $(260.00,1090.00)$\\
$   5$ & $   6$& $(2.10\cdot 10^{-2},1.31\cdot 10^{-1})$& $(1.07\cdot 10^{-1},7.06\cdot 10^{-1})$& $(5.47\cdot 10^{-1},3.82)$& $(2.79,20.70)$& $(14.20,120.00)$& $(72.50,699.00)$\\
$   6$ & $   7$& $(2.75\cdot 10^{-2},8.66\cdot 10^{-2})$& $(1.91\cdot 10^{-1},5.23\cdot 10^{-1})$& $(1.31,3.23)$& $(8.54,20.40)$& $(52.60,129.00)$& $(323.00,831.00)$\\
$   7$ & $   8$& $(1.35\cdot 10^{-2},5.45\cdot 10^{-2})$& $(1.07\cdot 10^{-1},3.96\cdot 10^{-1})$& $(8.38\cdot 10^{-1},2.87)$& $(6.20,20.90)$& $(45.80,152.00)$& $(339.00,1130.00)$\\
$   8$ & $   9$& $(9.06\cdot 10^{-3},3.31\cdot 10^{-2})$& $(7.93\cdot 10^{-2},2.68\cdot 10^{-1})$& $(6.95\cdot 10^{-1},2.17)$& $(6.09,17.60)$& $(53.30,143.00)$& $(467.00,1200.00)$\\
$   9$ & $  10$& $(5.14\cdot 10^{-3},1.89\cdot 10^{-2})$& $(4.97\cdot 10^{-2},1.72\cdot 10^{-1})$& $(4.80\cdot 10^{-1},1.56)$& $(4.64,14.60)$& $(44.90,137.00)$& $(434.00,1310.00)$\\
$  10$ & $  11$& $(1.88\cdot 10^{-3},1.02\cdot 10^{-2})$& $(2.06\cdot 10^{-2},1.06\cdot 10^{-1})$& $(2.27\cdot 10^{-1},1.10)$& $(2.50,11.40)$& $(27.50,118.00)$& $(302.00,1230.00)$\\
$  11$ & $  12$& $(1.45\cdot 10^{-3},6.69\cdot 10^{-3})$& $(1.73\cdot 10^{-2},7.46\cdot 10^{-2})$& $(2.07\cdot 10^{-1},8.33\cdot 10^{-1})$& $(2.48,9.29)$& $(29.60,104.00)$& $(354.00,1160.00)$\\
$  12$ & $  13$& $(7.21\cdot 10^{-4},4.21\cdot 10^{-3})$& $(9.21\cdot 10^{-3},5.08\cdot 10^{-2})$& $(1.17\cdot 10^{-1},6.13\cdot 10^{-1})$& $(1.50,7.40)$& $(19.20,89.40)$& $(245.00,1080.00)$\\
$  13$ & $  14$& $(4.55\cdot 10^{-4},2.61\cdot 10^{-3})$& $(6.37\cdot 10^{-3},3.40\cdot 10^{-2})$& $(8.91\cdot 10^{-2},4.44\cdot 10^{-1})$& $(1.24,5.78)$& $(17.40,75.40)$& $(244.00,984.00)$\\
$  14$ & $  15$& $(2.73\cdot 10^{-4},1.27\cdot 10^{-3})$& $(4.10\cdot 10^{-3},1.80\cdot 10^{-2})$& $(6.14\cdot 10^{-2},2.56\cdot 10^{-1})$& $(9.21\cdot 10^{-1},3.68)$& $(13.80,52.90)$& $(207.00,761.00)$\\
$  15$ & $  16$& $(1.87\cdot 10^{-4},9.01\cdot 10^{-4})$& $(2.95\cdot 10^{-3},1.36\cdot 10^{-2})$& $(4.65\cdot 10^{-2},2.06\cdot 10^{-1})$& $(7.34\cdot 10^{-1},3.10)$& $(11.20,46.80)$& $(169.00,706.00)$\\
$  16$ & $  17$& $(1.14\cdot 10^{-4},5.25\cdot 10^{-4})$& $(1.87\cdot 10^{-3},8.48\cdot 10^{-3})$& $(3.07\cdot 10^{-2},1.37\cdot 10^{-1})$& $(5.02\cdot 10^{-1},2.21)$& $(8.23,35.70)$& $(134.00,575.00)$\\
$  17$ & $  18$& $(6.19\cdot 10^{-5},2.99\cdot 10^{-4})$& $(1.06\cdot 10^{-3},5.19\cdot 10^{-3})$& $(1.84\cdot 10^{-2},9.04\cdot 10^{-2})$& $(3.17\cdot 10^{-1},1.58)$& $(5.48,27.50)$& $(94.50,478.00)$\\
$  18$ & $  19$& $(2.48\cdot 10^{-5},1.85\cdot 10^{-4})$& $(4.59\cdot 10^{-4},3.33\cdot 10^{-3})$& $(8.51\cdot 10^{-3},6.00\cdot 10^{-2})$& $(1.57\cdot 10^{-1},1.08)$& $(2.91,19.50)$& $(54.00,350.00)$\\
$  19$ & $  20$& $(2.35\cdot 10^{-5},1.05\cdot 10^{-4})$& $(4.48\cdot 10^{-4},2.01\cdot 10^{-3})$& $(8.51\cdot 10^{-3},3.85\cdot 10^{-2})$& $(1.61\cdot 10^{-1},7.40\cdot 10^{-1})$& $(3.07,14.30)$& $(58.40,275.00)$\\
$  20$ & $  21$& $(1.04\cdot 10^{-5},7.56\cdot 10^{-5})$& $(2.18\cdot 10^{-4},1.52\cdot 10^{-3})$& $(4.55\cdot 10^{-3},3.04\cdot 10^{-2})$& $(9.47\cdot 10^{-2},6.10\cdot 10^{-1})$& $(1.97,12.30)$& $(41.10,246.00)$\\
$  21$ & $  22$& $(8.53\cdot 10^{-6},4.27\cdot 10^{-5})$& $(1.85\cdot 10^{-4},8.97\cdot 10^{-4})$& $(4.01\cdot 10^{-3},1.89\cdot 10^{-2})$& $(8.69\cdot 10^{-2},3.96\cdot 10^{-1})$& $(1.88,8.32)$& $(40.80,175.00)$\\
$  22$ & $  23$& $(5.07\cdot 10^{-6},2.22\cdot 10^{-5})$& $(1.16\cdot 10^{-4},4.91\cdot 10^{-4})$& $(2.68\cdot 10^{-3},1.09\cdot 10^{-2})$& $(6.17\cdot 10^{-2},2.41\cdot 10^{-1})$& $(1.41,5.34)$& $(32.60,120.00)$\\
$  23$ & $  24$& $(2.69\cdot 10^{-6},1.64\cdot 10^{-5})$& $(6.41\cdot 10^{-5},3.78\cdot 10^{-4})$& $(1.52\cdot 10^{-3},8.76\cdot 10^{-3})$& $(3.63\cdot 10^{-2},2.03\cdot 10^{-1})$& $(8.65\cdot 10^{-1},4.70)$& $(20.50,109.00)$\\
$  24$ & $  25$& $(1.87\cdot 10^{-6},8.21\cdot 10^{-6})$& $(4.69\cdot 10^{-5},1.98\cdot 10^{-4})$& $(1.17\cdot 10^{-3},4.76\cdot 10^{-3})$& $(2.93\cdot 10^{-2},1.15\cdot 10^{-1})$& $(7.28\cdot 10^{-1},2.78)$& $(17.70,68.40)$\\
$  25$ & $  26$& $(8.87\cdot 10^{-7},5.40\cdot 10^{-6})$& $(2.28\cdot 10^{-5},1.35\cdot 10^{-4})$& $(5.87\cdot 10^{-4},3.41\cdot 10^{-3})$& $(1.50\cdot 10^{-2},8.66\cdot 10^{-2})$& $(3.88\cdot 10^{-1},2.21)$& $(9.98,56.00)$\\
$  26$ & $  27$& $(4.49\cdot 10^{-7},3.01\cdot 10^{-6})$& $(1.19\cdot 10^{-5},7.87\cdot 10^{-5})$& $(3.16\cdot 10^{-4},2.06\cdot 10^{-3})$& $(8.38\cdot 10^{-3},5.39\cdot 10^{-2})$& $(2.22\cdot 10^{-1},1.42)$& $(5.90,37.00)$\\
$  27$ & $  28$& $(4.76\cdot 10^{-7},1.93\cdot 10^{-6})$& $(1.32\cdot 10^{-5},5.26\cdot 10^{-5})$& $(3.67\cdot 10^{-4},1.44\cdot 10^{-3})$& $(1.02\cdot 10^{-2},3.92\cdot 10^{-2})$& $(2.79\cdot 10^{-1},1.07)$& $(7.55,29.20)$\\
$  28$ & $  29$& $(2.29\cdot 10^{-7},1.33\cdot 10^{-6})$& $(6.50\cdot 10^{-6},3.74\cdot 10^{-5})$& $(1.84\cdot 10^{-4},1.06\cdot 10^{-3})$& $(5.22\cdot 10^{-3},2.96\cdot 10^{-2})$& $(1.48\cdot 10^{-1},8.31\cdot 10^{-1})$& $(4.19,23.40)$\\
$  29$ & $  29.7$& $(1.84\cdot 10^{-7},7.98\cdot 10^{-7})$& $(5.36\cdot 10^{-6},2.32\cdot 10^{-5})$& $(1.56\cdot 10^{-4},6.73\cdot 10^{-4})$& $(4.54\cdot 10^{-3},1.96\cdot 10^{-2})$& $(1.32\cdot 10^{-1},5.67\cdot 10^{-1})$& $(3.86,16.50)$\\

\end{tabular}

\end{table}
\end{landscape}

\begin{landscape}
\begin{table}[h]
\caption{For all $e^{b_1} \leq x \leq e^{b_2}$ we have
                $\displaystyle\frac{m_k}{(\log(x)^k} < \log(\log(x)) + M - \lambda(x)  < \frac{M_k}{(\log(x)^k}$. Note in the given range the term is always negative.}
\label{MkmktableThm2}
{
\footnotesize

\begin{tabular}{cccccccc}
 $b_1$ & $b_2$ &  $(m_1,M_1)$& $(m_2,M_2)$& $(m_3,M_3)$& $(m_4,M_4)$& $(m_5,M_5)$ \\
 $   1$ & $   2$& $(-5.25\cdot 10^{-1},-1.54\cdot 10^{-1})$& $(-9.43\cdot 10^{-1},-2.48\cdot 10^{-1})$& $(-1.84,-4.00\cdot 10^{-1})$& $(-3.57,-6.43\cdot 10^{-1})$& $(-7.09,-7.64\cdot 10^{-1})$\\
$   2$ & $   3$& $(-4.44\cdot 10^{-1},-9.61\cdot 10^{-2})$& $(-9.89\cdot 10^{-1},-2.30\cdot 10^{-1})$& $(-2.92,-5.52\cdot 10^{-1})$& $(-8.58,-1.32)$& $(-25.30,-3.17)$\\
$   3$ & $   4$& $(-2.97\cdot 10^{-1},-7.14\cdot 10^{-2})$& $(-9.31\cdot 10^{-1},-2.63\cdot 10^{-1})$& $(-2.97,-8.91\cdot 10^{-1})$& $(-11.50,-3.00)$& $(-44.10,-10.10)$\\
$   4$ & $   5$& $(-1.66\cdot 10^{-1},-4.43\cdot 10^{-2})$& $(-7.81\cdot 10^{-1},-2.19\cdot 10^{-1})$& $(-3.69,-9.16\cdot 10^{-1})$& $(-17.50,-3.73)$& $(-82.50,-15.20)$\\
$   5$ & $   6$& $(-1.12\cdot 10^{-1},-2.24\cdot 10^{-2})$& $(-5.92\cdot 10^{-1},-1.31\cdot 10^{-1})$& $(-3.13,-7.19\cdot 10^{-1})$& $(-17.40,-3.88)$& $(-98.10,-21.00)$\\
$   6$ & $   7$& $(-6.87\cdot 10^{-2},-1.61\cdot 10^{-2})$& $(-4.23\cdot 10^{-1},-1.01\cdot 10^{-1})$& $(-2.60,-6.44\cdot 10^{-1})$& $(-16.00,-4.07)$& $(-101.00,-25.70)$\\
$   7$ & $   8$& $(-4.39\cdot 10^{-2},-5.21\cdot 10^{-3})$& $(-3.13\cdot 10^{-1},-3.78\cdot 10^{-2})$& $(-2.25,-2.75\cdot 10^{-1})$& $(-16.30,-1.99)$& $(-121.00,-14.50)$\\
$   8$ & $   9$& $(-2.23\cdot 10^{-2},-4.34\cdot 10^{-3})$& $(-1.82\cdot 10^{-1},-3.87\cdot 10^{-2})$& $(-1.52,-3.45\cdot 10^{-1})$& $(-12.80,-3.06)$& $(-109.00,-25.00)$\\
$   9$ & $  10$& $(-1.49\cdot 10^{-2},-1.58\cdot 10^{-3})$& $(-1.38\cdot 10^{-1},-1.56\cdot 10^{-2})$& $(-1.27,-1.54\cdot 10^{-1})$& $(-11.80,-1.52)$& $(-109.00,-15.00)$\\
$  10$ & $  11$& $(-1.09\cdot 10^{-2},-1.71\cdot 10^{-3})$& $(-1.10\cdot 10^{-1},-1.77\cdot 10^{-2})$& $(-1.11,-1.84\cdot 10^{-1})$& $(-11.20,-1.91)$& $(-113.00,-19.80)$\\
$  11$ & $  12$& $(-6.26\cdot 10^{-3},-1.27\cdot 10^{-3})$& $(-6.89\cdot 10^{-2},-1.45\cdot 10^{-2})$& $(-7.58\cdot 10^{-1},-1.63\cdot 10^{-1})$& $(-8.34,-1.82)$& $(-91.80,-20.30)$\\
$  12$ & $  13$& $(-3.14\cdot 10^{-3},-4.18\cdot 10^{-4})$& $(-3.88\cdot 10^{-2},-5.27\cdot 10^{-3})$& $(-4.79\cdot 10^{-1},-6.66\cdot 10^{-2})$& $(-5.91,-8.40\cdot 10^{-1})$& $(-73.60,-10.60)$\\
$  13$ & $  14$& $(-2.18\cdot 10^{-3},-2.76\cdot 10^{-4})$& $(-2.83\cdot 10^{-2},-3.84\cdot 10^{-3})$& $(-3.68\cdot 10^{-1},-5.34\cdot 10^{-2})$& $(-4.78,-7.42\cdot 10^{-1})$& $(-62.20,-10.30)$\\
$  14$ & $  15$& $(-1.25\cdot 10^{-3},-2.28\cdot 10^{-4})$& $(-1.77\cdot 10^{-2},-3.41\cdot 10^{-3})$& $(-2.52\cdot 10^{-1},-5.10\cdot 10^{-2})$& $(-3.59,-7.62\cdot 10^{-1})$& $(-51.00,-11.30)$\\
$  15$ & $  16$& $(-8.76\cdot 10^{-4},-1.56\cdot 10^{-4})$& $(-1.32\cdot 10^{-2},-2.40\cdot 10^{-3})$& $(-1.99\cdot 10^{-1},-3.67\cdot 10^{-2})$& $(-2.99,-5.62\cdot 10^{-1})$& $(-45.00,-8.60)$\\
$  16$ & $  17$& $(-4.47\cdot 10^{-4},-9.86\cdot 10^{-5})$& $(-7.31\cdot 10^{-3},-1.63\cdot 10^{-3})$& $(-1.20\cdot 10^{-1},-2.66\cdot 10^{-2})$& $(-1.96,-4.29\cdot 10^{-1})$& $(-32.10,-6.93)$\\
$  17$ & $  18$& $(-2.97\cdot 10^{-4},-2.99\cdot 10^{-5})$& $(-5.11\cdot 10^{-3},-5.22\cdot 10^{-4})$& $(-8.81\cdot 10^{-2},-9.09\cdot 10^{-3})$& $(-1.52,-1.58\cdot 10^{-1})$& $(-26.30,-2.76)$\\
$  18$ & $  19$& $(-1.66\cdot 10^{-4},-3.34\cdot 10^{-5})$& $(-3.03\cdot 10^{-3},-6.22\cdot 10^{-4})$& $(-5.61\cdot 10^{-2},-1.15\cdot 10^{-2})$& $(-1.04,-2.15\cdot 10^{-1})$& $(-19.30,-4.01)$\\
$  19$ & $  20$& $(-1.24\cdot 10^{-4},-1.76\cdot 10^{-5})$& $(-2.35\cdot 10^{-3},-3.49\cdot 10^{-4})$& $(-4.47\cdot 10^{-2},-6.89\cdot 10^{-3})$& $(-8.49\cdot 10^{-1},-1.36\cdot 10^{-1})$& $(-16.20,-2.68)$\\
$  20$ & $  21$& $(-6.46\cdot 10^{-5},-8.95\cdot 10^{-6})$& $(-1.30\cdot 10^{-3},-1.84\cdot 10^{-4})$& $(-2.59\cdot 10^{-2},-3.80\cdot 10^{-3})$& $(-5.19\cdot 10^{-1},-7.85\cdot 10^{-2})$& $(-10.40,-1.61)$\\
$  21$ & $  22$& $(-3.66\cdot 10^{-5},-5.74\cdot 10^{-6})$& $(-7.73\cdot 10^{-4},-1.26\cdot 10^{-4})$& $(-1.64\cdot 10^{-2},-2.76\cdot 10^{-3})$& $(-3.46\cdot 10^{-1},-6.06\cdot 10^{-2})$& $(-7.32,-1.33)$\\
$  22$ & $  23$& $(-2.50\cdot 10^{-5},-2.85\cdot 10^{-6})$& $(-5.51\cdot 10^{-4},-6.50\cdot 10^{-5})$& $(-1.22\cdot 10^{-2},-1.48\cdot 10^{-3})$& $(-2.68\cdot 10^{-1},-3.37\cdot 10^{-2})$& $(-5.90,-7.67\cdot 10^{-1})$\\
$  23$ & $  24$& $(-1.58\cdot 10^{-5},-1.98\cdot 10^{-6})$& $(-3.64\cdot 10^{-4},-4.59\cdot 10^{-5})$& $(-8.38\cdot 10^{-3},-1.06\cdot 10^{-3})$& $(-1.93\cdot 10^{-1},-2.46\cdot 10^{-2})$& $(-4.44,-5.71\cdot 10^{-1})$\\
$  24$ & $  25$& $(-7.97\cdot 10^{-6},-1.13\cdot 10^{-6})$& $(-1.95\cdot 10^{-4},-2.79\cdot 10^{-5})$& $(-4.73\cdot 10^{-3},-6.87\cdot 10^{-4})$& $(-1.16\cdot 10^{-1},-1.69\cdot 10^{-2})$& $(-2.81,-4.17\cdot 10^{-1})$\\
$  25$ & $  26$& $(-5.12\cdot 10^{-6},-6.26\cdot 10^{-7})$& $(-1.30\cdot 10^{-4},-1.59\cdot 10^{-5})$& $(-3.29\cdot 10^{-3},-4.05\cdot 10^{-4})$& $(-8.33\cdot 10^{-2},-1.03\cdot 10^{-2})$& $(-2.11,-2.61\cdot 10^{-1})$\\
$  26$ & $  27$& $(-3.31\cdot 10^{-6},-4.49\cdot 10^{-7})$& $(-8.60\cdot 10^{-5},-1.20\cdot 10^{-5})$& $(-2.24\cdot 10^{-3},-3.21\cdot 10^{-4})$& $(-5.83\cdot 10^{-2},-8.58\cdot 10^{-3})$& $(-1.52,-2.29\cdot 10^{-1})$\\
$  27$ & $  28$& $(-2.13\cdot 10^{-6},-3.01\cdot 10^{-7})$& $(-5.75\cdot 10^{-5},-8.33\cdot 10^{-6})$& $(-1.56\cdot 10^{-3},-2.30\cdot 10^{-4})$& $(-4.19\cdot 10^{-2},-6.38\cdot 10^{-3})$& $(-1.14,-1.76\cdot 10^{-1})$\\
$  28$ & $  29$& $(-1.15\cdot 10^{-6},-1.54\cdot 10^{-7})$& $(-3.22\cdot 10^{-5},-4.46\cdot 10^{-6})$& $(-9.11\cdot 10^{-4},-1.29\cdot 10^{-4})$& $(-2.59\cdot 10^{-2},-3.74\cdot 10^{-3})$& $(-7.32\cdot 10^{-1},-1.08\cdot 10^{-1})$\\
$  29$ & $  29.8$& $(-7.34\cdot 10^{-7},-1.32\cdot 10^{-7})$& $(-2.14\cdot 10^{-5},-3.90\cdot 10^{-6})$& $(-6.23\cdot 10^{-4},-1.14\cdot 10^{-4})$& $(-1.82\cdot 10^{-2},-3.38\cdot 10^{-3})$& $(-5.29\cdot 10^{-1},-9.95\cdot 10^{-2})$\\
\end{tabular}
}
\end{table}
\end{landscape}

\begin{landscape}
\begin{table}[h]
\caption{ For all $e^{b_1} \leq x \leq e^{b_2}$ we have
                $\displaystyle\frac{m_k}{(\log(x)^k} < \log(x) - M' - \Upsilon(x) < \frac{M_k}{(\log(x)^k}$. Note in the given range the term is always negative.}
\label{MkmktableThm3}
\footnotesize
\begin{tabular}{cccccccc}
 $b_1$ & $b_2$ &  $(m_1,M_1)$& $(m_2,M_2)$& $(m_3,M_3)$& $(m_4,M_4)$& $(m_5,M_5)$ \\
 $   1$ & $   2$& $(-1.37,-7.01\cdot 10^{-1})$& $(-2.65,-1.12)$& $(-5.17,-1.25)$& $(-10.40,-1.37)$& $(-20.70,-1.51)$\\
$   2$ & $   3$& $(-1.30,-5.93\cdot 10^{-1})$& $(-3.80,-1.42)$& $(-11.20,-3.41)$& $(-33.00,-8.17)$& $(-96.90,-19.60)$\\
$   3$ & $   4$& $(-1.21,-4.80\cdot 10^{-1})$& $(-3.77,-1.71)$& $(-14.50,-5.77)$& $(-55.60,-19.40)$& $(-214.00,-65.40)$\\
$   4$ & $   5$& $(-9.17\cdot 10^{-1},-3.40\cdot 10^{-1})$& $(-4.34,-1.66)$& $(-20.50,-6.79)$& $(-96.90,-27.70)$& $(-458.00,-113.00)$\\
$   5$ & $   6$& $(-6.96\cdot 10^{-1},-2.10\cdot 10^{-1})$& $(-3.69,-1.23)$& $(-20.20,-6.77)$& $(-114.00,-36.60)$& $(-643.00,-198.00)$\\
$   6$ & $   7$& $(-4.93\cdot 10^{-1},-1.65\cdot 10^{-1})$& $(-3.03,-1.04)$& $(-18.60,-6.63)$& $(-117.00,-41.90)$& $(-791.00,-265.00)$\\
$   7$ & $   8$& $(-3.55\cdot 10^{-1},-7.91\cdot 10^{-2})$& $(-2.56,-5.74\cdot 10^{-1})$& $(-18.50,-4.17)$& $(-137.00,-30.20)$& $(-1020.00,-219.00)$\\
$   8$ & $   9$& $(-2.08\cdot 10^{-1},-5.69\cdot 10^{-2})$& $(-1.72,-5.07\cdot 10^{-1})$& $(-14.40,-4.52)$& $(-122.00,-39.40)$& $(-1070.00,-319.00)$\\
$   9$ & $  10$& $(-1.53\cdot 10^{-1},-2.70\cdot 10^{-2})$& $(-1.42,-2.67\cdot 10^{-1})$& $(-13.10,-2.63)$& $(-121.00,-26.00)$& $(-1130.00,-257.00)$\\
$  10$ & $  11$& $(-1.20\cdot 10^{-1},-2.68\cdot 10^{-2})$& $(-1.22,-2.78\cdot 10^{-1})$& $(-12.30,-2.88)$& $(-124.00,-29.90)$& $(-1250.00,-310.00)$\\
$  11$ & $  12$& $(-7.55\cdot 10^{-2},-2.00\cdot 10^{-2})$& $(-8.31\cdot 10^{-1},-2.28\cdot 10^{-1})$& $(-9.14,-2.58)$& $(-101.00,-28.80)$& $(-1110.00,-321.00)$\\
$  12$ & $  13$& $(-4.22\cdot 10^{-2},-8.26\cdot 10^{-3})$& $(-5.21\cdot 10^{-1},-1.04\cdot 10^{-1})$& $(-6.43,-1.31)$& $(-79.40,-16.60)$& $(-1020.00,-209.00)$\\
$  13$ & $  14$& $(-3.08\cdot 10^{-2},-5.42\cdot 10^{-3})$& $(-4.00\cdot 10^{-1},-7.54\cdot 10^{-2})$& $(-5.20,-1.04)$& $(-67.60,-14.50)$& $(-879.00,-202.00)$\\
$  14$ & $  15$& $(-1.91\cdot 10^{-2},-4.37\cdot 10^{-3})$& $(-2.71\cdot 10^{-1},-6.53\cdot 10^{-2})$& $(-3.86,-9.75\cdot 10^{-1})$& $(-54.90,-14.50)$& $(-780.00,-217.00)$\\
$  15$ & $  16$& $(-1.41\cdot 10^{-2},-3.19\cdot 10^{-3})$& $(-2.12\cdot 10^{-1},-4.89\cdot 10^{-2})$& $(-3.19,-7.48\cdot 10^{-1})$& $(-48.00,-11.40)$& $(-723.00,-175.00)$\\
$  16$ & $  17$& $(-7.78\cdot 10^{-3},-2.05\cdot 10^{-3})$& $(-1.28\cdot 10^{-1},-3.41\cdot 10^{-2})$& $(-2.09,-5.65\cdot 10^{-1})$& $(-34.10,-9.13)$& $(-559.00,-147.00)$\\
$  17$ & $  18$& $(-5.42\cdot 10^{-3},-7.99\cdot 10^{-4})$& $(-9.34\cdot 10^{-2},-1.39\cdot 10^{-2})$& $(-1.61,-2.42\cdot 10^{-1})$& $(-27.80,-4.23)$& $(-479.00,-73.70)$\\
$  18$ & $  19$& $(-3.20\cdot 10^{-3},-7.75\cdot 10^{-4})$& $(-5.91\cdot 10^{-2},-1.44\cdot 10^{-2})$& $(-1.10,-2.68\cdot 10^{-1})$& $(-20.30,-5.00)$& $(-376.00,-93.20)$\\
$  19$ & $  20$& $(-2.48\cdot 10^{-3},-4.38\cdot 10^{-4})$& $(-4.71\cdot 10^{-2},-8.65\cdot 10^{-3})$& $(-8.95\cdot 10^{-1},-1.70\cdot 10^{-1})$& $(-17.10,-3.37)$& $(-324.00,-66.50)$\\
$  20$ & $  21$& $(-1.37\cdot 10^{-3},-2.41\cdot 10^{-4})$& $(-2.75\cdot 10^{-2},-4.98\cdot 10^{-3})$& $(-5.49\cdot 10^{-1},-1.02\cdot 10^{-1})$& $(-11.00,-2.12)$& $(-221.00,-43.70)$\\
$  21$ & $  22$& $(-8.18\cdot 10^{-4},-1.55\cdot 10^{-4})$& $(-1.73\cdot 10^{-2},-3.42\cdot 10^{-3})$& $(-3.66\cdot 10^{-1},-7.50\cdot 10^{-2})$& $(-7.73,-1.64)$& $(-164.00,-36.10)$\\
$  22$ & $  23$& $(-5.80\cdot 10^{-4},-8.47\cdot 10^{-5})$& $(-1.28\cdot 10^{-2},-1.92\cdot 10^{-3})$& $(-2.82\cdot 10^{-1},-4.39\cdot 10^{-2})$& $(-6.21,-9.99\cdot 10^{-1})$& $(-137.00,-22.70)$\\
$  23$ & $  24$& $(-3.81\cdot 10^{-4},-6.21\cdot 10^{-5})$& $(-8.77\cdot 10^{-3},-1.44\cdot 10^{-3})$& $(-2.02\cdot 10^{-1},-3.33\cdot 10^{-2})$& $(-4.65,-7.72\cdot 10^{-1})$& $(-108.00,-17.80)$\\
$  24$ & $  25$& $(-2.03\cdot 10^{-4},-3.57\cdot 10^{-5})$& $(-4.95\cdot 10^{-3},-8.81\cdot 10^{-4})$& $(-1.21\cdot 10^{-1},-2.17\cdot 10^{-2})$& $(-2.94,-5.35\cdot 10^{-1})$& $(-71.60,-13.10)$\\
$  25$ & $  26$& $(-1.36\cdot 10^{-4},-2.11\cdot 10^{-5})$& $(-3.43\cdot 10^{-3},-5.38\cdot 10^{-4})$& $(-8.68\cdot 10^{-2},-1.36\cdot 10^{-2})$& $(-2.20,-3.48\cdot 10^{-1})$& $(-55.80,-8.85)$\\
$  26$ & $  27$& $(-8.99\cdot 10^{-5},-1.47\cdot 10^{-5})$& $(-2.34\cdot 10^{-3},-3.95\cdot 10^{-4})$& $(-6.09\cdot 10^{-2},-1.05\cdot 10^{-2})$& $(-1.59,-2.82\cdot 10^{-1})$& $(-41.30,-7.53)$\\
$  27$ & $  28$& $(-5.99\cdot 10^{-5},-1.00\cdot 10^{-5})$& $(-1.62\cdot 10^{-3},-2.78\cdot 10^{-4})$& $(-4.37\cdot 10^{-2},-7.70\cdot 10^{-3})$& $(-1.18,-2.13\cdot 10^{-1})$& $(-31.90,-5.89)$\\
$  28$ & $  29$& $(-3.37\cdot 10^{-5},-5.37\cdot 10^{-6})$& $(-9.46\cdot 10^{-4},-1.55\cdot 10^{-4})$& $(-2.69\cdot 10^{-2},-4.50\cdot 10^{-3})$& $(-7.61\cdot 10^{-1},-1.30\cdot 10^{-1})$& $(-21.60,-3.78)$\\
$  29$ & $  30$& $(-2.23\cdot 10^{-5},-3.26\cdot 10^{-6})$& $(-6.47\cdot 10^{-4},-9.77\cdot 10^{-5})$& $(-1.89\cdot 10^{-2},-2.92\cdot 10^{-3})$& $(-5.49\cdot 10^{-1},-8.77\cdot 10^{-2})$& $(-16.00,-2.62)$\\
\end{tabular}
\end{table}
\end{landscape}

\begin{landscape}
\begin{table}[h]
\caption{ For all $e^{b_1} \leq x \leq e^{b_2}$ we have $\displaystyle\frac{m_k}{(\log(x)^k} < \log(x) - \gamma - \sum_{p^r<x} \frac{\log(p)}{p^r} < \frac{M_k}{(\log(x)^k}$.}
\label{MkmktableThm4}
\footnotesize
\begin{tabular}{cccccccc}
 $b_1$ & $b_2$ &  $(m_1,M_1)$& $(m_2,M_2)$& $(m_3,M_3)$& $(m_4,M_4)$& $(m_5,M_5)$ \\
 $   1$ & $   2$& $(-2.83\cdot 10^{-1},3.13\cdot 10^{-1})$& $(-4.56\cdot 10^{-1},6.09\cdot 10^{-1})$& $(-8.64\cdot 10^{-1},1.19)$& $(-1.69,2.31)$& $(-3.28,4.49)$\\
$   2$ & $   3$& $(-3.18\cdot 10^{-1},3.03\cdot 10^{-1})$& $(-9.34\cdot 10^{-1},8.25\cdot 10^{-1})$& $(-2.75,2.34)$& $(-8.10,6.62)$& $(-23.90,18.80)$\\
$   3$ & $   4$& $(-2.97\cdot 10^{-1},2.75\cdot 10^{-1})$& $(-1.02,1.02)$& $(-3.50,3.79)$& $(-12.10,14.10)$& $(-43.40,52.30)$\\
$   4$ & $   5$& $(-2.59\cdot 10^{-1},2.71\cdot 10^{-1})$& $(-1.23,1.24)$& $(-5.78,5.66)$& $(-27.40,25.90)$& $(-130.00,119.00)$\\
$   5$ & $   6$& $(-2.41\cdot 10^{-1},2.33\cdot 10^{-1})$& $(-1.28,1.26)$& $(-6.74,6.82)$& $(-35.70,36.90)$& $(-192.00,200.00)$\\
$   6$ & $   7$& $(-1.54\cdot 10^{-1},1.37\cdot 10^{-1})$& $(-9.46\cdot 10^{-1},8.61\cdot 10^{-1})$& $(-5.81,5.45)$& $(-36.50,35.60)$& $(-237.00,239.00)$\\
$   7$ & $   8$& $(-1.08\cdot 10^{-1},1.47\cdot 10^{-1})$& $(-7.95\cdot 10^{-1},1.07)$& $(-5.88,7.70)$& $(-43.50,55.90)$& $(-339.00,406.00)$\\
$   8$ & $   9$& $(-5.90\cdot 10^{-2},9.01\cdot 10^{-2})$& $(-5.17\cdot 10^{-1},7.30\cdot 10^{-1})$& $(-4.53,5.91)$& $(-39.70,47.90)$& $(-348.00,388.00)$\\
$   9$ & $  10$& $(-4.99\cdot 10^{-2},5.96\cdot 10^{-2})$& $(-4.62\cdot 10^{-1},5.59\cdot 10^{-1})$& $(-4.27,5.24)$& $(-39.40,49.10)$& $(-365.00,469.00)$\\
$  10$ & $  11$& $(-4.69\cdot 10^{-2},3.68\cdot 10^{-2})$& $(-4.73\cdot 10^{-1},3.82\cdot 10^{-1})$& $(-4.78,3.96)$& $(-48.20,41.10)$& $(-487.00,426.00)$\\
$  11$ & $  12$& $(-2.67\cdot 10^{-2},2.52\cdot 10^{-2})$& $(-2.94\cdot 10^{-1},2.81\cdot 10^{-1})$& $(-3.23,3.14)$& $(-35.60,35.00)$& $(-392.00,390.00)$\\
$  12$ & $  13$& $(-1.51\cdot 10^{-2},1.79\cdot 10^{-2})$& $(-1.89\cdot 10^{-1},2.16\cdot 10^{-1})$& $(-2.41,2.61)$& $(-30.80,33.00)$& $(-394.00,416.00)$\\
$  13$ & $  14$& $(-1.03\cdot 10^{-2},1.23\cdot 10^{-2})$& $(-1.34\cdot 10^{-1},1.60\cdot 10^{-1})$& $(-1.74,2.09)$& $(-23.20,27.20)$& $(-309.00,355.00)$\\
$  14$ & $  15$& $(-6.96\cdot 10^{-3},5.77\cdot 10^{-3})$& $(-9.90\cdot 10^{-2},8.30\cdot 10^{-2})$& $(-1.41,1.20)$& $(-20.10,17.20)$& $(-285.00,247.00)$\\
$  15$ & $  16$& $(-5.60\cdot 10^{-3},4.85\cdot 10^{-3})$& $(-8.42\cdot 10^{-2},7.32\cdot 10^{-2})$& $(-1.27,1.11)$& $(-19.10,16.70)$& $(-288.00,252.00)$\\
$  16$ & $  17$& $(-3.06\cdot 10^{-3},3.09\cdot 10^{-3})$& $(-5.01\cdot 10^{-2},4.98\cdot 10^{-2})$& $(-8.19\cdot 10^{-1},8.03\cdot 10^{-1})$& $(-13.50,13.00)$& $(-220.00,210.00)$\\
$  17$ & $  18$& $(-2.22\cdot 10^{-3},2.16\cdot 10^{-3})$& $(-3.83\cdot 10^{-2},3.76\cdot 10^{-2})$& $(-6.61\cdot 10^{-1},6.55\cdot 10^{-1})$& $(-11.40,11.40)$& $(-197.00,199.00)$\\
$  18$ & $  19$& $(-1.40\cdot 10^{-3},1.03\cdot 10^{-3})$& $(-2.58\cdot 10^{-2},1.85\cdot 10^{-2})$& $(-4.77\cdot 10^{-1},3.33\cdot 10^{-1})$& $(-8.83,6.11)$& $(-164.00,114.00)$\\
$  19$ & $  20$& $(-1.03\cdot 10^{-3},6.94\cdot 10^{-4})$& $(-1.96\cdot 10^{-2},1.34\cdot 10^{-2})$& $(-3.72\cdot 10^{-1},2.58\cdot 10^{-1})$& $(-7.07,4.97)$& $(-135.00,95.70)$\\
$  20$ & $  21$& $(-4.56\cdot 10^{-4},5.86\cdot 10^{-4})$& $(-9.18\cdot 10^{-3},1.18\cdot 10^{-2})$& $(-1.88\cdot 10^{-1},2.36\cdot 10^{-1})$& $(-3.90,4.73)$& $(-81.20,94.90)$\\
$  21$ & $  22$& $(-2.69\cdot 10^{-4},3.06\cdot 10^{-4})$& $(-5.68\cdot 10^{-3},6.43\cdot 10^{-3})$& $(-1.20\cdot 10^{-1},1.35\cdot 10^{-1})$& $(-2.55,2.84)$& $(-55.30,59.60)$\\
$  22$ & $  23$& $(-2.15\cdot 10^{-4},1.78\cdot 10^{-4})$& $(-4.73\cdot 10^{-3},4.04\cdot 10^{-3})$& $(-1.05\cdot 10^{-1},9.20\cdot 10^{-2})$& $(-2.30,2.10)$& $(-50.70,47.70)$\\
$  23$ & $  24$& $(-1.49\cdot 10^{-4},1.57\cdot 10^{-4})$& $(-3.42\cdot 10^{-3},3.63\cdot 10^{-3})$& $(-7.87\cdot 10^{-2},8.39\cdot 10^{-2})$& $(-1.82,1.95)$& $(-41.70,45.10)$\\
$  24$ & $  25$& $(-7.72\cdot 10^{-5},7.51\cdot 10^{-5})$& $(-1.89\cdot 10^{-3},1.86\cdot 10^{-3})$& $(-4.59\cdot 10^{-2},4.56\cdot 10^{-2})$& $(-1.12,1.13)$& $(-27.30,27.70)$\\
$  25$ & $  26$& $(-5.49\cdot 10^{-5},5.59\cdot 10^{-5})$& $(-1.40\cdot 10^{-3},1.42\cdot 10^{-3})$& $(-3.53\cdot 10^{-2},3.62\cdot 10^{-2})$& $(-8.94\cdot 10^{-1},9.19\cdot 10^{-1})$& $(-22.70,23.40)$\\
$  26$ & $  27$& $(-3.44\cdot 10^{-5},2.76\cdot 10^{-5})$& $(-9.12\cdot 10^{-4},7.36\cdot 10^{-4})$& $(-2.42\cdot 10^{-2},1.97\cdot 10^{-2})$& $(-6.42\cdot 10^{-1},5.26\cdot 10^{-1})$& $(-17.10,14.10)$\\
$  27$ & $  28$& $(-2.29\cdot 10^{-5},2.03\cdot 10^{-5})$& $(-6.17\cdot 10^{-4},5.54\cdot 10^{-4})$& $(-1.67\cdot 10^{-2},1.52\cdot 10^{-2})$& $(-4.50\cdot 10^{-1},4.13\cdot 10^{-1})$& $(-12.20,11.30)$\\
$  28$ & $  29$& $(-1.35\cdot 10^{-5},1.51\cdot 10^{-5})$& $(-3.82\cdot 10^{-4},4.23\cdot 10^{-4})$& $(-1.09\cdot 10^{-2},1.19\cdot 10^{-2})$& $(-3.07\cdot 10^{-1},3.35\cdot 10^{-1})$& $(-8.70,9.40)$\\
$  29$ & $  29.4$& $(-8.40\cdot 10^{-6},8.75\cdot 10^{-6})$& $(-2.45\cdot 10^{-4},2.54\cdot 10^{-4})$& $(-7.13\cdot 10^{-3},7.38\cdot 10^{-3})$& $(-2.08\cdot 10^{-1},2.15\cdot 10^{-1})$& $(-6.05,6.22)$\\
\end{tabular}
\end{table}
\end{landscape}

\end{document}